\documentclass{article}

\usepackage{arxiv}

\usepackage{amsmath,amssymb}
\usepackage[utf8]{inputenc} % allow utf-8 input
\usepackage[T1]{fontenc}    % use 8-bit T1 fonts
\usepackage{hyperref}       % hyperlinks
\usepackage{url}            % simple URL typesetting
\usepackage{booktabs}       % professional-quality tables
\usepackage{amsfonts}       % blackboard math symbols
\usepackage{nicefrac}       % compact symbols for 1/2, etc.
\usepackage{microtype}      % microtypography
\usepackage{cleveref}       % smart cross-referencing
\usepackage{lipsum}         % Can be removed after putting your text content
\usepackage{graphicx}
\usepackage[numbers]{natbib}
\usepackage{doi}

\title{Threshold estimation for jump-diffusions \\ under small noise asymptotics}
\author{ \href{https://orcid.org/0000-0003-0083-3187}{\includegraphics[scale=0.06]{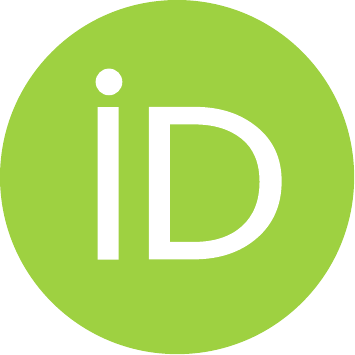}\hspace{1mm}Mitsuki~Kobayashi} \\
  Department of Pure and Applied Mathematics \\
  Waseda University, \\
	\texttt{mitsuki@fuji.waseda.jp} \\
	\And
	\href{https://orcid.org/0000-0003-3479-1149}{\includegraphics[scale=0.06]{orcid.pdf}\hspace{1mm}Yasutaka~Shimizu} \\
  Department of Pure and Applied Mathematics \\
  Waseda University, \\
	\texttt{shimizu@waseda.jp} \\
}
\date{}
\renewcommand{\shorttitle}{Threshold estimation for jump-diffusions under small noise asymptotics}

\hypersetup{
pdftitle={Threshold estimation for jump-diffusions under small noise asymptotics},
pdfauthor={Mitsuki~Kobayashi, Yasutaka~Shimizu},
pdfkeywords={small noise asymptotics, asymptotic distribution, discrete observations, threshold estimation, localization argument},
}

\usepackage{amsthm}
\theoremstyle{plain}

\newtheorem{theorem}{Theorem}[section]
\newtheorem{lemma}[theorem]{Lemma}

\theoremstyle{remark}

\newtheorem{remark}{Remark}[section]
\newtheorem{example}{Example}[section]

\newtheorem{assumption}{Assumption}[section]
\newtheorem*{notation}{Notation}

\usepackage{enumitem}

\DeclareMathOperator*{\argmax}{argmax}
\DeclareMathOperator{\dist}{dist}
\DeclareMathOperator{\Image}{Image}
\DeclareMathOperator{\Id}{Id}

\begin{document}

\maketitle

\begin{abstract}
  We consider parameter estimation of stochastic differential equations driven by a Wiener process and a compound Poisson process as small noises.
  The goal is to give a threshold-type quasi-likelihood estimator and show its consistency and asymptotic normality under new asymptotics.
  One of the novelties of the paper is that we give a new localization argument, which enables us to avoid truncation in the contrast function that has been used in earlier works
  and to deal with a wider class of jumps in threshold estimation than ever before.
\end{abstract}

\msc{Primary 62M20 \and Secondary 62F12, 60J74}

\keywords{small noise asymptotics \and asymptotic distribution \and discrete observations \and threshold estimation \and localization argument}

\section{Introduction}
  This paper is concerned with the following stochastic differential equation (SDE):
  \begin{equation}\label{eq: main SDE}
    dX^\varepsilon_t = a(X^\varepsilon_t,\mu_0) \, dt + \varepsilon b(X^\varepsilon_t, \sigma_0) dW_t + \varepsilon c(X^\varepsilon_{t-}, \alpha_0) dZ_t^{\lambda_\varepsilon},
    \quad X_0^{\varepsilon} = x_0 \in \mathbb{R},
  \end{equation}
  where $\varepsilon>0$, and $\Theta_i$ ($i=1,2,3$) are smooth bounded open convex sets in $\mathbb{R}^{d_i}$ with $d_i\in\mathbb{N}$ ($i=1,2,3$), respectively,
  and $\theta_0=(\mu_0,\sigma_0,\alpha_0)\in\Theta_0:=\Theta_1\times\Theta_2\times\Theta_3\subset\mathbb{R}^d$ with $d:=d_1+d_2+d_3$ with $\Theta:=\bar{\Theta}_0$,
  and each domain of $a,b,c$ is $\mathbb{R}\times\bar{\Theta}_i$ ($i=1,2,3$), respectively.
  Also, $Z^{\lambda_\varepsilon}=(Z_t^{\lambda_\varepsilon})_{t\geq0}$ is a compound Poisson process given by
  \begin{equation*}
    Z^{\lambda_\varepsilon}_t = \sum_{i=1}^{N_t^{\lambda_\varepsilon}} V_i,
    \quad
    Z^{\lambda_\varepsilon}_0 = 0,
  \end{equation*}
  where $N^{\lambda_\varepsilon}=(N^{\lambda_\varepsilon}_t)_{t\geq0}$ is a Poisson process with intensity $\lambda_\varepsilon>0$,
  and $V_i$'s are i.i.d. random variables with common probability density function $f_{\alpha_0}$,
  and are independent of $N^{\lambda_\varepsilon}$ (cf. Example 1.3.10 in Applebaum \cite{applebaum2009levy}).
  $W=(W_t)_{t\geq0}$ is a Wiener process.
  Here, we denote the probability space by $(\Omega,\mathcal{F}, P)$.

  Suppose that we have discrete data $X^\varepsilon_{t_0}, \dots, X^\varepsilon_{t_n}$ from \eqref{eq: main SDE}
  for $0=t_0<\dots<t_n=1$ with $t_i - t_{i-1} = 1/n$.
  We consider the problem of estimating the true $\theta_0\in\Theta_0$ under $n\to\infty$ and $\varepsilon\to0$.
  We also define $x_t$ as the solution of the corresponding deterministic differential equation
  \begin{equation*}
    \frac{dx_t}{dt} = a(x_t, \mu_0)
  \end{equation*}
  with the initial condition $x_0$.

  There are a lot of estimators for SDEs with small noise except for estimation of jump size density (see, e.g., \cite{gloter2009estimation,kobayashi2022least,long2013least,sorensen2003small}, and references are given in \cite{long2017least}).
  One may immediately notice that if the background SDEs are driven by Wiener process and a compound Poisson process, then the number of large jumps never go infinity in probability.
  This means that we would never establish a consistent estimator of jump size density.
  To avoid this difficulty, we propose to assume that the intensity $\lambda_\varepsilon$ of jumps goes to infinity as $\varepsilon\downarrow0$ ($\lambda_\varepsilon$ is not necessary to depend on $\varepsilon$ as in Remark \ref{rmk:net}),
  and this type of asymptotics would be the first in many works of literature.
  Intuitively, this assumption seems natural when we deal with data obtained in the long term by shortening the pitch of observations that
  is familiar in both cases of ergodic and small noise, so one may agree with our proposal.
  In practice, the intesity $\lambda_\varepsilon$ of jumps should be estimated, and it is possible by Lemma \ref{lem:4.8}:
  \begin{equation*}
    \lambda_\varepsilon \overset{p}{\sim} \sum_{k=1}^n 1_{D^{n,\varepsilon,\rho}_{k}}
    \quad \text{as}~\varepsilon\downarrow0.
  \end{equation*}

  Another attempt in this paper is to give a proof with localization argument (as in, e.g., Remark 1 in S{\o}rensen and Uchida \cite{sorensen2003small}) in the entire context.
  The argument is usually omitted, or instead, Propostion 1 in Gloter and S{\o}rensen \cite{gloter2009estimation} is just refered.
  However, in this paper, we shall give a proof under localization assumptions \ref{asmp:A9} to \ref{asmp:A12},  which are more complicated but includes more examples than ever, together with usual localization assumptions \ref{asmp:A6} and \ref{asmp:A7}.
  Indeed, we assume that the sign of $c$ does not change on a compact set, and $\psi$ defined in \ref{asmp:A12} may be not continuous on the whole space $\mathbb{R}\times\mathbb{R}\times\Theta$.
  Thus, we show our main results under the localization argument in the entire of our proof,
  which is one of the novelties of our paper.

  In the ergodic case, threshold estimation for SDEs with L\'evy noise is proposed in Shimizu and Yoshida \cite{shimizu2006estimation}, and has been considered so far by various researchers \cite{gloter2018jump,ogihara2011quasi,shimizu2017threshold}, and other references are given in Amorino and Gloter \cite{amorino2021joint}.
  The examples of jump size density in Shimizu and Yoshida \cite{shimizu2006estimation} they cannot take gamma distributions as an example of jump size densities,
  and then Ogihara and Yoshida \cite{ogihara2011quasi} establish assumptions for examples of jump size density to include gamma distributions.
  However, the contrast functions in both papers \cite{shimizu2006estimation,ogihara2011quasi} are established using truncation and are of complicated form.
  Our research begins with the creed that we would be able to remove the truncations from our contrast functions when we deal with SDE with small L\'evy noise, which is another novelty.
  Then, we establish our estimator which has the consistency and the asymptotic normality, and includes gamma distributions as an example of jump size densities.
  As we shall see in Remark \ref{rmk:rangeofrho}, it is interesting that the range of the rate $\rho$ of our filters $D^{n,\varepsilon,\rho}_k$ may be disjoint from the range of the rate $\rho$ of the filters $\{\Delta X_i^n>Dh^\rho\}$ in Ogihara and Yoshida \cite{ogihara2011quasi}.
  One can see some numerical examples, e.g., in Shimizu and Yoshida \cite{shimizu2006estimation}, and techniques how to choose the threshold (i.e., $v_{nk}/n^\rho$ in this paper) of the filter in our estimators are found in Shimizu \cite{shimizu2008practical,shimizu2010thredhold} and Masuda and Uehara \cite{masuda2021estimating}.

  In Section \ref{sec:assumptionnotation}, we set up some assumptions and notations.
  In Section \ref{sec:mainresults}, we state our main results, i.e., the consistency and the asymptotic normality of our estimator.
  In Section \ref{sec:proofs}, we give a proof of our main results.
  In Section \ref{sec:examples}, we give some examples of the jump size density for compound Poisson processes in our model.
  In the \hyperref[appn]{Appendix}, we state and prove some slightly different versions of well-known results.

\section{Assumptions and notations}\label{sec:assumptionnotation}
  \begin{notation}
    We set the following notations:
    \begin{enumerate}[label={\rm (N\arabic*)}]
      \item \label{ntt:Range}
            Let $I_{x_0}$ be the image of $t\mapsto x_t$ on $[0,1]$, and set
            \begin{equation*}
              I_{x_0}^{\delta} := \{ y\in\mathbb{R} | \dist(y,I_{x_0}) < \delta \}.
            \end{equation*}
      \item \label{ntt:psi}
            A function $\psi$ on $\mathbb{R}\times\mathbb{R}\times\Theta_3$ is of the form
            \begin{equation*}
              \psi(x,y,\alpha)
              := \left\{ \begin{aligned}
                  & \log \left| \frac{1}{c(x,\alpha)} f_\alpha \left( \frac{y}{c(x,\alpha)} \right) \right|
                        & & \text{if } c(x,\alpha) \neq 0 \text{ and } f_{\alpha}\left(\frac{y}{c(x,\alpha)}\right)>0, \\
                  & 0   & & \text{otherwise}.
                \end{aligned} \right.
            \end{equation*}
    \end{enumerate}
  \end{notation}

  We prepare the following assumptions:

  \begin{assumption}[Identifiability]\label{asmp:A1}
    If $\mu\neq\mu_0$, $\sigma\neq\sigma_0$ or $\alpha\neq\alpha_0$,
    then
    \begin{equation*}
      a(y,\mu)\not\equiv a(y,\mu_0),                ~
      b(y,\sigma)\not\equiv b(y,\sigma_0)           ~\text{or}~
      \psi(y,c(y,z,\alpha) \neq \psi(y,z,\alpha_0), ~\text{respectively},
    \end{equation*}
    for some $y\in I_{x_0}^{\delta}$ with some $\delta>0$, and for some $z\in\mathbb{R}$.
  \end{assumption}

  \begin{assumption}[Lipschitz continuity of coefficient functions]\label{asmp:A2}
    $a(\cdot,\mu_0)$, $b(\cdot,\sigma_0)$ and $c(\cdot,\alpha_0)$ are Lipschitz continuous on $\mathbb{R}$.
  \end{assumption}

  \begin{assumption}[$L^p$-summability of jump size density]\label{asmp:A3}
    For any $p\geq0$, $f_{\alpha_0}$ satisfies
    \begin{equation*}
      \int_{\mathbb{R}} |z|^p f_{\alpha_0}(z) \, dz < \infty,
    \end{equation*}
  \end{assumption}

  \begin{assumption}[Domain of jump size density]\label{asmp:A4}
    The family $\{f_{\alpha}\}_{\alpha\in\bar{\Theta}_3}$ satisfies either of the following condtions:
    \begin{enumerate}[label=(\roman*)]
      \item \label{asmp:A4(i)}
            $f_{\alpha}$, $\alpha\in\bar{\Theta}_3$ are positive and continuous on $\mathbb{R}$.
      \item \label{asmp:A4(ii)}
            $f_{\alpha}$, $\alpha\in\bar{\Theta}_3$ are positive and continuous on $\mathbb{R}_+(=(0,\infty))$, and are zero on $(-\infty,0]$.
    \end{enumerate}
  \end{assumption}

  \begin{assumption}[Boundedness of jump filters]\label{asmp:A5}
    Random variables $v_{n1},\dots,v_{nn}$ satisfy
    \begin{equation*}
      0 < v_1 \leq v_{nk} \leq v_2
    \end{equation*}
    for some constants $v_1$ and $v_2$.
  \end{assumption}

  \begin{assumption}\label{asmp:A6}
    The familiy $\{c(\cdot,\alpha)\}_{\sigma\in\bar{\Theta}_3}$ satisfies
    \begin{equation*}
      0 < c_1 \leq |c(x, \alpha)| \leq c_2 \quad \text{for } (x,\alpha) \in I_{x_0}\times\Theta_3
    \end{equation*}
    with some positve constants $c_1$ and $c_2$.
    In this paper, without loss of generality, we may assume
    \begin{equation*}
      c(x_t, \alpha) > c_1
      \quad \text{for } (x,\alpha) \in I_{x_0}\times\Theta_3.
    \end{equation*}
  \end{assumption}

  \begin{assumption}\label{asmp:A7}
    The family $\{b(\cdot,\sigma)\}_{\sigma\in\bar{\Theta}_2}$ satisfies
    \begin{equation*}
      \inf_{(x,\sigma)\in I_{x_0}\times\Theta_2} |b(x_t, \sigma)| > 0.
    \end{equation*}
  \end{assumption}

  \begin{assumption}\label{asmp:A8}
    The functions $a,b,c$ are differentiable with respect to $\mu,\sigma,\alpha$, respectively, on $I_{x_0}^\delta\times\Theta$ for some $\delta$,
    and the families $\{ \frac{\partial a}{\partial \mu_{j_1}} (\cdot,\mu) \}_{\mu\in\Theta_1}$,
    $ \{ \frac{\partial b}{\partial \sigma_{j_2}} (\cdot,\sigma) \}_{\sigma\in\Theta_2}$,
    $\{ \frac{\partial c}{\partial \alpha_{j_3}} (\cdot,\alpha) \}_{\sigma\in\Theta_2}$, $(j_\ell=1,\dots,d_\ell,~\ell=1,\dots,3)$ are equi-Lipschitz continuous on $I_{x_0}^\delta$.
  \end{assumption}

  \begin{assumption}\label{asmp:A9}
    \begin{enumerate}[label=(\roman*)]
      \item \label{asmp:A9(i)}
            Under Assumption \ref{asmp:A4} \ref{asmp:A4(i)},
            there exist constants $C>0$, $q\geq1$ and $\delta>0$ such that
            \begin{equation*}
              \sup_{(x,\alpha) \in I_{x_0}^\delta\times\Theta_3}
              \left| \frac{\partial \psi}{\partial y} (x,y,\alpha) \right|
              \leq C (1+|y|^q)
              \quad (y \in \mathbb{R}).
            \end{equation*}
      \item \label{asmp:A9(ii)}
            Under Assumption \ref{asmp:A4} \ref{asmp:A4(ii)},
            we assume the following four conditions:
            \begin{enumerate}[label=(ii.\alph*)]
              \item There exists $\delta>0$ and $L>0$ such that if $0<y_1\leq y \leq y_2$, then
                    \begin{equation*}
                      \left| \frac{\partial \psi}{\partial y} (x,y,\alpha) \right|
                      \leq \left| \frac{\partial \psi}{\partial y} (x,y_1,\alpha) \right|
                      + \left| \frac{\partial \psi}{\partial y} (x,y_2,\alpha) \right|
                      + L
                      \quad
                      \text{for all }
                      (x,\alpha)\in I_{x_0}^\delta\times\Theta_3.
                    \end{equation*}
              \item \label{asmp:A9(iib)}
                    There exist constants $q\geq0$ and $\delta>0$ such that
                    \begin{equation*}
                      \sup_{(x,\alpha) \in I_{x_0}^\delta\times\Theta_3}
                      \left| \frac{\partial \psi}{\partial y} (x,y,\alpha) \right|
                      \leq O \left( \frac{1}{|y|^q} \right)
                      \quad \text{as } |y| \to 0.
                    \end{equation*}
              \item There exists $\delta>0$ such that
                    for any $C_1>0$ and $C_2\geq0$
                    the map
                    \begin{equation*}
                      x \mapsto \int \sup_{\alpha\in\Theta_3} \left| \frac{\partial \psi}{\partial y}(x, C_1 y + C_2,\alpha) \right| f_{\alpha_0}(y) \, dy
                    \end{equation*}
                    takes values in $\mathbb{R}$ from $I_{x_0}^\delta$, and is continous on $I_{x_0}^\delta$.
            \end{enumerate}
    \end{enumerate}
  \end{assumption}

  \begin{assumption}\label{asmp:A10}
    There exists $\delta>0$ such that for $(x,y,\alpha)\in I_{x_0}^\delta\times\mathbb{R}\times\Theta$ with $\psi\neq0$, $\psi$ is differentiable with respect to $\alpha_i$ ($i=1,\dots,d_3$).
    For $\alpha\in\Theta_3$
    \begin{equation*}
      x \mapsto \int \psi(x,c(x,\alpha_0)z,\alpha) f_{\alpha_0}(z) \, dz, \quad
      x \mapsto \int |\psi(x,c(x,\alpha_0)z,\alpha)|^2 f_{\alpha_0}(z) \, dz
    \end{equation*}
    are continuous at every points in $I_{x_0}$,
    and there exist $\delta>0$ and $C>0$ such that
    \begin{align*}
      \int \left\{
      \sup_{(x,\alpha)\in I_{x_0}^\delta \times \Theta_3}
      \left| \psi (x,c(x,\alpha_0)z,\alpha) \right|
      + \sum_{j=1}^{d_3} \sup_{(x,\alpha)\in I_{x_0}^\delta \times \Theta_3}
      \left| \frac{\partial \psi}{\partial \alpha_j} (x,c(x,\alpha_0)z,\alpha) \right|
      \right\} f_{\alpha_0} (z) \, dz
      < \infty.
    \end{align*}
  \end{assumption}

  \begin{assumption}\label{asmp:A11}
    For $(x,y,\alpha)\in I_{x_0}^\delta\times\mathbb{R}\times\Theta$ with $\psi\neq0$, $\psi$ is differentiable with respect to $\alpha_i$ ($i=1,\dots,d_3$).
    For $i,j=1,\dots,d_3$
    \begin{equation*}
      x \mapsto \int \frac{\partial\psi}{\partial\alpha_i}
      \frac{\partial\psi}{\partial\alpha_j} \left( x, c(x,\alpha_0) z, \alpha_0 \right) f_{\alpha_0}(z) \, dz
    \end{equation*}
    is continuous at every point $x\in I_{x_0}$.
  \end{assumption}

  \begin{assumption}\label{asmp:A12}
    The functions $a,b,c$ are twice differentiable with respect to $\mu,\sigma,\alpha$, respectively, on $I_{x_0}^\delta\times\Theta$ for some $\delta$,
    and the families $\{ \frac{\partial^2 a}{\partial \mu_{i_1} \partial \mu_{j_1}} (\cdot,\mu) \}_{\mu\in\Theta_1}$,
    $\{ \frac{\partial^2 b}{\partial \sigma_{i_2} \partial \sigma_{j_2}} (\cdot,\sigma) \}_{\sigma\in\Theta_2}$,
    $\{ \frac{\partial^2 c}{\partial \alpha_{i_3} \partial \alpha_{j_3}}  (\cdot,\alpha) \}_{\sigma\in\Theta_2}$, $(i_\ell,j_\ell=1,\dots,d_\ell,~\ell=1,\dots,3)$ are equi-Lipschitz continuous on $I_{x_0}^\delta$.
    There exists $\delta>0$ such that for $(x,y,\alpha)\in I_{x_0}^\delta\times\mathbb{R}\times\Theta$ with $\psi\neq0$, $\psi$ is twice differentiable with respect to $\alpha_i$ ($i=1,\dots,d_3$).
    For $\alpha\in\Theta$, $i=1,\dots,d_3$
    \begin{equation*}
      x \mapsto \int \frac{\partial \psi}{\partial\alpha_i} (x,c(x,\alpha_0)z,\alpha) f_{\alpha_0}(z) \, dz, \quad
      x \mapsto \int \left| \frac{\partial \psi}{\partial \alpha_i} (x,c(x,\alpha_0)z,\alpha)\right|^2 f_{\alpha_0}(z) \, dz
    \end{equation*}
    are continuous at every points $x\in I_{x_0}$,
    and there exist $\delta>0$ such that
    \begin{equation*}
      \int
      \sum_{i,j=1}^{d_3} \sup_{(x,\theta)\in I_{x_0}^\delta \times \Theta}
      \left| \frac{\partial^2 \psi}{\partial \alpha_i \partial \alpha_j} (x,c(x,\alpha_0)z,\theta) \right|
      f_{\alpha_0} (z) \, dz
      < \infty.
    \end{equation*}
    We assume either of the following conditions \ref{asmp:A12(i)}
    or \ref{asmp:A12(ii)}:
    \begin{enumerate}[label=(\roman*)]
      \item \label{asmp:A12(i)}
            Under Assumption \ref{asmp:A4} \ref{asmp:A4(i)},
            there exist constants $C>0$, $q\geq1$ and $\delta>0$ such that
            \begin{equation*}
              \sup_{(x,\alpha) \in I_{x_0}^\delta\times\Theta_3}
              \left| \frac{\partial^2 \psi}{\partial y\partial\alpha_i} (x,y,\alpha) \right|
              \leq C (1+|y|^q)
              \quad (y \in \mathbb{R}).
            \end{equation*}
      \item \label{asmp:A12(ii)}
            Under Assumption \ref{asmp:A4} \ref{asmp:A4(ii)},
            we assume the following four conditions:
            \begin{enumerate}[label=(ii.\alph*)]
              \item There exists $\delta>0$ and $L>0$ such that if $0<y_1\leq y \leq y_2$, then
                    \begin{equation*}
                      \left| \frac{\partial^2 \psi}{\partial y\partial\alpha_i} (x,y,\alpha) \right|
                      \leq \left| \frac{\partial^2 \psi}{\partial y\partial\alpha_i} (x,y_1,\alpha) \right|
                      + \left| \frac{\partial^2 \psi}{\partial y\partial\alpha_i} (x,y_2,\alpha) \right|+L
                    \end{equation*}
                    for all $(x,\alpha)\in I_{x_0}^\delta\times\Theta_3$.
              \item \label{asmp:A12(iib)}
                    There exist constants $q\geq0$ and $\delta>0$ such that
                    \begin{equation*}
                      \sup_{(x,\alpha) \in I_{x_0}^\delta\times\Theta_3}
                      \left| \frac{\partial^2 \psi}{\partial y\partial\alpha_i} (x,y,\alpha) \right|
                      \leq O \left( \frac{1}{|y|^q} \right)
                      \quad \text{as } |y| \to 0.
                    \end{equation*}
              \item There exists $\delta>0$ such that
                    for any $C_1>0$ and $C_2\geq0$
                    the map
                    \begin{equation*}
                      x \mapsto \int \sup_{\alpha\in\Theta_3} \left| \frac{\partial^2 \psi}{\partial y\partial\alpha_i} (x, C_1 y + C_2,\alpha) \right| f_{\alpha_0}(y) \, dy
                    \end{equation*}
                    takes values in $\mathbb{R}$ from $I_{x_0}^\delta$, and is continous on $I_{x_0}^\delta$.
            \end{enumerate}
          \end{enumerate}
  \end{assumption}

  \begin{remark}
    Instead of Assumptions \ref{asmp:A6} and \ref{asmp:A7},
    the following stronger assumptions are often used:
    \begin{equation*}
      \inf_{(x,\sigma)\in\mathbb{R}\times\bar{\Theta}_2} |b(x, \sigma)| > 0, \quad
      \inf_{(x,\alpha)\in\mathbb{R} \times \bar{\Theta}_3} |c(x, \alpha)| > 0.
    \end{equation*}
    (see, e.g., Remark 1 in S{\o}rensen and Uchida \cite{sorensen2003small}).
    However, the `classical' localization argument mentioned in \cite{sorensen2003small} is hard to apply for our purpose.
    Thus, we employ our milder assumptions and show how it works well.
  \end{remark}

  \begin{remark}
    Under Assumption \ref{asmp:A10},
    \begin{equation*}
      \int \frac{\partial\psi}{\partial\alpha_i} \left( x, c(x,\alpha_0) z, \alpha_0 \right) f_{\alpha_0}(z) \, dz
      = \frac{\partial}{\partial\alpha_i} \left( \int \psi \left( x, c(x,\alpha_0) z, \alpha \right) f_{\alpha_0}(z) \, dz \right)_{\alpha=\alpha_0},
    \end{equation*}
    at every $x\in I_{x_0}^\delta$.
  \end{remark}

  \begin{remark}
    Assumption \ref{asmp:A12} is given by replacing $a,b,c,\psi$ with $\frac{\partial a}{\partial \mu_i}$,
    $\frac{\partial b}{\partial \sigma_i}$, $\frac{\partial c}{\partial \alpha_i}$, $\frac{\partial\psi}{\partial \alpha_i}$, respectively, in Assumptions \ref{asmp:A8} to \ref{asmp:A10}.
    We will use this assumption only for obtaining \eqref{eq:matrixCtoI0inp}.
  \end{remark}

  \begin{notation}
    We further introduce the following notations:
    \begin{enumerate}[label={\rm (N\arabic*)}]
      \setcounter{enumi}{3}
      \item \label{ntt: jump operator}
            Denote
            \begin{equation*}
              \Delta X^\varepsilon_t := X^\varepsilon_t - X^\varepsilon_{t-} \quad
              \text{for}~t>0,
            \end{equation*}
            where $\varepsilon>0$.
      \item \label{ntt:difference}
            Denote
            \begin{equation*}
              \Delta^n_k X^{\varepsilon} := X^\varepsilon_{t_k} - X^\varepsilon_{t_{k-1}}, \quad
              \Delta^n_k N^{\lambda_\varepsilon} := N^{\lambda_\varepsilon}_{t_k} - N^{\lambda_\varepsilon}_{t_{k-1}} \quad
              \text{for}~k=1,\dots,n,
            \end{equation*}
            where $n\in\mathbb{N}$, $\varepsilon>0$.
      \item \label{ntt:stoppingTime}
            Define random times $\tau_k$ and $\eta_k$ $(k=1,\dots,n)$ by
            \begin{align*}
              \tau_k&:=\inf\{t\in[t_{k-1},t_k]\,|\,\Delta X^\varepsilon_t\neq0~\text{or}~t=t_k\},\\
              \eta_k&:=\sup\{t\in[t_{k-1},t_k]\,|\,\Delta X^\varepsilon_t\neq0~\text{or}~t=t_{k-1}\},
            \end{align*}
            where $n\in\mathbb{N}$, $\varepsilon>0$.
      \item \label{ntt:J}
            Define events $J^{n,\varepsilon}_{k,i}$ $(k=1,\dots,n,~i=0,1,2)$ by
            \begin{equation*}
              J^{n,\varepsilon}_{k,0} := \left\{ \Delta^n_k N^{\lambda_\varepsilon}= 0 \right\}, ~
              J^{n,\varepsilon}_{k,1} := \left\{ \Delta^n_k N^{\lambda_\varepsilon}= 1 \right\} ~\text{and}~
              J^{n,\varepsilon}_{k,2} := \left\{ \Delta^n_k N^{\lambda_\varepsilon}\geq 2 \right\},
            \end{equation*}
            where $n\in\mathbb{N}$, $\varepsilon>0$.
      \item \label{ntt:CkDk}
            Under Assumptions \ref{asmp:A4} \ref{asmp:A4(i)} and \ref{asmp:A5},
            set events $C^{n,\varepsilon,\rho}_k$ and $D^{n,\varepsilon,\rho}_k$ $(k=1,\dots,n)$ by
            \begin{equation*}
              C^{n,\varepsilon,\rho}_k := \left\{ \left| \Delta^n_k X^{\varepsilon} \right| \leq \frac{v_{nk}}{n^\rho} \right\} \quad \text{and} \quad
              D^{n,\varepsilon,\rho}_k := \left\{ \left| \Delta^n_k X^{\varepsilon} \right| > \frac{v_{nk}}{n^\rho} \right\},
            \end{equation*}
            where $n\in\mathbb{N}$, $\varepsilon>0$, $\rho\in(0,1/2)$.

            Under Assumptions \ref{asmp:A4} \ref{asmp:A4(ii)} and \ref{asmp:A5},
            set events $C^{n,\varepsilon,\rho}_k$ and $D^{n,\varepsilon,\rho}_k$ $(k=1,\dots,n)$ by
            \begin{equation*}
              C^{n,\varepsilon,\rho}_k := \left\{ \Delta^n_k X^{\varepsilon} \leq \frac{v_{nk}}{n^\rho} \right\} \quad \text{and} \quad
              D^{n,\varepsilon,\rho}_k := \left\{ \Delta^n_k X^{\varepsilon} > \frac{v_{nk}}{n^\rho} \right\},
            \end{equation*}
            where $n\in\mathbb{N}$, $\varepsilon>0$, $\rho\in(0,1/2)$.

            Then, put
            \begin{equation*}
              C^{n,\varepsilon,\rho}_{k,i} := C^{n,\varepsilon,\rho}_k \cap J^{n,\varepsilon}_{k,i}, \quad
              D^{n,\varepsilon,\rho}_{k,i} := D^{n,\varepsilon,\rho}_k \cap J^{n,\varepsilon}_{k,i} \quad
              \text{for}~k=1,\dots,n,~i=0,1,2,
            \end{equation*}
            where $n\in\mathbb{N}$, $\varepsilon>0$, $\rho\in(0,1/2)$.
            Furthermore, for sufficiently small $\delta>0$, we may put
            \begin{equation*}
              \begin{aligned}
                \tilde{C}^{n,\varepsilon,\rho}_{k,i}
                &:= C^{n,\varepsilon,\rho}_{k,i} \cap \{ X^\varepsilon_t \in I_{x_0}^\delta \text{ for all } t\in[0,1] \}, \\
                \tilde{D}^{n,\varepsilon,\rho}_{k,i}
                &:= D^{n,\varepsilon,\rho}_{k,i} \cap \{ X^\varepsilon_t \in I_{x_0}^\delta \text{ for all } t\in[0,1] \}
              \end{aligned}
            \end{equation*}
            for $k=1,\dots,n$, $i=0,1,2$.
      \item \label{ntt:N9}
            Let $\kappa:=4v_2/c_1$, where the constatns $v_2$ and $c_1$ are given in Assumptions \ref{asmp:A5} and \ref{asmp:A6}, respectively.
    \end{enumerate}
  \end{notation}

  \begin{remark}\label{rmk:net}
    We treat $(n,\varepsilon)$ as a directed set with a suitable order according to a convergence.
    For examples,
    when we say that $n\to\infty$, $\varepsilon\to0$ and $\lambda_\varepsilon\to\infty$,
    we mean that the index set $(n,\varepsilon)$ is a directed set in $\mathbb{N}\times(0,\infty)$ with order $\prec_1$ defined by
    \begin{equation*}
      (n_1,\varepsilon_1) \prec_1 (n_2,\varepsilon)
      \quad \text{if} ~
      n_1 < n_2, ~
      \varepsilon_1 > \varepsilon_2 ~
      \text{and} ~
      \lambda_{\varepsilon_1} < \lambda_{\varepsilon_2},
    \end{equation*}
    and when we say that $n\to\infty$, $\varepsilon\to0$, $\lambda_\varepsilon\to\infty$ and $\lambda_\varepsilon \int_{|z|\leq \kappa/n^\rho} f_{\alpha_0}(z) \, dz\to0$ with some $\rho>0$ and $\kappa$ given in \ref{ntt:N9},
    we mean that the index set $(n,\varepsilon)$ is a directed set in $\mathbb{N}\times(0,\infty)$ with order $\prec_2$ defined by
    \begin{align*}
      (n_1,\varepsilon_1) \prec_2 (n_2,\varepsilon)
      \quad \text{if} ~
       n_1 < n_2, ~
      \varepsilon_1 > \varepsilon_2, ~
      \lambda_{\varepsilon_1} < \lambda_{\varepsilon_2}~
       \text{and} ~
      \lambda_{\varepsilon_1} \int_{|z|\leq\kappa/n_1^\rho} f_{\alpha_0}(z) \, dz
      > \lambda_{\varepsilon_2} \int_{|z|\leq\kappa/n_2^\rho} f_{\alpha_0}(z) \, dz.
    \end{align*}
    Needless to say, the identity map $\Id$ from $(\{(n,\varepsilon)\},\prec_2)$ to $(\{(n,\varepsilon)\},\prec_1)$ is monotone, and $\Id(\{(n,\varepsilon)\})$ is cofinal in $(\{(n,\varepsilon)\},\prec_1)$.
  \end{remark}

  \begin{remark}
    In this paper, we can assume $\lambda_\varepsilon$ does not depend on $\varepsilon$.
    In this case,
    we treat $\{(n,\varepsilon,\lambda)\}$ instead of $\{(n,\varepsilon)\}$ as a driected set,
    and we must write $X^{\varepsilon,\lambda}$, $Z^{\lambda}$, $\Psi_{n,\varepsilon,\lambda}$, etc., instead of $X^\varepsilon$, $Z^{\lambda_\varepsilon}$, $\Psi_{n,\varepsilon}$, etc., respectively.
    But, for simplicity, we assume $\lambda_\varepsilon$ depends on $\varepsilon$.
  \end{remark}

\section{Main results}\label{sec:mainresults}

  We define the following contrast function $\Psi_{n,\varepsilon}(\theta)$ after the quasi-log likelihood proposed in Shimizu and Yoshida \cite{shimizu2006estimation}:
  \begin{equation*}
    \Psi_{n,\varepsilon}(\theta)
    := \Psi_{n,\varepsilon}^{(1)}(\mu,\sigma) + \Psi_{n,\varepsilon}^{(2)}(\alpha)
    \quad \text{for } \theta=(\mu,\sigma,\alpha)\in\Theta,
  \end{equation*}
  where
  for $\rho\in(0,1/2)$, $\Psi_{n,\varepsilon}^{(1)}(\mu,\sigma)$ and $\Psi_{n,\varepsilon}^{(2)}(\alpha)$ are given by using the notations \ref{ntt:difference} and \ref{ntt:CkDk} as the following:
  \begin{align*}
    \Psi_{n,\varepsilon}^{(1)}(\mu,\sigma)
    &:= - \frac1n \sum_{k=1}^{n} \left \{ \frac{ \left | \Delta^n_k X^{\varepsilon} - a(X^\varepsilon_{t_{k-1}}, \mu) / n \right |^2 }{2\left | \varepsilon b(X^\varepsilon_{t_{k-1}},\sigma) \right |^2 / n}
    + \frac{1}{2} \log |b(X^\varepsilon_{t_{k-1}},\sigma)|^2 \right \}
    1_{C^{n,\varepsilon,\rho}_k}, \\
    \Psi_{n,\varepsilon}^{(2)}(\alpha)
    &:= \frac{1}{\lambda_\varepsilon} \sum_{k=1}^{n}
      \psi \left( X^\varepsilon_{t_{k-1}}, \frac{\Delta^n_k X^{\varepsilon}}\varepsilon, \alpha \right)
      1_{D^{n,\varepsilon,\rho}_k}
  \end{align*}
  with
  \begin{equation*}
    \psi(x,y,\alpha)
    := \left\{ \begin{aligned}
        & \log \left| \frac{1}{c(x,\alpha)} f_\alpha \left( \frac{y}{c(x,\alpha)} \right) \right|
              & & \text{if } c(x,\alpha) \neq 0 \text{ and } f_{\alpha}\left(\frac{y}{c(x,\alpha)}\right)>0, \\
        & 0   & & \text{otherwise}.
      \end{aligned} \right.
  \end{equation*}
  Then, the quasi-maximum likelihood estimator is given by
  \begin{equation*}
    \hat{\theta}_{n,\varepsilon} := \argmax_{\theta\in\Theta} \Psi_{n,\varepsilon}(\theta).
  \end{equation*}
  Note that the truncation adopted in Shimizu and Yoshida \cite{shimizu2006estimation} is not required here.

  The goal is to show the asymptotic normality of $\hat{\theta}_{n,\varepsilon}$ when $n\to\infty$ and $\varepsilon\to0$ at the sametime.
  In the sequel, we will also assume that $\lambda_\varepsilon\to\infty$ as $\varepsilon\downarrow0$ for consistency of $\hat{\theta}_{n,\varepsilon}$. Our interest is in a situation where the number of jumps is large and the L\'evy noise is small. In practice, $\lambda_\varepsilon$, the intensity of jumps, should be estimated, and it is possible by Lemma \ref{lem:4.8}:
  \begin{equation*}
    \lambda_\varepsilon \overset{p}{\sim} \sum_{k=1}^n 1_{D^{n,\varepsilon,\rho}_{k}}
    \quad \text{as}~\varepsilon\downarrow0.
  \end{equation*}

  \begin{theorem}\label{thm:3.1}
    Under Assumptions \ref{asmp:A1} to \ref{asmp:A10},
    take $\rho$ as either of the followings:
   \begin{enumerate}[label=(\roman*)]
     \item Under Assumption \ref{asmp:A4} \ref{asmp:A4(i)},
                 take $\rho\in(0,1/2)$.
     \item Under Assumtpion \ref{asmp:A4} \ref{asmp:A4(ii)},
                 take $\rho\in(0,\min\{1/2,1/4q\})$, where $q$ is the constant in Assumption \ref{asmp:A9} \ref{asmp:A9(iib)}.
   \end{enumerate}
   Then,
    \begin{equation*}
      \hat\theta_{n,\varepsilon} \overset{p}\longrightarrow \theta_0
    \end{equation*}
    as $n\to\infty$, $\varepsilon\to0$, $\lambda_\varepsilon\to\infty$, $\lambda_\varepsilon^2/n \to 0$, $\varepsilon\lambda_\varepsilon\to0$ and $\lambda_\varepsilon \int_{|z|\leq\kappa/n^\rho} f_{\alpha_0}(z) \, dz\to0$ with $\lim (\varepsilon^2 n)^{-1}<\infty$.
    Here, the constant $\kappa$ is given in \ref{ntt:N9}.
  \end{theorem}

  \begin{theorem}\label{thm:3.2}
    Under Assumptions \ref{asmp:A1} to \ref{asmp:A12},
    take $\rho$ as either of the followings:
   \begin{enumerate}
     \item[(i)]  Under Assumption \ref{asmp:A4} \ref{asmp:A4(i)},
                 take $\rho\in(0,1/2)$.
     \item[(ii)] Under Assumption \ref{asmp:A4} \ref{asmp:A4(ii)},
                 take $\rho\in(0,\min\{1/2,1/4q\})$, where $q$ is the constant
                 in Assumptions \ref{asmp:A9} \ref{asmp:A9(iib)}
                 and \ref{asmp:A12} \ref{asmp:A12(iib)}.
   \end{enumerate}
    If $\theta_0\in\Theta$ and $I_{\theta_0}$ is positive definite, then
    \begin{equation*}
      \left( \begin{aligned}
        \varepsilon^{-1} (\hat{\mu}_{n,\varepsilon}-\mu_0) \\
        \sqrt{n} (\hat{\sigma}_{n,\varepsilon}-\sigma_0) \\
        \sqrt{\lambda_\varepsilon} (\hat{\alpha}_{n,\varepsilon}-\alpha_0)
      \end{aligned} \right)
      \overset{d}\longrightarrow
      \mathcal{N} \left( 0, I_{\theta_0}^{-1} \right)
    \end{equation*}
    as $n\to\infty$, $\varepsilon\to0$, $\lambda_\varepsilon\to\infty$, $\lambda_\varepsilon^2/n \to 0$, $\varepsilon\lambda_\varepsilon\to0$ and $\lambda_\varepsilon \int_{|z|\leq\kappa/n^\rho} f_{\alpha_0}(z) \, dz\to0$ with $\lim (\varepsilon^2 n)^{-1}<\infty$,
    where
    \begin{equation*}
      I_{\theta_0} := \left(
      \begin{matrix}
        I_1 & 0 & 0 \\
        0 & I_2 & 0 \\
        0 & 0 & I_3
      \end{matrix} \right)
    \end{equation*}
    and
    \begin{align*}
      I^{i_1i_2}_1 &:= \int_0^1 \frac{\frac{\partial a}{\partial \mu_{i_1}}(x_t,\mu_0)\frac{\partial a}{\partial \mu_{i_2}}(x_t,\mu_0)}{|b(x_t,\mu_0)|^2} dt, \quad
      I^{i_1i_2}_2 := 2 \int_0^1 \frac{\frac{\partial b}{\partial \sigma_{i_1}}(x_t,\sigma_0)\frac{\partial b}{\partial \sigma_{i_2}}(x_t,\sigma_0)}{|b(x_t,\sigma_0)|^2} dt, \\
      I^{i_1i_2}_3 &:= \int_0^1 \int \frac{\partial \psi}{\partial \alpha_{i_1}}(x_t,c(x_t,\alpha_0) z,\alpha_0)\frac{\partial \psi}{\partial \alpha_{i_2}}(x_t,c(x_t,\alpha_0) z,\alpha_0) f_{\alpha_0}(z) \, dz \, dt.
    \end{align*}
  \end{theorem}

  \begin{remark}\label{rmk:rangeofrho}
    If the jump size density $\{f_\alpha\}_{\alpha\in\Theta_3}$ is given as the family of probability density functions of normal distribution (see Example \ref{eg:normaldistr}),
    then the range of $\rho$ in Theorems \ref{thm:3.1} and \ref{thm:3.2} is same as in Shimizu and Yoshida \cite{shimizu2006estimation} and Ogihara and Yoshida \cite{ogihara2011quasi}.
    However, if the jump size density $\{f_\alpha\}_{\alpha\in\Theta_3}$ is given as the family of probability density functions of gamma distribution (see Example \ref{eg:gammadistr}), then the range of $\rho$ is $(0,1/4)$ which is different from the range $(3/8+b,1/2)$ of $\rho$ in Ogihara and Yoshida \cite{ogihara2011quasi}, where $b$ is the constant defined in the equation (1) in Ogihara and Yoshida \cite{ogihara2011quasi}. Needless to say that this paper deal with a small noise model and they deal with an ergodic model; we note that our estimator is established without truncation, and their estimator is established by using truncation.
  \end{remark}

\section{Proofs}\label{sec:proofs}

  \subsection{Inequalities}

    \begin{lemma}\label{lem:4.1}
      Under Assumptions \ref{asmp:A2} and \ref{asmp:A3}, let
      $0<\varepsilon\leq1$, $\lambda_\varepsilon\geq1$, $\varepsilon\lambda_\varepsilon\leq1$ and $0\leq s<t\leq1$.
      Then, for $p\geq2$,
      \begin{align*}
        E \left[
            \sup_{u\in[s,t]} \left| X^\varepsilon_u - X^\varepsilon_s \right|^p \,
          \middle| \, \mathcal{F}_s \right]
        &\leq C \left\{ (t-s)^p + \varepsilon^p \left( (t-s)^{p/2}  + \lambda_\varepsilon (t-s) \right. \right. \\
        &\quad \left. \left. + \lambda_\varepsilon^{p/2}(t-s)^{p/2} + \lambda_\varepsilon^{p}(t-s)^{p} \right) \right\}
          \left(1 + |X^\varepsilon_s|^p \right),
      \end{align*}
      where $C$ depends only on $p,a,b,c$ and $f_{\alpha_0}$.
      In particular, when $\lambda_\varepsilon/n\leq1$ and $\lambda_\varepsilon\geq1$,
      it holds for $p\geq2$ and $k=1,\dots,n$ that
      \begin{align*}
        E \left[
            \sup_{t\in[t_{k-1},t_k]} \frac{\left| X^\varepsilon_t - X^\varepsilon_{t_{k-1}} \right|^p}{\varepsilon^p} \,
          \middle| \, \mathcal{F}_{t_{k-1}} \right]
        &\leq C \left\{ \frac{1}{\varepsilon^p n^p} + \frac{1}{n^{p/2}} + \frac{\lambda_\varepsilon}{n} \right\}
          \left(1 + |X^\varepsilon_s|^p \right), \\
        E \left[
            \sup_{t\in[0,1]} \left| X^\varepsilon_t - x_0 \right|^p \,
          \middle| \, \mathcal{F}_{t_0} \right]
        &\leq C \left\{ 1 + \varepsilon^p \lambda_\varepsilon^p \right\}
          \left(1 + |x_0|^p \right),
      \end{align*}
      where $C$ depends only on $p,a,b,c$ and $f_{\alpha_0}$.
    \end{lemma}

    \begin{proof}
      For any $p\geq2$, we have
      \begin{equation}\label{eq: tri ineq to E[|X_u-X_s|^p|F_s]}
        \begin{aligned}
          \left( E \left[
              \sup_{u\in[s,t]} \left| X^\varepsilon_u - X^\varepsilon_{s} \right|^p \,
            \middle| \, \mathcal{F}_{s} \right] \right)^{1/p}
          &\leq \left( E \left[ \left| \int_s^t
                  \left| a(X^\varepsilon_u,\mu_0) - a(X^\varepsilon_s,\mu_0) \right|
                du \right|^p \, \middle| \, \mathcal{F}_s \right] \right)^{1/p} \\
          &\quad
            + \varepsilon \left( E \left[ \sup_{u\in[s,t]} \left| \int_s^u
                \left( b(X^\varepsilon_{v}, \sigma_0) - b(X^\varepsilon_s, \sigma_0) \right)
              \, dW_v \right|^p \, \middle| \, \mathcal{F}_s \right] \right)^{1/p} \\
          &\quad
            + \varepsilon \left( E \left[ \sup_{u\in[s,t]} \left| \int_s^u
                \left( c(X^\varepsilon_v, \alpha_0) - c(X^\varepsilon_s, \alpha_0) \right)
              \, dZ_v^{\lambda_\varepsilon} \right|^p \, \middle| \, \mathcal{F}_s \right] \right)^{1/p} \\
          &\quad
            + (t-s) \left| a(X^\varepsilon_s,\mu_0) \right|
            + C \varepsilon\sqrt{t-s} \left| b(X^\varepsilon_s,\sigma_0) \right| \\
          &\quad
            + \varepsilon \left| c(X^\varepsilon_s, \alpha_0) \right| \left( E \left[ \sup_{u\in[s,t]} \left| \int_s^u
            \, dZ_v^{\lambda_\varepsilon} \right|^p \right] \right)^{1/p},
        \end{aligned}
      \end{equation}
      where $C$ depends only on $p$.
      Then, it follows from the Lipschitz continuity of $a(\cdot,\mu_0)$ that
      \begin{equation}\label{eq: E[|a(X_u)-a(X_s)|^p|F_s]}
        \begin{aligned}
          E \left[ \left( \int_s^t
            \left| a(X^\varepsilon_u,\mu_0) - a(X^\varepsilon_s,\mu_0) \right|
            du \right)^p \, \middle| \, \mathcal{F}_s \right]
          &
            \leq C E \left[ \left(
              \int_s^t |X^\varepsilon_u - X^\varepsilon_s| \, du
            \right)^p \, \middle| \, \mathcal{F}_s \right]\\
          &
            \leq C (t-s)^{p-1} \int_s^t E \left[
              \left| X^\varepsilon_u - X^\varepsilon_s \right|^p \,
            \middle| \, \mathcal{F}_s \right] du,
        \end{aligned}
      \end{equation}
      where $C$ depends only on $a$,
      and it follows from the Lipschitz continuity of $b(\cdot,\sigma_0)$ and Burkholder's inequality (see, e.g., Theorem 4.4.21 in Applebaum \cite{applebaum2009levy}) that
      \begin{equation}\label{eq: E[|b(X_u)-b(X_s)|^p|F_s]}
        \begin{aligned}
          E \left[ \sup_{u\in[s,t]} \left| \int_s^u
            \left( b(X^\varepsilon_v, \sigma_0) - b(X^\varepsilon_s, \sigma_0) \right)
          dW_v \right|^p \, \middle| \, \mathcal{F}_s \right]
          &\leq C E \left[ \left( \int_s^t
            \left| X^\varepsilon_u - X^\varepsilon_s \right|^2
          du \right)^{p/2} \, \middle| \, \mathcal{F}_s \right] \\
          &\leq C (t-s)^{p/2-1} \int_s^t E \left[
            \left| X^\varepsilon_u - X^\varepsilon_s \right|^p \,
          \middle| \, \mathcal{F}_s \right] du,
        \end{aligned}
      \end{equation}
      where $C$ depends only on $p$ and $b$,
      and from the Lipschitz continuity of $c(\cdot,\alpha_0)$,
      it is analogous to the proof of Theorem 4.4.23 in Applebaum \cite{applebaum2009levy} that
      \begin{align*}
        &E \left[ \sup_{u\in[s,t]} \left| \int_s^u
          \left( c(X^\varepsilon_v, \alpha_0) - c(X^\varepsilon_s, \alpha_0) \right)
        \, dZ_v^{\lambda_\varepsilon} \right|^p \, \middle| \, \mathcal{F}_s \right] \\
        &\quad
        \begin{aligned}
          \leq & ~ C \left\{ E \left [ \left ( \int_s^t \int_{\mathbb{R}}
            \left| X^\varepsilon_u - X^\varepsilon_s \right|^2 \left | z \right |^2 \lambda_\varepsilon \, f_{\alpha_0}(z)
          \, dz \, du \right )^{p/2} \, \middle| \, \mathcal{F}_s \right ] \right. \\
          &
          + E \left [ \int_s^t \int_{\mathbb{R}}
            \left| X^\varepsilon_u - X^\varepsilon_s \right|^p \left | z \right |^p \lambda_\varepsilon \, f_{\alpha_0}(z)
          \, dz \, du \, \middle| \, \mathcal{F}_s \right]
          + \left. E \left [ \left ( \int_s^t \int_{\mathbb{R}}
            \left| X^\varepsilon_u - X^\varepsilon_s \right| \left | z \right | \lambda_\varepsilon \, f_{\alpha_0}(z)
          \, dz \, du \right )^{p} \, \middle| \, \mathcal{F}_s \right ] \right\},
        \end{aligned}
      \end{align*}
      where $C$ depends only on $p$ and $c$.
      Here, we have
      \begin{align*}
        E \left [ \left ( \int_s^t \int_{\mathbb{R}}
          \left| X^\varepsilon_u - X^\varepsilon_s \right|^2 \left | z \right |^2 \lambda_\varepsilon \, f_{\alpha_0}(z)
        \, dz \, du \right )^{p/2} \, \middle| \, \mathcal{F}_s \right ]
        &\leq C \lambda_\varepsilon^{p/2} \left( \int_s^t \left( E \left[
          \left| X^\varepsilon_u - X^\varepsilon_s \right|^p \,
        \middle| \, \mathcal{F}_s \right] \right)^{2/p} du \right)^{p/2} \\
        &\leq C \lambda_\varepsilon^{p/2} (t-s)^{p/2-1} \int_s^t E \left[
          \left| X^\varepsilon_u - X^\varepsilon_s \right|^p \,
        \middle| \, \mathcal{F}_s \right] du,
      \end{align*}
      where $C$ depends only on $p$ and $f_{\alpha_0}$, and
      \begin{align*}
        E \left [ \left ( \int_s^t \int_{\mathbb{R}}
          \left| X^\varepsilon_u - X^\varepsilon_s \right| \left | z \right | \lambda_\varepsilon \, f_{\alpha_0}(z)
        \, dz \, du \right )^{p} \, \middle| \, \mathcal{F}_s \right ]
        &\leq C \lambda_\varepsilon^{p} \left( \int_s^t \left( E \left[
          \left| X^\varepsilon_u - X^\varepsilon_s \right|^p \,
        \middle| \, \mathcal{F}_s \right] \right)^{1/p} du \right)^p \\
        &\leq C \lambda_\varepsilon^p (t-s)^{p-1} \int_s^t E \left[
          \left| X^\varepsilon_u - X^\varepsilon_s \right|^p \,
        \middle| \, \mathcal{F}_s \right] du,
      \end{align*}
      where $C$ depends only on $p$ and $f_{\alpha_0}$.
      Thus,
      \begin{equation}\label{eq: E[|c(X_u)-c(X_s)|^p|F_s]}
        \begin{aligned}
          & E \left[ \sup_{u\in[s,t]} \left| \int_s^u
            \left( c(X^\varepsilon_v, \alpha_0) - c(X^\varepsilon_s, \alpha_0) \right)
          \, dZ_v^{\lambda_\varepsilon} \right|^p \, \middle| \, \mathcal{F}_s \right] \\
          & \qquad
          \leq C \left( \lambda_\varepsilon^{p/2}(t-s)^{p/2-1} + \lambda_\varepsilon + \lambda_\varepsilon^p (t-s)^{p-1} \right)
          \int_s^t E \left[
            \left| X^\varepsilon_u - X^\varepsilon_s \right|^p \,
          \middle| \, \mathcal{F}_s \right] du,
        \end{aligned}
      \end{equation}
      where $C$ depends only on $p$, $c$ and $f_{\alpha_0}$.
      By using the Burkholder--Davis--Gundy inequality,
      \begin{equation}\label{eq: E[|Z_u-Z_s|^p|F_s]}
        E \left[ \sup_{u\in[s,t]} \left| \int_s^u \, dZ_v^{\lambda_\varepsilon} \right|^p \right]
        \leq C \left( \lambda_\varepsilon^{p/2} (t-s)^{p/2} + \lambda_\varepsilon(t-s) + \lambda_\varepsilon^p (t-s)^p \right),
      \end{equation}
      where $C$ depends only on $p$ and $f_{\alpha_0}$.
      From the equations \eqref{eq: tri ineq to E[|X_u-X_s|^p|F_s]} to \eqref{eq: E[|Z_u-Z_s|^p|F_s]},
      \begin{align*}
        &E \left[
            \sup_{u\in[s,t]} \left| X^\varepsilon_u - X^\varepsilon_s \right|^p \,
          \middle| \, \mathcal{F}_s \right] \\
          & \quad
            \leq C \left\{ \left((t-s)^{p-1} + \varepsilon^p(t-s)^{\frac{p-2}{2}}  + \varepsilon^p \left( \lambda_\varepsilon + \lambda_\varepsilon^{\frac{p}{2}}(t-s)^{\frac{p-2}{2}} + \lambda_\varepsilon^{p}(t-s)^{p-1} \right) \right)
            \int_s^t E \left[
              \left| X^\varepsilon_u - X^\varepsilon_s \right|^p \,
            \middle| \, \mathcal{F}_s \right] du \right. \\
          & \quad \quad
            + (t-s)^p \left| a(X^\varepsilon_s,\mu_0) \right|^p
            + \varepsilon^p(t-s)^{p/2} \left| b(X^\varepsilon_s,\sigma_0) \right|^p \\
          &\quad \quad \left.
            + \varepsilon^p \left( \lambda_\varepsilon (t-s) + \lambda_\varepsilon^{p/2} (t-s)^{p/2} + \lambda_\varepsilon^p (t-s)^p \right) \left| c(X^\varepsilon_s, \alpha_0) \right|^p \right\},
      \end{align*}
      where $C$ depends only on $p,a,b,c$ and $f_{\alpha_0}$.
      By Gronwall's inequality,
      \begin{align*}
        &E \left[
            \sup_{u\in[s,t]} \left| X^\varepsilon_u - X^\varepsilon_s \right|^p \,
          \middle| \, \mathcal{F}_s \right] \\
        &\leq C  \left\{
          (t-s)^p \left| a(X^\varepsilon_s,\mu_0) \right|^p
          + \varepsilon^p(t-s)^{p/2} \left| b(X^\varepsilon_s,\sigma_0) \right|^p \right. \\
        &\quad \left.
          + \varepsilon^p \left( \lambda_\varepsilon(t-s) + \lambda_\varepsilon^{p/2}(t-s)^{p/2} + \lambda_\varepsilon^p (t-s)^p \right) \left| c(X^\varepsilon_s, \alpha_0) \right|^p \right\} \\
        &\quad \times \exp \left( C \left\{ (t-s)^p + \varepsilon^p(t-s)^{p/2}  + \varepsilon^p \lambda_\varepsilon (t-s) + \varepsilon^p \lambda_\varepsilon^{p/2}(t-s)^{p/2} + \varepsilon^p \lambda_\varepsilon^{p}(t-s)^{p} \right\} \right).
      \end{align*}
      This implies the conclusion.
      % \begin{align}
      %   &E \left[
      %       \sup_{u\in[s,t]} \left| X^\varepsilon_u - X^\varepsilon_s \right|^p \,
      %     \middle| \, \mathcal{F}_s \right] \\
      %   &\quad
      %   \begin{aligned}
      %     &\leq C  \left\{
      %       (t-s)^p \left| a(X^\varepsilon_s,\mu_0) \right|^p
      %       + \varepsilon^p(t-s)^{p/2} \left| b(X^\varepsilon_s,\sigma_0) \right|^p \right. \\
      %     &\quad \left.
      %       + \varepsilon^p \left( \lambda_\varepsilon(t-s) + \lambda_\varepsilon^{p/2}(t-s)^{p/2} + \lambda_\varepsilon^p (t-s)^p \right) \left| c(X^\varepsilon_s, \alpha_0) \right|^p \right\} \\
      %     &\quad \times \exp \left( C \left((t-s)^p + \varepsilon^p(t-s)^{p/2}  + \varepsilon^p \lambda_\varepsilon (t-s) + \varepsilon^p \lambda_\varepsilon^{p/2}(t-s)^{p/2} + \varepsilon^p \lambda_\varepsilon^{p}(t-s)^{p} \right) \right).
      %   \end{aligned}
      % \end{align}
    \end{proof}

    \begin{lemma}\label{lem:4.2}
      Under Assumptions \ref{asmp:A2} and \ref{asmp:A3}, let
      $0<\varepsilon\leq1$, $\lambda_\varepsilon\geq1$, $\varepsilon\lambda_\varepsilon\leq1$ and $0\leq s<t\leq1$.
      Then, for $p\geq2$
      \begin{equation*}
        E \left[ \sup_{u\in[s,t]} \left| X^\varepsilon_u - x_u \right|^p \, \middle| \, \mathcal{F}_s \right]
        \leq C \varepsilon^p \left( (t-s)^{p/2}  + \lambda_\varepsilon (t-s) + \lambda_\varepsilon^{p/2}(t-s)^{p/2} + \lambda_\varepsilon^{p}(t-s)^{p} \right),
      \end{equation*}
      where $C$ depends only on $p$, $a$ and $b$.
    \end{lemma}

    \begin{proof}
      Same as the proof of Lemma \ref{lem:4.1},
      for any $p\geq2$, we obtain
      \begin{align*}
        &E \left[ \sup_{u\in[s,t]} \left| X^\varepsilon_u - x_u \right|^p \,
          \middle| \, \mathcal{F}_s \right] \\
        &\quad
          \leq C \varepsilon^p \left( (t-s)^{p/2}  + \lambda_\varepsilon (t-s) + \lambda_\varepsilon^{p/2}(t-s)^{p/2} + \lambda_\varepsilon^{p}(t-s)^{p} \right) \\
        &\quad\quad
          \times \exp \left( C \left\{ (t-s)^p + \varepsilon^p \left( (t-s)^{p/2}  + \lambda_\varepsilon (t-s) + \lambda_\varepsilon^{p/2}(t-s)^{p/2} + \lambda_\varepsilon^{p}(t-s)^{p} \right) \right\} \right),
      \end{align*}
      where $C$ depends only on $p,a,b,c$ and $f_{\alpha_0}$.
    \end{proof}

    \begin{lemma}\label{lem:4.3}
      Under Assumptions \ref{asmp:A2} and \ref{asmp:A3}, for $p\geq1$
      \begin{equation*}
        \left\| X^\varepsilon_\cdot - x_\cdot \right\|_{L^p(\Omega;L^\infty([0,1]))} = O(\varepsilon\lambda_\varepsilon)
      \end{equation*}
      as $\varepsilon\to0$, $\lambda_\varepsilon\to\infty$ and $\varepsilon\lambda_\varepsilon\to0$, and
      \begin{align*}
        \left\| \sup_{\substack{0\leq u,s\leq 1 \\ |u-s|\leq1/n}} \left | X^\varepsilon_u - x_s \right | \right\|_{L^p(\Omega)} = O(1/n+\varepsilon\lambda_\varepsilon)
      \end{align*}
      as $n\to\infty$, $\varepsilon\to0$, $\lambda_\varepsilon\to\infty$ and $\varepsilon\lambda_\varepsilon\to0$.
    \end{lemma}

    \begin{proof}
      Both rates of convergence are obtained immediately from Lemma \ref{lem:4.2}.
    \end{proof}

    \begin{lemma}\label{lem:4.4}
      Under Assumptions \ref{asmp:A2} and \ref{asmp:A3}, let
      a family $\{g(\cdot,\theta)\}_{\theta\in\Theta}$ of functions from $\mathbb{R}$ to $\mathbb{R}$ be equicontinuous at every points in $I_{x_0}$.
      Then,
      \begin{equation*}
        \frac{1}{n}
          \sum_{k=1}^n
            g \left( X^\varepsilon_{t_{k-1}}, \theta \right)
        \overset{p}\longrightarrow
        \int_0^1 g(x_t, \theta) \, dt
      \end{equation*}
      as $n\to\infty$, $\varepsilon\to0$, $\lambda_\varepsilon\to\infty$ and $\varepsilon\lambda_\varepsilon\to0$, uniformly in $\theta\in\Theta$.
    \end{lemma}

    \begin{proof}
      This follows from Lemmas \ref{lem:4.3} and \ref{lem:A.2}.
    \end{proof}

    \begin{lemma}\label{lem:4.5}
      Under Assumptions \ref{asmp:A2} and \ref{asmp:A3},
      let $0<\varepsilon\leq1$, $\lambda_\varepsilon\geq1$, $\varepsilon\lambda_\varepsilon\leq1$.
      Then, for any $p\in[1,\infty)$,
      \begin{equation*}
        E \left[ \sup_{t\in[t_{k-1},\tau_k)} \left| X^\varepsilon_t - X^\varepsilon_{t_{k-1}} \right|^p \, \middle| \, \mathcal{F}_{t_{k-1}} \right]
        \leq C \left( \frac{1}{n^p} + \frac{\varepsilon^p}{n^{p/2}} \right) \left( 1+|X^\varepsilon_{t_{k-1}}|^p \right),
      \end{equation*}
      where $C$ depends only on $p$, $a$ and $b$, and
      \begin{equation*}
        E \left[ \sup_{t\in[\eta_k,t_k]} \left| X^\varepsilon_t - X^\varepsilon_{t_k} \right|^p \, \middle| \, \mathcal{F}_{t_{k-1}} \right]
        \leq C \left( \frac{1}{n^p} + \frac{\varepsilon^p}{n^{p/2}} \right) \left( 1+|X^\varepsilon_{t_{k-1}}|^p \right),
      \end{equation*}
      where $C$ depends only on $p,a,b,c$ and $f_{\alpha_0}$.
    \end{lemma}

    \begin{proof}
      For $t\in[t_{k-1},\tau_{k})$ and $p\geq2$,
      \begin{align*}
        &\left( E \left[ \sup_{s\in[t_{k-1},t)} \left| X^\varepsilon_s - X^\varepsilon_{t_{k-1}} \right|^p \, \middle| \, \mathcal{F}_{t_{k-1}} \right] \right)^{1/p} \\
        &\qquad
        \begin{aligned}
          &\leq C \int_{t_{k-1}}^t \left( E \left[ \left| X^\varepsilon_s - X^\varepsilon_{t_{k-1}} \right|^p \, \middle| \, \mathcal{F}_{t_{k-1}} \right] \right)^{1/p} ds
          + \frac{1}{n} \left| a(X^\varepsilon_{t_{k-1}},\mu_0) \right| \\
          &\quad + C \varepsilon \left( \int_{t_{k-1}}^t \left( E \left[ \left| X^\varepsilon_s - X^\varepsilon_{t_{k-1}} \right|^p \, \middle| \, \mathcal{F}_{t_{k-1}} \right] \right)^{2/p} ds \right)^{1/2}
          + \frac{\varepsilon}{\sqrt{n}} \left| b(X^\varepsilon_{t_{k-1}},\sigma_0) \right|,
        \end{aligned}
      \end{align*}
      where $C$ depnds only on $p$, $a$ and $b$.
      By using Gronwall's inequality, we obtain
      \begin{equation*}
        \left( E \left[ \sup_{s\in[t_{k-1},t)} \left| X^\varepsilon_s - X^\varepsilon_{t_{k-1}} \right|^p \, \middle| \, \mathcal{F}_{t_{k-1}} \right] \right)^{2/p}
        \leq C e^{C(1/n+\varepsilon^2)t} \left( \frac{1}{n^2} + \frac{\varepsilon^2}{n} \right) (1+|X^\varepsilon_{t_{k-1}}|^2),
      \end{equation*}
      where $C$ depnds only on $p$, $a$ and $b$.
      Similarly,
      \begin{equation*}
        \left( E \left[ \sup_{s\in[\eta_k,t_k]} \left| X^\varepsilon_s - X^\varepsilon_{t_k} \right|^p \, \middle| \, \mathcal{F}_{t_{k-1}} \right] \right)^{2/p}
        \leq C \left( \frac{1}{n^2} + \frac{\varepsilon^2}{n} \right)
          \left( 1 + E \left[ |X^\varepsilon_{t_k}|^2 \, \middle| \, \mathcal{F}_{t_{k-1}} \right] \right),
      \end{equation*}
      where $C$ depnds only on $p$, $a$ and $b$.
      From Lemma \ref{lem:4.1}, we have
      \begin{equation*}
        E \left[ \sup_{u,s\in[t_{k-1},t_k]} \left| X^\varepsilon_u - X^\varepsilon_s \right|^p \, \middle| \, \mathcal{F}_{t_{k-1}} \right]
        \leq C \left( \frac{1}{n^p} + \varepsilon^p\frac{\lambda_\varepsilon}{n} \right) \left( 1+|X^\varepsilon_{t_{k-1}}|^p \right),
      \end{equation*}
      where $C$ depnds only on $p,a,b,c$ and $f_{\alpha_0}$.
      We can easily extend this result to the case $p\in[1,2)$ by using H\"older inequality.
    \end{proof}

    \begin{lemma}\label{lem:4.6}
      Under Assumptions \ref{asmp:A2} and \ref{asmp:A3}, let
      $0<\varepsilon\leq1$, $\lambda_\varepsilon\geq1$, $\varepsilon\lambda_\varepsilon\leq1$,
      and
      \begin{equation*}
        Y^\varepsilon_k := \sup_{t\in[t_{k-1},\tau_k)} \frac{| X^\varepsilon_t - X^\varepsilon_{t_{k-1}} |}{\varepsilon} + \sup_{t\in[\eta_k,t_k]} \frac{| X^\varepsilon_t - X^\varepsilon_{t_k}|}{\varepsilon}.
      \end{equation*}
      Then, for any $p\in(2,\infty)$,
      \begin{equation*}
        \sup_{k=1,\dots,n} Y^\varepsilon_k
        = O_p \left( \frac{1}{\varepsilon n^{1-1/p}} + \frac{1}{n^{1/2-1/p}} \right)
      \end{equation*}
      as $n\to\infty$, $\varepsilon\to0$, $\lambda_\varepsilon\to\infty$ and $\varepsilon\lambda_\varepsilon\to0$.
    \end{lemma}

    \begin{proof}
      By using Lemmas \ref{lem:4.4} and \ref{lem:4.5}, we have
      \begin{equation*}
        \sum_{k=1}^n E \left[ \left| Y_k^{\varepsilon} \right|^p \, \middle| \, \mathcal{F}_{t_{k-1}} \right]
        \leq C \left( \frac{n}{(\varepsilon n)^p} + \frac{n}{n^{p/2}} \right) \frac{1}{n} \sum_{k=1}^n \left( 1 + \left| X^\varepsilon_{t_{k-1}} \right|^p \right)
        = O_p \left( \frac{n}{(\varepsilon n)^p} + \frac{n}{n^{p/2}} \right)
      \end{equation*}
      as $n\to\infty$, $\varepsilon\to0$, $\lambda_\varepsilon\to\infty$ and $\varepsilon\lambda_\varepsilon\to0$.
      It follows from Lemma \ref{lem:A.3} that
      \begin{equation*}
        \sup_{k=1,\dots,n} |Y^\varepsilon_k|
        \leq \left( \sum_{k=1}^n \left| Y_k^{\varepsilon} \right|^p \right)^{1/p}
        = O_p \left( \frac{1}{\varepsilon n^{1-1/p}} + \frac{1}{n^{1/2-1/p}} \right).
      \end{equation*}
      as $n\to\infty$, $\varepsilon\to0$, $\lambda_\varepsilon\to\infty$ and $\varepsilon\lambda_\varepsilon\to0$.
    \end{proof}

  \subsection{Limit theorems}

    We make a version of Lemma 2.2 in Shimizu and Yoshida \cite{shimizu2006estimation} in the sequel, and remark that the estimations for $C^{n,\varepsilon,\rho}_{k,1}$ and $D^{n,\varepsilon,\rho}_{k,1}$, the notations of which are denoted in \ref{ntt:CkDk}, are essentially important (the estimation for $C^{n,\varepsilon,\rho}_{k,0}$ may be useless).

    \begin{lemma}\label{lem:4.7}
      Under Assumptions \ref{asmp:A2} to \ref{asmp:A6},
      let $0<\varepsilon\leq1$, $\lambda_\varepsilon\geq1$, $\varepsilon\lambda_\varepsilon\leq1$.
      Then, for $p\geq2$ and $\rho\in(0,1/2)$
      \begin{align*}
        &P \left[ C^{n,\varepsilon,\rho}_{k,0} \, \middle | \, \mathcal{F}_{t_{k-1}} \right]
          \geq e^{-\lambda_\varepsilon/n} \left\{ 1 - C \left( \frac{1}{n^{p(1-\rho)}} + \frac{\varepsilon^p}{n^{p(1/2-\rho)}} \right) \left( 1+|X^\varepsilon_{t_{k-1}}|^p \right) \right\}\\
        &P \left[ D^{n,\varepsilon,\rho}_{k,0} \, \middle | \, \mathcal{F}_{t_{k-1}} \right]
          \leq C \left( \frac{1}{n^{p(1-\rho)}} + \frac{\varepsilon^p}{n^{p(1/2-\rho)}} \right) \left( 1+|X^\varepsilon_{t_{k-1}}|^p \right) \\
        &P \left[ C^{n,\varepsilon,\rho}_{k,1} \, \middle | \, \mathcal{F}_{t_{k-1}} \right]
          \leq \frac{\lambda_\varepsilon}{n}
          \left\{ C \left( \frac{1}{n^{p(1-\rho)}} + \frac{\varepsilon^p}{n^{p(1/2-\rho)}}
          + \frac{\varepsilon^p\lambda_\varepsilon}{n} \right) \left( 1+|X^\varepsilon_{t_{k-1}}|^p \right)
          + \int_{|z|\leq\kappa/n^\rho} f_{\alpha_0}(z) \, dz \right\}, \\
        &P \left[ D^{n,\varepsilon,\rho}_{k,1} \, \middle | \, \mathcal{F}_{t_{k-1}} \right]
          \leq \frac{\lambda_\varepsilon}{n}
          \left\{ C \left( \frac{1}{n^{p(1-\rho)}} + \frac{\varepsilon^p}{n^{p(1/2-\rho)}}
          + \frac{\varepsilon^p\lambda_\varepsilon}{n} \right) \left( 1+|X^\varepsilon_{t_{k-1}}|^p \right)
          + 1 \right\}, \\
        &P \left[ C^{n,\varepsilon,\rho}_{k,2} \, \middle | \, \mathcal{F}_{t_{k-1}} \right]
          \leq \frac{\lambda_\varepsilon^2}{n^2}, \quad
          P \left[ D^{n,\varepsilon,\rho}_{k,2} \, \middle | \, \mathcal{F}_{t_{k-1}} \right]
          \leq \frac{\lambda_\varepsilon^2}{n^2},
      \end{align*}
      where $\kappa$ is given in \ref{ntt:N9},
      and $C$ depends only on $p,a,b,c,f_{\alpha_0}$ and $v_1$.
    \end{lemma}

    \begin{proof}
      We only give a proof under Assumption \ref{asmp:A4} \ref{asmp:A4(i)}, because the same argument still works under Assumption \ref{asmp:A4} \ref{asmp:A4(ii)}.
      Same as in the proof of Lemma 2.2 in Shimizu and Yoshida \cite[Section 4.2]{shimizu2006estimation},
      it follows that
      \begin{equation*}
        P \left[ C^{n,\varepsilon,\rho}_{k,2} \, \middle | \, \mathcal{F}_{t_{k-1}} \right],~
        P \left[ D^{n,\varepsilon,\rho}_{k,2} \, \middle | \, \mathcal{F}_{t_{k-1}} \right]
        \leq \frac{\lambda_\varepsilon^2}{n^2}.
      \end{equation*}
      Also, it follows from
      \begin{align*}
        &P \left[
          \left| X^\varepsilon_{t_k} - X^\varepsilon_{\tau_k}
          + \Delta X^\varepsilon_{\tau_k} + X^\varepsilon_{\tau_k-} - X^\varepsilon_{t_{k-1}} \right|
          \leq \frac{v_{nk}}{n^\rho}
          \, \middle | \, \mathcal{F}_{t_{k-1}},~\Delta^n_k N^{\lambda_\varepsilon}=1 \right] \\
        & \quad
        \begin{aligned}
          &\leq P \left[ \left| X^\varepsilon_{t_k} - X^\varepsilon_{\tau_k} \right|
          + \sup_{t\in[t_{k-1},\tau_k)} \left| X^\varepsilon_t - X^\varepsilon_{t_{k-1}} \right| > \frac{v_{nk}}{n^\rho}
          \, \middle | \, \mathcal{F}_{t_{k-1}},~\Delta^n_k N^{\lambda_\varepsilon}=1 \right] \\
          &\quad + P \left[ \left| \Delta Z^{\lambda_\varepsilon}_{\tau_k} \right| \leq \frac{4v_{nk}}{c_1n^\rho}
          ~\text{or}~
          \sup_{t\in[t_k,t_{k-1}]} \left| c(X^\varepsilon_t,\alpha_0) - c(x_t, \alpha_0) \right| > \frac{c_1}{2}
          \, \middle | \, \mathcal{F}_{t_{k-1}},~\Delta^n_k N^{\lambda_\varepsilon}=1 \right], \\
        \end{aligned}\\
        &P \left[
          \left| X^\varepsilon_{t_k} - X^\varepsilon_{\tau_k}
          + \Delta X^\varepsilon_{\tau_k} + X^\varepsilon_{\tau_k-} - X^\varepsilon_{t_{k-1}} \right|
          > \frac{v_{nk}}{n^\rho}
          \, \middle | \, \mathcal{F}_{t_{k-1}},~\Delta^n_k N^{\lambda_\varepsilon}=1 \right] \\
        & \quad
        \begin{aligned}
          &\leq P \left[ \left| X^\varepsilon_{t_k} - X^\varepsilon_{\tau_k} \right|
          + \sup_{t\in[t_{k-1},\tau_k)} \left| X^\varepsilon_t - X^\varepsilon_{t_{k-1}} \right| > \frac{v_{nk}}{2n^\rho}
          \, \middle | \, \mathcal{F}_{t_{k-1}},~\Delta^n_k N^{\lambda_\varepsilon}=1 \right] \\
          &\quad + P \left[ \left| \Delta Z^{\lambda_\varepsilon}_{\tau_k} \right| > \frac{v_{nk}}{4c_2n^\rho}
          ~\text{or}~
          \sup_{t\in[t_k,t_{k-1}]} \left| c(X^\varepsilon_t,\alpha_0) - c(x_t, \alpha_0) \right| > c_2
          \, \middle | \, \mathcal{F}_{t_{k-1}},~\Delta^n_k N^{\lambda_\varepsilon}=1 \right] \\
        \end{aligned}
      \end{align*}
      and Lemmas \ref{lem:4.2}, \ref{lem:4.5} and \ref{lem:A.2} that
      \begin{align*}
        &P \left[ C^{n,\varepsilon,\rho}_{k,1} \, \middle | \, \mathcal{F}_{t_{k-1}} \right] \\
        &\quad
          \leq \frac{\lambda_\varepsilon}{n} e^{-\lambda_\varepsilon/n}
          \left\{ C \left( \frac{1}{n^{p(1-\rho)}} + \frac{\varepsilon^p}{n^{p(1/2-\rho)}} \right) \left( 1+|X^\varepsilon_{t_{k-1}}|^p \right)
          + \int_{|z|\leq4v_{nk}/c_1n^\rho} f_{\alpha_0}(z) \, dz
          + C \frac{\varepsilon^p \lambda_\varepsilon}{n} \right\}, \\
        &P \left[ D^{n,\varepsilon,\rho}_{k,1} \, \middle | \, \mathcal{F}_{t_{k-1}} \right] \\
        &\quad
          \leq \frac{\lambda_\varepsilon}{n} e^{-\lambda_\varepsilon/n}
          \left\{ C \left( \frac{1}{n^{p(1-\rho)}} + \frac{\varepsilon^p}{n^{p(1/2-\rho)}} \right) \left( 1+|X^\varepsilon_{t_{k-1}}|^p \right)
          + \int_{|z|>v_{nk}/4c_2n^\rho} f_{\alpha_0}(z) \, dz
          + C \frac{\varepsilon^p \lambda_\varepsilon}{n} \right\},
      \end{align*}
      where $C$ depends only on $p,a,b,c,f_{\alpha_0}$ and $v_1$.
      Here, note that $\kappa:=4v_2/c_1\geq4v_{nk}/c_1$.
      The other inequalities follow from Lemma \ref{lem:4.5}.
    \end{proof}

    In the proof of Proposition 3.3 (ii) in Shimizu and Yoshida \cite{shimizu2006estimation}, the intensity of the Poisson process driving on the background is constant, although we assume the intensity $\lambda_\varepsilon$ goes to infinity. So, we prepare the following lemma.

    \begin{lemma}\label{lem:4.8}
      Under Assumptions \ref{asmp:A2} to \ref{asmp:A6},
      for $\rho\in(0,1/2)$
      \begin{align*}
        \frac{1}{\lambda_\varepsilon} \sum_{k=1}^n  1_{D^{n,\varepsilon,\rho}_{k}}
        \overset{p}\longrightarrow 1,
      \end{align*}
      as $n\to\infty$, $\varepsilon\to0$, $\lambda_\varepsilon\to\infty$, $\lambda_\varepsilon/n \to 0$ and $\varepsilon\lambda_\varepsilon\to0$.
      More precisely, for $\rho\in(0,1/2)$ and $p\in[2/(1-2\rho),\infty)$
      \begin{align*}
        \frac{1}{\lambda_\varepsilon} \sum_{k=1}^n  1_{D^{n,\varepsilon,\rho}_{k,0}}
          &= O_p \left( \frac{1}{\lambda_\varepsilon n^{p(1-\rho)-1}}
          + \frac{\varepsilon^p}{\lambda_\varepsilon n^{p(1/2-\rho)-1}} \right), \\
        \frac{1}{\lambda_\varepsilon} \sum_{k=1}^n  1_{D^{n,\varepsilon,\rho}_{k,1}}
          &= 1 + O_p \left( \frac{\lambda_\varepsilon}{n}
            + \frac{1}{n^{p(1-\rho)}}
            + \frac{\varepsilon^p}{n^{p(1/2-\rho)}}
            + \int_{|z|\leq\kappa/n^\rho} f_{\alpha_0}(z) \, dz \right), \\
        \frac{1}{\lambda_\varepsilon} \sum_{k=1}^n  1_{D^{n,\varepsilon,\rho}_{k,2}}
          &= O_p \left( \frac{\lambda_\varepsilon}{n} \right)
      \end{align*}
      as $n\to\infty$, $\varepsilon\to0$, $\lambda_\varepsilon\to\infty$ and $\varepsilon\lambda_\varepsilon\to0$.
    \end{lemma}

    \begin{proof} Since
      \begin{align*}
        \left| \frac{\lambda_\varepsilon}{n} - P \left[ D^{n,\varepsilon,\rho}_{k,1} \, \middle | \, \mathcal{F}_{t_{k-1}} \right] \right|
        &\leq \frac{\lambda_\varepsilon}{n} - \frac{\lambda_\varepsilon}{n} e^{-\lambda_\varepsilon/n}
        + \left| \frac{\lambda_\varepsilon}{n} e^{-\lambda_\varepsilon/n} - P \left[ D^{n,\varepsilon,\rho}_{k,1} \, \middle | \, \mathcal{F}_{t_{k-1}} \right] \right| \\
        &\leq \left( \frac{\lambda_\varepsilon}{n} \right)^2 + P \left[ C^{n,\varepsilon,\rho}_{k,1} \, \middle | \, \mathcal{F}_{t_{k-1}} \right],
      \end{align*}
      it follows from Lemmas \ref{lem:4.4} and \ref{lem:4.7} that for $p\geq2$ and $\rho\in(0,1/2)$
      \begin{align*}
        &\sum_{k=1}^n E \left[ \left| \frac{1}{\lambda_\varepsilon} 1_{D^{n,\varepsilon,\rho}_{k,1}} - \frac{1}{n} \right| \, \middle | \, \mathcal{F}_{t_{k-1}} \right] \\
        &\quad \leq \frac{\lambda_\varepsilon}{n}
          + \frac{1}{\lambda_\varepsilon}\sum_{k=1}^n P \left[ C^{n,\varepsilon,\rho}_{k,1} \, \middle | \, \mathcal{F}_{t_{k-1}} \right] \\
        &\quad \leq \frac{\lambda_\varepsilon}{n} +
          C \left( \frac{1}{n^{p(1-\rho)}} + \frac{\varepsilon^p}{n^{p(1/2-\rho)}}
          + \frac{\varepsilon^p\lambda_\varepsilon}{n} \right) \frac{1}{n} \sum_{k=1}^n\left( 1+|X^\varepsilon_{t_{k-1}}|^p \right)
          + \int_{|z|\leq\kappa/n^\rho} f_{\alpha_0}(z) \, dz
      \end{align*}
      converges to zero in probability as $n\to\infty$, $\varepsilon\to0$, $\lambda_\varepsilon\to\infty$, $\lambda_\varepsilon/n \to 0$ and $\varepsilon\lambda_\varepsilon\to0$.
      Similarly, we obtain
      \begin{align*}
        \sum_{k=1}^n E \left[ \frac{1}{\lambda_\varepsilon} 1_{D^{n,\varepsilon,\rho}_{k,0}} \, \middle | \, \mathcal{F}_{t_{k-1}} \right]
        &\leq C \frac{1}{\lambda_\varepsilon} \left( \frac{1}{n^{p(1-\rho)}} + \frac{\varepsilon^p}{n^{p(1/2-\rho)}} \right)
          \sum_{k=1}^n \left( 1+|X^\varepsilon_{t_{k-1}}|^p \right), \\
        \sum_{k=1}^n E \left[ \frac{1}{\lambda_\varepsilon} 1_{D^{n,\varepsilon,\rho}_{k,2}} \, \middle | \, \mathcal{F}_{t_{k-1}} \right]
        &\leq \frac{\lambda_\varepsilon}{n}.
      \end{align*}
      Hence, the conclusion follows from Lemma \ref{lem:A.3}.
    \end{proof}

    \begin{lemma}\label{lem:4.9}
      Under Assumptions \ref{asmp:A2} to \ref{asmp:A6},
      for $\rho\in(0,1/2)$
      \begin{align*}
        \frac{1}{n} \sum_{k=1}^n  1_{C^{n,\varepsilon,\rho}_{k}}
        \overset{p}\longrightarrow 1,
      \end{align*}
      as $n\to\infty$, $\varepsilon\to0$, $\lambda_\varepsilon\to\infty$, $\lambda_\varepsilon/n \to 0$ and $\varepsilon\lambda_\varepsilon\to0$.
      More precisely, for $\rho\in(0,1/2)$ and $p\in[2,\infty)$
      \begin{align*}
        \frac{1}{n} \sum_{k=1}^n  1_{C^{n,\varepsilon,\rho}_{k,0}}
          &= 1 + O_p \left( \frac{\lambda_\varepsilon}{n} \right), \\
        \frac{1}{n} \sum_{k=1}^n  1_{C^{n,\varepsilon,\rho}_{k,1}}
          &= O_p \left( \frac{\lambda_\varepsilon}{n^{p(1-\rho)+1}}
            + \frac{\varepsilon^p \lambda_\varepsilon}{n^{p(1/2-\rho)+1}}
            + \frac{\lambda_\varepsilon}{n} \int_{|z|\leq\kappa/n^\rho} f_{\alpha_0}(z) \, dz \right), \\
        \frac{1}{n} \sum_{k=1}^n  1_{C^{n,\varepsilon,\rho}_{k,2}}
          &= O_p \left( \frac{\lambda_\varepsilon^2}{n^2} \right)
      \end{align*}
      as $n\to\infty$, $\varepsilon\to0$, $\lambda_\varepsilon\to\infty$, $\lambda_\varepsilon/n \to 0$ and $\varepsilon\lambda_\varepsilon\to0$.
    \end{lemma}

    \begin{proof}
      From Lemma \ref{lem:4.8} we have
      \begin{equation*}
        \frac{1}{n} \sum_{k=1}^n  1_{C^{n,\varepsilon,\rho}_{k}} - 1
        = \frac{\lambda_\varepsilon}{n} \frac{1}{\lambda_\varepsilon} \sum_{k=1}^n  1_{D^{n,\varepsilon,\rho}_{k}}
        = O_p \left( \frac{\lambda_\varepsilon}{n} \right)
      \end{equation*}
      as $n\to\infty$, $\varepsilon\to0$, $\lambda_\varepsilon\to\infty$, $\lambda_\varepsilon/n \to 0$ and $\varepsilon\lambda_\varepsilon\to0$.
      It follows from Lemmas \ref{lem:4.4} and \ref{lem:4.7} that for any $p\in[0,\infty)$
      \begin{align*}
        &\sum_{k=1}^n E \left[ \left| \frac{1}{n} 1_{C^{n,\varepsilon,\rho}_{k,1}} \right| \, \middle | \, \mathcal{F}_{t_{k-1}} \right]
        = \frac{1}{n} \sum_{k=1}^n P \left[ C^{n,\varepsilon,\rho}_{k,1} \, \middle | \, \mathcal{F}_{t_{k-1}} \right] \\
        &\quad
          \leq C \frac{\lambda_\varepsilon}{n}
            \left( \frac{1}{n^{p(1-\rho)}} + \frac{\varepsilon^p}{n^{p(1/2-\rho)}}
            + \frac{\varepsilon^p\lambda_\varepsilon}{n} \right)
            \frac{1}{n} \sum_{k=1}^n \left( 1+|X^\varepsilon_{t_{k-1}}|^p \right)
            + \frac{\lambda_\varepsilon}{n} \int_{|z|\leq\kappa/n^\rho} f_{\alpha_0}(z) \, dz,
      \end{align*}
      and
      \begin{equation*}
        \sum_{k=1}^n E \left[ \left| \frac{1}{n} 1_{C^{n,\varepsilon,\rho}_{k,2}} \right| \, \middle | \, \mathcal{F}_{t_{k-1}} \right]
        = \frac{1}{n} \sum_{k=1}^n P \left[ C^{n,\varepsilon,\rho}_{k,2} \, \middle | \, \mathcal{F}_{t_{k-1}} \right]
        \leq \frac{\lambda_\varepsilon^2}{n^2}.
      \end{equation*}
      The conclusion follows from Lemma \ref{lem:A.3}.
    \end{proof}

    \begin{remark}\label{rmk:sumIndC}
      Under Assumptions \ref{asmp:A2} to \ref{asmp:A6},
      let $\rho\in(0,1/2)$
      and let $\xi_{k,\theta}^{n,\varepsilon}$ $(k=1,\dots,n,~n\in\mathbb{N},~\varepsilon>0,~\theta\in\bar{\Theta})$ be random variables.
      From Lemma \ref{lem:4.9},
      if
      \begin{equation*}
        \lambda_\varepsilon \to \infty,         ~
        \varepsilon \lambda_\varepsilon \to 0,  ~
        \frac{\lambda_\varepsilon^2}{n} \to 0   ~\text{and}~
        \lambda_\varepsilon \int_{|z|\leq\kappa/n^\rho} f_{\alpha_0}(z) \, dz \to 0
      \end{equation*}
      as $\varepsilon\to0$, then
      \begin{equation*}
        \sum_{k=1}^n
          \xi_{k,\theta}^{n,\varepsilon}
          \left\{ 1_{C^{n,\varepsilon,\rho}_{k}}
            - 1_{C^{n,\varepsilon,\rho}_{k,0}} \right\}
        = o_p(1)
      \end{equation*}
      as $n\to\infty$ and $\varepsilon\to0$, uniformly in $\theta\in\Theta$,
      since for any $\eta>0$
      \begin{equation*}
        P \left( \sup_{\theta\in\Theta} \left|
          \sum_{k=1}^n
            \xi_{k,\theta}^{n,\varepsilon}
            1_{C^{n,\varepsilon,\rho}_{k,j}}
          \right| > \eta \right)
        \leq P \left( \left| \sum_{k=1}^n  1_{C^{n,\varepsilon,\rho}_{k,j}} \right| > 1/2 \right)
        \quad \text{for } j=1,2.
      \end{equation*}
      Similarly, from Lemma \ref{lem:4.8},
      when
      \begin{equation*}
        \lambda_\varepsilon \to \infty, \quad
        \varepsilon \lambda_\varepsilon \to 0, \quad
        \frac{\lambda_\varepsilon^2}{n} \to 0
      \end{equation*}
      as $\varepsilon\to0$,
      \begin{equation*}
        \sum_{k=1}^n
          \xi_{k,\theta}^{n,\varepsilon}
          \left\{ 1_{D^{n,\varepsilon,\rho}_{k}}
            - 1_{D^{n,\varepsilon,\rho}_{k,1}} \right\}
        = o_p(1)
      \end{equation*}
      as $n\to\infty$ and $\varepsilon\to0$, uniformly in $\theta\in\Theta$,
    \end{remark}

    \begin{lemma}\label{lem:4.10}
      Under Assumptions \ref{asmp:A2} to \ref{asmp:A6},
      let $\rho\in(0,1/2)$, $\delta>0$,
      \begin{equation*}
        \tilde{D}^{n,\varepsilon,\rho}_{k,1}
        := D^{n,\varepsilon,\rho}_{k,1}
          \cap \{ X^\varepsilon_t \in I_{x_0}^\delta \text{ for all } t\in[0,1] \},
      \end{equation*}
      and let $\xi_{k,\theta}^{n,\varepsilon}$ $(k=1,\dots,n,~n\in\mathbb{N},~\varepsilon>0,~\theta\in\bar{\Theta})$
      be random variables.
      If
      \begin{equation}\label{eq:indicatorSumCondition}
        \lambda_\varepsilon \to \infty,         ~
        \varepsilon \lambda_\varepsilon \to 0,  ~
        \frac{\lambda_\varepsilon^2}{n} \to 0   ~\text{and}~
        \lambda_\varepsilon \int_{|z|\leq\kappa/n^\rho} f_{\alpha_0}(z) \, dz \to 0
      \end{equation}
      as $\varepsilon\to0$, then
      \begin{equation*}
        \sum_{k=1}^n
          \xi_{k,\theta}^{n,\varepsilon}
          \left\{ 1_{D^{n,\varepsilon,\rho}_{k}}
            - 1_{\tilde{D}^{n,\varepsilon,\rho}_{k,1}} \right\}
        = o_p(1), \quad
        \sum_{k=1}^n
          \xi_{k,\theta}^{n,\varepsilon}
          \left\{ 1_{\tilde{D}^{n,\varepsilon,\rho}_{k,1}} - 1_{J^{n,\varepsilon}_{k,1}} \right\}
        = o_p(1)
      \end{equation*}
      as $n\to\infty$ and $\varepsilon\to0$, uniformly in $\theta\in\Theta$.
    \end{lemma}

    \begin{proof}
      Since from Lemma \ref{lem:4.3}
      \begin{equation*}
        P \left( X^\varepsilon_t \in I_{x_0}^\delta \text{ for all } t\in[0,1] \right)
        \geq P \left( \sup_{t\in[0,1]} | X^\varepsilon_t - x_t | \leq \delta \right) \to 1
      \end{equation*}
      as $n\to\infty$, $\varepsilon\to0$ and $\varepsilon\lambda_\varepsilon\to0$,
      for any $\eta>0$
      \begin{align*}
        &P \left(
          \left|
            \sum_{k=1}^n
              \xi_{k,\theta}^{n,\varepsilon}
                \left\{ 1_{D^{n,\varepsilon,\rho}_{k}}
                - 1_{\tilde{D}^{n,\varepsilon,\rho}_{k,1}} \right\}
          \right| > \eta \right) \\
        & \qquad
          \leq
            P \left(
              \left|
                \sum_{k=1}^n
                  \xi_{k,\theta}^{n,\varepsilon}
                    \left\{ 1_{D^{n,\varepsilon,\rho}_{k}}
                    - 1_{D^{n,\varepsilon,\rho}_{k,1}} \right\}
              \right| > \eta/2 \right)
            + P \left( \{ X^\varepsilon_t \not\in I_{x_0}^\delta, ~ \exists \, t\in[0,1] \} \right), \\
        & P \left(
          \left|
            \sum_{k=1}^n
              \xi_{k,\theta}^{n,\varepsilon}
              \left\{ 1_{\tilde{D}^{n,\varepsilon,\rho}_{k,1}}
              - 1_{J^{n,\varepsilon}_{k,1}} \right\}
          \right| > \eta \right) \\
        & \qquad
        \leq P \left( \{ X^\varepsilon_t \not\in I_{x_0}^\delta, ~ \exists \, t\in[0,1] \} \right)
          + P \left(
            \left|
              \sum_{k=1}^n
                \xi_{k,\theta}^{n,\varepsilon}
                  1_{C^{n,\varepsilon,\rho}_{k,1}}
            \right| > \eta/2 \right).
      \end{align*}
      Take sufficiently large $p\in[2/(1-2\rho),\infty)$.
      Thus, we obtain from Remark \ref{rmk:sumIndC} the conclusion.
    \end{proof}

    \begin{remark}%\label{rmk: }
      In this lemma, if $\{\xi_{k,\theta}^{n,\varepsilon}\}_{n,\varepsilon,k,\theta}$ is bounded in probability,
      we can replace the condition \eqref{eq:indicatorSumCondition} with a milder condition
      \begin{equation*}
        \lambda_\varepsilon \to \infty, ~
        \varepsilon \lambda_\varepsilon \to 0 ~ \text{and}~
        \frac{\lambda_\varepsilon}{n} \to 0.
      \end{equation*}
      But, we will never use this fact in this paper.
    \end{remark}

    \begin{lemma}\label{lem:4.11}
      Under Assumptions \ref{asmp:A2} to \ref{asmp:A6}, let $\rho\in(0,1/2)$
      and let a family $\{g(\cdot,\theta)\}_{\theta\in\Theta}$ of functions from $\mathbb{R}$ to $\mathbb{R}$
      be equicontinuous at every points in $I_{x_0}$.
      Then,
      \begin{equation*}
        \frac{1}{n}
          \sum_{k=1}^n
            g ( X^\varepsilon_{t_{k-1}}, \theta )
            1_{C^{n,\varepsilon,\rho}_k}
        \overset{p}\longrightarrow
        \int_0^1 g(x_t, \theta) \, dt
      \end{equation*}
      as $n\to\infty$, $\varepsilon\to0$, $\lambda_\varepsilon\to\infty$, $\lambda_\varepsilon/n\to0$ and $\varepsilon\lambda_\varepsilon\to0$, uniformly in $\theta\in\Theta$.
      Also, for $p\in[2,\infty)$
      \begin{align*}
        &\frac{1}{n} \sum_{k=1}^n g(X^\varepsilon_{t_{k-1}}, \theta) 1_{C^{n,\varepsilon,\rho}_{k,0}}
        \overset{p}\longrightarrow \int_0^1 g(x_t,\theta) \, dt, \\
        &\frac{1}{n} \sum_{k=1}^n g(X^\varepsilon_{t_{k-1}}, \theta) 1_{C^{n,\varepsilon,\rho}_{k,1}}
          = O_p \left( \frac{\lambda_\varepsilon}{n^{p(1-\rho)+1}}
          + \frac{\varepsilon^p \lambda_\varepsilon}{n^{p(1/2-\rho)+1}}
          + \frac{\lambda_\varepsilon}{n} \int_{|z|\leq\kappa/n^\rho} f_{\alpha_0}(z) \, dz \right), \\
        &\frac{1}{n} \sum_{k=1}^n g(X^\varepsilon_{t_{k-1}}, \theta) 1_{C^{n,\varepsilon,\rho}_{k,2}}
          = O_p \left( \frac{\lambda_\varepsilon^2}{n^2} \right)
      \end{align*}
      as $n\to\infty$, $\varepsilon\to0$, $\lambda_\varepsilon\to\infty$, $\lambda_\varepsilon/n\to0$ and $\varepsilon\lambda_\varepsilon\to0$, uniformly in $\theta\in\Theta$.
    \end{lemma}

    \begin{proof}[Proof of Lemma \ref{lem:4.11}]
      Since $\{g(\cdot,\theta)\}_{\theta\in\Theta}$ is equicontinuous at every points in $I_{x_0}$,
      there exists $\delta>0$ such that
      \begin{equation*}
        \sup_{(x,\theta)\in I_{x_0}^\delta \times \Theta} |g(x,\theta)| < \infty.
      \end{equation*}
      For any $\eta>0$
      \begin{align*}
        &P \left( \left| \frac{1}{n} \sum_{k=1}^n g(X^\varepsilon_{t_{k-1}}, \theta) 1_{D^{n,\varepsilon,\rho}_{k}} \right| > \eta \right) \\
        &\quad
        \leq P \left( \sup_{k=0,\dots,n-1} |X^\varepsilon_{t_k} - x_{t_k}| \geq \delta \right)
        + P \left( \sup_{(x,\theta)\in I_{x_0}^\delta \times \Theta} |g(x,\theta)| \frac{\lambda_\varepsilon}{n} \frac{1}{\lambda_\varepsilon} \sum_{k=1}^n 1_{D^{n,\varepsilon,\rho}_k} > \eta \right),
      \end{align*}
      and for $j=1,2$,
      \begin{align*}
        &P \left( \left| \frac{1}{n} \sum_{k=1}^n g(X^\varepsilon_{t_{k-1}}, \theta) 1_{C^{n,\varepsilon,\rho}_{k,j}} \right| > \eta \right) \\
        &\quad
        \leq P \left( \sup_{k=0,\dots,n-1} |X^\varepsilon_{t_k} - x_{t_k}| \geq \delta \right)
        + P \left( \sup_{(x,\theta)\in I_{x_0}^\delta \times \Theta} |g(x,\theta)| \frac{1}{n} \sum_{k=1}^n 1_{C^{n,\varepsilon,\rho}_{k,j}} > \eta \right).
      \end{align*}
      It follows from Lemmas \ref{lem:4.3}, \ref{lem:4.4} and \ref{lem:4.9} that
      \begin{align*}
        &\left | \frac{1}{n} \sum_{k=1}^n g(X^\varepsilon_{t_{k-1}}, \theta) 1_{C^{n,\varepsilon,\rho}_k}
        -\int_0^1 g(x_t,\theta) \, dt \right | \\
        &\quad
        \leq \left | \frac{1}{n} \sum_{k=1}^n g(X^\varepsilon_{t_{k-1}}, \theta) 1_{D^{n,\varepsilon,\rho}_k} \right |
        + \left | \frac{1}{n} \sum_{k=1}^n g(X^\varepsilon_{t_{k-1}}, \theta) - \int_0^1 g(x_t,\theta) \, dt \right |
        \overset{p}\longrightarrow 0, \\
        &\frac{1}{n} \sum_{k=1}^n g(X^\varepsilon_{t_{k-1}}, \theta) 1_{C^{n,\varepsilon,\rho}_{k,1}}
          = O_p \left( \frac{\lambda_\varepsilon}{n^{p(1-\rho)+1}}
          + \frac{\varepsilon^p \lambda_\varepsilon}{n^{p(1/2-\rho)+1}}
          + \frac{\lambda_\varepsilon}{n} \int_{|z|\leq\kappa/n^\rho} f_{\alpha_0}(z) \, dz \right), \\
        &\frac{1}{n} \sum_{k=1}^n g(X^\varepsilon_{t_{k-1}}, \theta) 1_{C^{n,\varepsilon,\rho}_{k,2}}
          = O_p \left( \frac{\lambda_\varepsilon^2}{n^2} \right)
      \end{align*}
      as $n\to\infty$, $\varepsilon\to0$, $\lambda_\varepsilon\to\infty$, $\lambda_\varepsilon/n\to0$ and $\varepsilon\lambda_\varepsilon\to0$, uniformly in $\theta\in\Theta$.
    \end{proof}

    \begin{lemma}\label{lem:4.12}
      Under Assumptions \ref{asmp:A2} to \ref{asmp:A6},
      let $\rho\in(0,1/2)$.
      We assume either of the following conditions \ref{item:asmp:ContrastJumpPreConv1(i)}
      or \ref{item:asmp:ContrastJumpPreConv1(ii)}:
      \begin{enumerate}[label=(\roman*)]
        \item \label{item:asmp:ContrastJumpPreConv1(i)}
              Under Assumption \ref{asmp:A4} \ref{asmp:A4(i)},
              we assume the following two conditions:
              \begin{enumerate}[label=(i.\alph*)]
                \item \label{item:asmp:ContrastJumpPreConv1(i-a)}
                      There exists $\delta>0$ such that for every $(x,\theta)\in I_{x_0}^\delta\times\bar{\Theta}$, $g(x,y,\theta)$ is continuously differentiable with respect to $y \in \mathbb{R}$.
                \item There exist constants $C>0$, $q\geq1$ and $\delta>0$ such that
                      \begin{equation*}
                        \sup_{(x,\theta) \in I_{x_0}^\delta\times\bar{\Theta}}
                        \left| \frac{\partial g}{\partial y} (x,y,\theta) \right|
                        \leq C (1+|y|^q)
                        \quad (y \in \mathbb{R}).
                      \end{equation*}
                \item \label{item:asmp:ContrastJumpPreConv1(i-c)}
                      There exists a sufficiently large $p\geq2$ such that
                      \begin{equation*}
                        \lambda_\varepsilon             \to \infty, ~
                        \frac{\lambda_\varepsilon^2}{n} \to 0,      ~
                        \varepsilon \lambda_\varepsilon \to 0,      ~
                        \varepsilon n^{1-1/p}\to\infty              ~\text{and}~
                        \lambda_\varepsilon \int_{|z|\leq\kappa/n^\rho} f_{\alpha_0}(z) \, dz \to 0
                      \end{equation*}
                      as $n\to\infty$ and $\varepsilon\to0$.
                \item \label{item:asmp:ContrastJumpPreConv1(i-d)}
                      Let $p$ be taken as in the condition \ref{item:asmp:ContrastJumpPreConv1(i-c)}.
                      Put $r_{n,\varepsilon}$ by
                      \begin{equation*}
                        r_{n,\varepsilon} := \frac{1}{\varepsilon n^{1-1/p}} + \frac{1}{n^{1/2-1/p}}.
                      \end{equation*}
              \end{enumerate}
        \item \label{item:asmp:ContrastJumpPreConv1(ii)}
              Under Assumption \ref{asmp:A4} \ref{asmp:A4(ii)},
              we assume the following four conditions:
              \begin{enumerate}[label=(ii.\alph*)]
                \item \label{item:asmp:ContrastJumpPreConv1(ii-a)}
                      There exists $\delta>0$ such that for every $(x,\theta)\in I_{x_0}^\delta\times\bar{\Theta}$, $g(x,y,\theta)$ is continuously differentiable with respect to $y \in (0,\infty)$.
                \item \label{item:asmp:ContrastJumpPreConv1(ii-b)}
                      There exists $\delta>0$ and $L>0$ such that if $0<y_1\leq y \leq y_2$, then
                      \begin{equation*}
                        \left| \frac{\partial g}{\partial y} (x,y,\theta) \right|
                        \leq \left| \frac{\partial g}{\partial y} (x,y_1,\theta) \right|
                        + \left| \frac{\partial g}{\partial y} (x,y_2,\theta) \right| + L
                        \quad
                        \text{for all }
                        (x,\theta)\in I_{x_0}^\delta\times\bar{\Theta}.
                      \end{equation*}
                \item \label{item:asmp:ContrastJumpPreConv1(ii-c)}
                      There exist $q\geq0$ and $\delta>0$ such that
                      \begin{equation*}
                        \sup_{(x,\theta) \in I_{x_0}^\delta\times\bar{\Theta}}
                        \left| \frac{\partial g}{\partial y} (x,y,\theta) \right|
                        \leq O \left( \frac{1}{|y|^q} \right)
                        \quad \text{as } |y| \to 0.
                      \end{equation*}
                \item \label{item:asmp:ContrastJumpPreConv1(ii-d)}
                      There exists $\delta>0$ such that
                      for any $C_1>0$ and $C_2\geq0$
                      the map
                      \begin{equation*}
                        x \mapsto \int \sup_\theta \left| \frac{\partial g}{\partial y}(x, C_1 y + C_2,\theta) \right| f_{\alpha_0}(y) \, dy
                      \end{equation*}
                      takes values in $\mathbb{R}$ from $I_{x_0}^\delta$, and is continous on $I_{x_0}^\delta$.
                \item \label{item:asmp:ContrastJumpPreConv1(ii-e)}
                      Let $q$ be the constant in the condition \ref{item:asmp:ContrastJumpPreConv1(ii-c)},
                      and let $\rho<1/4q$.
                      For any large $p\geq2/(1-2q\rho)$,
                      \begin{equation*}
                        \lambda_\varepsilon             \to \infty, ~
                        \frac{\lambda_\varepsilon^2}{n} \to 0,      ~
                        \varepsilon \lambda_\varepsilon \to 0,      ~
                        \varepsilon n^{1-q\rho-1/p}\to\infty        ~\text{and}~
                        \lambda_\varepsilon \int_{|z|\leq\kappa/n^\rho} f_{\alpha_0}(z) \, dz \to 0
                      \end{equation*}
                      as $n\to\infty$, $\varepsilon\to0$.
                \item \label{item:asmp:ContrastJumpPreConv1(ii-f)}
                      Let $p$ and $q$ be the constants in the condition \ref{item:asmp:ContrastJumpPreConv1(ii-e)}.
                      Put $r_{n,\varepsilon}$ by
                      \begin{equation*}
                        r_{n,\varepsilon} := \frac{1}{\varepsilon n^{1-1/p-q\rho}} + \frac{1}{n^{1/2-1/p-q\rho}}.
                      \end{equation*}
              \end{enumerate}
      \end{enumerate}
      Then,
      \begin{align*}
        &\left|
            \frac{1}{\lambda_\varepsilon}
              \sum_{k=1}^n
                g \left( X^\varepsilon_{t_{k-1}}, \frac{\Delta^n_k X^{\varepsilon}}{\varepsilon}, \theta \right)
                1_{D^{n,\varepsilon,\rho}_{k}}
            - \frac{1}{\lambda_\varepsilon}
              \sum_{k=1}^n
                g \left( X^\varepsilon_{t_{k-1}}, c(X^\varepsilon_{t_{k-1}},\alpha_0) V_{N^{\lambda_\varepsilon}_{\tau_k}}, \theta \right)
                1_{J^{n,\varepsilon}_{k,1}}
          \right| \\
        &\quad
          = O_p(r_{n,\varepsilon})
      \end{align*}
      as $n\to\infty$ and $\varepsilon\to0$, uniformly in $\theta\in\Theta$.
    \end{lemma}

    \begin{remark}
      In Lemmas \ref{lem:4.7}, \ref{lem:4.8}, \ref{lem:4.10} and \ref{lem:4.11},
      Assumption \ref{asmp:A4} is needed only for defining $D^{n,\varepsilon,\rho}_{k}$.
      Assumption \ref{asmp:A4} is essentially used in this lemma.
    \end{remark}

    \begin{remark}
      The conditions \ref{item:asmp:ContrastJumpPreConv1(i-c)}
      and \ref{item:asmp:ContrastJumpPreConv1(ii-e)} in Lemma \ref{lem:4.12}
      are ensured if
      \begin{equation*}
        \lambda_\varepsilon\to\infty,         ~
        \varepsilon \lambda_\varepsilon\to0,  ~
        (\varepsilon \sqrt{n})^{-1} <\infty   ~\text{and}~
        \lambda_\varepsilon \int_{|z|\leq\kappa/n^\rho} f_{\alpha_0}(z) \, dz \to 0
      \end{equation*}
      as $n\to\infty$ and $\varepsilon\to0$.
      This condition seems to be natural when we consider the asymptotic normality for our estimator
      (see, e.g., the condition (B2) in S{\o}rensen and Uchida \cite{sorensen2003small}).
    \end{remark}

    \begin{proof}[Proof of Lemma \ref{lem:4.12}]
      Let $\delta>0$ be a sufficiently small number satisfying the conditions of the statement and
      \begin{equation*}
        \frac{c_1}{2} \leq c(x,\alpha_0) \leq 2 c_2 \quad \text{for } x\in I_{x_0}^\delta,
      \end{equation*}
      where $c_1$ and $c_2$ are the constants from Assumption \ref{asmp:A6}.
      In this proof, we may simply write the maps
      \begin{equation*}
        (y,\theta) \mapsto g(X^\varepsilon_{t_{k-1}},y,\theta) =: g_k(y,\theta) \quad
        \text{and} \quad
        (y,\theta) \mapsto \left. \frac{\partial g}{\partial y} (x,y,\theta) \right|_{x=X^\varepsilon_{t_{k-1}}} =: \frac{\partial g_k}{\partial y} (y, \theta),
      \end{equation*}
      and we denote the following event by $\tilde{D}^{n,\varepsilon,\rho}_{k,1}$
      \begin{equation*}
        \tilde{D}^{n,\varepsilon,\rho}_{k,1}
        := D^{n,\varepsilon,\rho}_{k,1}
          \cap \{ X^\varepsilon_t \in I_{x_0}^\delta \text{ for all } t\in[0,1] \}.
      \end{equation*}

      Since
      \begin{equation*}
        \frac{\lambda_\varepsilon^2}{n} \to 0 \quad\text{and}\quad
        \lambda_\varepsilon \int_{|z|\leq\kappa/n^\rho} f_{\alpha_0}(z) \, dz \to 0
      \end{equation*}
      under either of the conditions \ref{item:asmp:ContrastJumpPreConv1(i-c)}
      or \ref{item:asmp:ContrastJumpPreConv1(ii-e)},
      we obtain from Lemma \ref{lem:4.10} that for any non-random $r_{n,\varepsilon}'>0$ ($n\in\mathbb{N}, \varepsilon>0$),
      \begin{align*}
        % \label{eq:ContrastJumpPreConv1(1)}
        \frac{1}{\lambda_\varepsilon} \sum_{k=1}^n
          g_k \left( \frac{\Delta^n_k X^{\varepsilon}}{\varepsilon}, \theta \right)
          \left\{ 1_{D^{n,\varepsilon,\rho}_{k}} - 1_{\tilde{D}^{n,\varepsilon,\rho}_{k,1}} \right\}
          &= o_p(r_{n,\varepsilon}'), \\
        % \label{eq:ContrastJumpPreConv1(2)}
        \frac{1}{\lambda_\varepsilon} \sum_{k=1}^n
          g_k \left( c(X^\varepsilon_{t_{k-1}},\alpha_0) V_{N^{\lambda_\varepsilon}_{\tau_k}}, \theta \right)
          \left\{ 1_{\tilde{D}^{n,\varepsilon,\rho}_{k,1}} - 1_{J^{n,\varepsilon}_{k,1}} \right\}
          &= o_p(r_{n,\varepsilon}')
      \end{align*}
      as $n\to\infty$, $\varepsilon\to0$ $\lambda_\varepsilon\to\infty$ and $\varepsilon \lambda_\varepsilon\to0$, uniformly in $\theta\in\Theta$.
      Thus, it is sufficient to show that
      \begin{align}
        \label{eq:ContrastJumpPreConv1(3)}
        \frac{1}{\lambda_\varepsilon} \sum_{k=1}^n \left\{
          g_k \left( \frac{\Delta^n_k X^{\varepsilon}}{\varepsilon}, \theta \right)
          - g_k \left( \frac{\Delta X^\varepsilon_{\tau_k}}{\varepsilon}, \theta \right)
          \right\} 1_{\tilde{D}^{n,\varepsilon,\rho}_{k,1}}
          &= O_p(r_{n,\varepsilon}), \\
        \label{eq:ContrastJumpPreConv1(4)}
        \frac{1}{\lambda_\varepsilon} \sum_{k=1}^n \left\{
          g_k \left( \frac{\Delta X^\varepsilon_{\tau_k}}{\varepsilon}, \theta \right)
          - g_k \left( c(X^\varepsilon_{t_{k-1}},\alpha_0) V_{N^{\lambda_\varepsilon}_{\tau_k}}, \theta \right)
          \right\} 1_{\tilde{D}^{n,\varepsilon,\rho}_{k,1}}
          &= O_p \left( \frac{1}{\varepsilon n^{1-1/p}} + \frac{1}{n^{1/2-1/p}} \right)
      \end{align}
      as $n\to\infty$, $\varepsilon\to0$, $\lambda_\varepsilon\to\infty$ and $\varepsilon \lambda_\varepsilon\to0$, uniformly in $\theta\in\Theta$.

      Put
      \begin{equation*}
        Y^\varepsilon_k := \frac{X^\varepsilon_{t_k} - X^\varepsilon_{\eta_k}}{\varepsilon} + \frac{X^\varepsilon_{\tau_k-} - X^\varepsilon_{t_{k-1}}}{\varepsilon}
        \left( = \frac{\Delta^n_k X^{\varepsilon}}{\varepsilon} - \frac{\Delta X^\varepsilon_{\tau_k}}{\varepsilon} \quad \text{on}~D^{n,\varepsilon,\rho}_{k,1}\right).
      \end{equation*}
      By using Taylor's theorem under either of the conditions
      \ref{item:asmp:ContrastJumpPreConv1(i-a)}
      or \ref{item:asmp:ContrastJumpPreConv1(ii-a)}, we have
      \begin{equation*}
            g_k \left( \frac{\Delta^n_k X^{\varepsilon}}{\varepsilon}, \theta \right)
            - g_k \left( \frac{\Delta X^\varepsilon_{\tau_k}}{\varepsilon}, \theta \right)
        = \int_0^1 \frac{\partial g_k}{\partial y} \left( \frac{\Delta X^\varepsilon_{\tau_k}}{\varepsilon}  +\zeta Y^\varepsilon_k, \theta \right) Y^\varepsilon_k \, d\zeta \quad
        \text{on } \tilde{D}^{n,\varepsilon,\rho}_{k,1}.
      \end{equation*}
      Here, we remark that $\Delta^n_k X^{\varepsilon}$ and $\Delta X^\varepsilon_{\tau_k}$ are almost surely positive on $\tilde{D}^{n,\varepsilon,\rho}_{k,1}$ under Assumption \ref{asmp:A3} \ref{asmp:A4(ii)}.
      To see \eqref{eq:ContrastJumpPreConv1(3)}, it is sufficient to show that
      \begin{equation}\label{eq:ContrastJumpPreConv1(5)}
        \sup_{\theta\in\Theta} \left|
          \frac{1}{\lambda_\varepsilon} \sum_{k=1}^n
          \int_0^1 \frac{\partial g_k}{\partial y} \left( \frac{\Delta X^\varepsilon_{\tau_k}}{\varepsilon}  +\zeta Y^\varepsilon_k, \theta \right) Y^\varepsilon_k \, d\zeta \,
          1_{\tilde{D}^{n,\varepsilon,\rho}_{k,1} \cap \{|Y^\varepsilon_k|\leq 1\}} \right|
        = O_p(r_{n,\varepsilon})
      \end{equation}
      as $n\to\infty$, $\varepsilon\to0$, $\lambda_\varepsilon\to\infty$ and $\varepsilon \lambda_\varepsilon\to0$.
      Indeed, for any $M>0$
      \begin{align*}
        &P \left(
          \sup_{\theta\in\Theta} \left|
          \frac{1}{\lambda_\varepsilon} \sum_{k=1}^n \left\{
          g_k \left( \frac{\Delta^n_k X^{\varepsilon}}{\varepsilon}, \theta \right)
          - g_k \left( \frac{\Delta X^\varepsilon_{\tau_k}}{\varepsilon}, \theta \right)
          \right\} 1_{\tilde{D}^{n,\varepsilon,\rho}_{k,1}} \right|
          > M r_{n,\varepsilon} \right) \\
        &~
          \leq P \left( \sup_{k=1,\dots,n}|Y^\varepsilon_k| > 1 \right)
          + P \left(
            \sup_{\theta\in\Theta} \left|
            \frac{1}{\lambda_\varepsilon} \sum_{k=1}^n
            \int_0^1 \frac{\partial g_k}{\partial y} \left( \frac{\Delta X^\varepsilon_{\tau_k}}{\varepsilon}  +\zeta Y^\varepsilon_k, \theta \right) Y^\varepsilon_k \, d\zeta \,
            1_{\tilde{D}^{n,\varepsilon,\rho}_{k,1} \cap \{|Y^\varepsilon_k|\leq 1\}} \right|
            > M r_{n,\varepsilon} \right),
      \end{align*}
      and from Lemma \ref{lem:4.6} the first term converges to zero as $n\to\infty$, $\varepsilon\to0$, $\lambda_\varepsilon\to\infty$ and $\varepsilon \lambda_\varepsilon\to0$,
      since from either of the conditions \ref{item:asmp:ContrastJumpPreConv1(i-c)}
      or \ref{item:asmp:ContrastJumpPreConv1(ii-e)}
      we have $\varepsilon n^{1-1/p}\to\infty$ or $\varepsilon n^{1-q\rho-1/p}\to\infty$, respectively.

      We first consider the case under Assumption \ref{asmp:A3} \ref{asmp:A4(ii)}.
      Since for $\zeta\in[0,1]$ we have
      \begin{equation*}
        \frac{\Delta X^\varepsilon_{\tau_k}}{\varepsilon} + \zeta Y^\varepsilon_k
        \geq (1-\zeta) \, c(X^\varepsilon_{\tau_k-},\alpha_0) \, V_{N^{\lambda_\varepsilon}_{\tau_k}} + \zeta \frac{v_{nk}}{n^\rho}
        \geq \min \left\{ \frac{c_1}{2} V_{N^{\lambda_\varepsilon}_{\tau_k}}, \frac{v_1}{n^\rho} \right\}
        ~\text{on}~ \tilde{D}^{n,\varepsilon,\rho}_{k,1},
      \end{equation*}
      we obtain from the condition \ref{item:asmp:ContrastJumpPreConv1(ii-b)} that
      \begin{align*}
        &\int_0^1 \left| \frac{\partial g_k}{\partial y}
          \left( \frac{\Delta X^\varepsilon_{\tau_k}}{\varepsilon}  +\zeta Y^\varepsilon_k, \theta \right) \right| d\zeta
          \, 1_{\tilde{D}^{n,\varepsilon,\rho}_{k,1} \cap \{|Y^\varepsilon_k|\leq 1\}} \\
        &\quad
          \begin{aligned}
            &\leq \left\{
              \left| \frac{\partial g_k}{\partial y} \left( \frac{c_1}{2} V_{N^{\lambda_\varepsilon}_{\tau_k}}, \theta \right) \right|
              + \left| \frac{\partial g_k}{\partial y} \left( \frac{v_1}{n^\rho}, \theta \right) \right|
              + \left| \frac{\partial g_k}{\partial y} \left( 2c_2 V_{N^{\lambda_\varepsilon}_{\tau_k}} + 1, \theta \right) \right| + L \right\}
              1_{\tilde{D}^{n,\varepsilon,\rho}_{k,1} \cap \{|Y^\varepsilon_k|\leq 1\}} \\
            &\leq \left\{
              \left| \frac{\partial g_k}{\partial y} \left( \frac{c_1}{2} V_{N^{\lambda_\varepsilon}_{\tau_k}}, \theta \right) \right|
              + \left| \frac{\partial g_k}{\partial y} \left( \frac{v_1}{n^\rho}, \theta \right) \right|
              + \left| \frac{\partial g_k}{\partial y} \left( 2c_2 V_{N^{\lambda_\varepsilon}_{\tau_k}} + 1, \theta \right) \right| + L \right\}
              1_{J^{n,\varepsilon}_{k,1}}
          \end{aligned}
      \end{align*}
      Since
      \begin{align*}
        \frac{1}{\lambda_\varepsilon} \sum_{k=1}^n
          E \left[ \sup_\theta \left| \frac{\partial g_k}{\partial y} \left( \frac{c_1}{2} V_{N^{\lambda_\varepsilon}_{\tau_k}}, \theta \right) \right|
          1_{J^{n,\varepsilon}_{k,1}}
          \, \middle| \, \mathcal{F}_{t_{k-1}} \right]
        &
          \leq \frac{1}{\lambda_\varepsilon} \sum_{k=1}^n
          \int \sup_\theta \left| \frac{\partial g_k}{\partial y} \left( \frac{c_1}{2} z,  \theta \right) \right| f_{\alpha_0}(z) \, dz \cdot P \left( J^{n,\varepsilon}_{k,1} \right) \\
        &
          \leq \frac{1}{n} \sum_{k=1}^n
          \int \sup_\theta \left| \frac{\partial g_k}{\partial y} \left( \frac{c_1}{2} z,  \theta \right) \right| f_{\alpha_0}(z) \, dz,
      \end{align*}
      it follows from Lemmas \ref{lem:4.4}, \ref{lem:4.6} and \ref{lem:A.3} \ref{item:ConseqFromLenglart(ii)},
      and the condition \ref{item:asmp:ContrastJumpPreConv1(ii-d)} that
      \begin{equation*}
        \frac{1}{\lambda_\varepsilon} \sum_{k=1}^n \left|
          \frac{\partial g_k}{\partial y} \left( \frac{c_1}{2} V_{N^{\lambda_\varepsilon}_{\tau_k}}, \theta \right)
          \right|
          1_{J^{n,\varepsilon}_{k,1}}
          \sup_{k=1,\dots,n} \left| Y^\varepsilon_k \right|
        = O_p \left( \frac{1}{\varepsilon n^{1-1/p}} + \frac{1}{n^{1/2-1/p}} \right)
      \end{equation*}
      as $n\to\infty$, $\varepsilon\to0$, $\lambda_\varepsilon\to\infty$ and $\varepsilon \lambda_\varepsilon\to0$, uniformly in $\theta\in\Theta$,
      where $p$ is given in the condition \ref{item:asmp:ContrastJumpPreConv1(ii-e)}.
      Similarly,
      it follows from Lemmas \ref{lem:4.4}, \ref{lem:4.6} and \ref{lem:A.3} \ref{item:ConseqFromLenglart(ii)},
      and the condition \ref{item:asmp:ContrastJumpPreConv1(ii-d)} that
      \begin{equation*}
        \frac{1}{\lambda_\varepsilon} \sum_{k=1}^n \left|
          \frac{\partial g_k}{\partial y} \left( 2c_2 V_{N^{\lambda_\varepsilon}_{\tau_k}} + 1, \theta \right) \right|
          1_{J^{n,\varepsilon}_{k,1}}
          \sup_{k=1,\dots,n} \left| Y^\varepsilon_k \right|
        = O_p \left( \frac{1}{\varepsilon n^{1-1/p}} + \frac{1}{n^{1/2-1/p}} \right)
      \end{equation*}
      as $n\to\infty$, $\varepsilon\to0$, $\lambda_\varepsilon\to\infty$ and $\varepsilon \lambda_\varepsilon\to0$, uniformly in $\theta\in\Theta$, and
      it follows from Lemmas \ref{lem:4.6} and \ref{lem:A.3} \ref{item:ConseqFromLenglart(ii)},
      and the condition \ref{item:asmp:ContrastJumpPreConv1(ii-c)} that
      \begin{equation*}
        \frac{1}{\lambda_\varepsilon} \sum_{k=1}^n \left|
        \frac{\partial g_k}{\partial y} \left( \frac{v_1}{n^\rho}, \theta \right)
          \right| 1_{J^{n,\varepsilon}_{k,1}}
          \sup_{k=1,\dots,n} \left| Y^\varepsilon_k \right|
        = O_p \left( \frac{1}{\varepsilon n^{1-1/p-q\rho}} + \frac{1}{n^{1/2-1/p-q\rho}} \right)
      \end{equation*}
      as $n\to\infty$, $\varepsilon\to0$, $\lambda_\varepsilon\to\infty$ and $\varepsilon \lambda_\varepsilon\to0$, uniformly in $\theta\in\Theta$.
      Thus, we obtain \eqref{eq:ContrastJumpPreConv1(5)}.

      Under Assumption \ref{asmp:A3} \ref{asmp:A4(i)}, as in the same argument in the above, we have
      \begin{align*}
        &\frac{1}{\lambda_\varepsilon} \sum_{k=1}^n \sup_{\theta\in\bar{\Theta}} \left|
          \int_0^1 \frac{\partial g_k}{\partial y} \left( \frac{\Delta X^\varepsilon_{\tau_k}}{\varepsilon}  +\zeta Y^\varepsilon_k, \theta \right) Y^\varepsilon_k \, d\zeta \right|
          1_{\tilde{D}^{n,\varepsilon,\rho}_{k,1}\cap\{|Y^\varepsilon_k|\leq1\}} \\
        &\qquad
          \leq \frac{C}{\lambda_\varepsilon} \sum_{k=1}^n
          \left( 2 + \left| 2c_2 V_{N^{\lambda_\varepsilon}_{\tau_k}} \right|^p \right)
          1_{J^{n,\varepsilon}_{k,1}}
          \sup_{k=1,\dots,n} \left| Y^\varepsilon_k \right|
        = O_p \left( \frac{1}{\varepsilon n^{1-1/p}} + \frac{1}{n^{1/2-1/p}} \right)
      \end{align*}
      as $n\to\infty$, $\varepsilon\to0$, $\lambda_\varepsilon\to\infty$ and $\varepsilon \lambda_\varepsilon\to0$.
      Thus, we obtain \eqref{eq:ContrastJumpPreConv1(5)}.

      Analogously, it follows that for $\zeta\in[0,1]$
      \begin{equation*}
        (1-\zeta) \frac{\Delta X^\varepsilon_{\tau_k}}{\varepsilon}
          + \zeta c(X^\varepsilon_{t_{k-1}},\alpha_0) V_{N^{\lambda_\varepsilon}_{\tau_k}}
        \geq \frac{c_1}{2} V_{N^{\lambda_\varepsilon}_{\tau_k}}
        \quad
        \text{on } \tilde{D}^{n,\varepsilon,\rho}_{k,1},
      \end{equation*}
      and that on $\tilde{D}^{n,\varepsilon,\rho}_{k,1}$
      \begin{align*}
        &\int_0^1 \left| \frac{\partial g_k}{\partial y}
          \left( (1-\zeta) \frac{\Delta X^\varepsilon_{\tau_k}}{\varepsilon}
          + \zeta c(X^\varepsilon_{t_{k-1}},\alpha_0) V_{N^{\lambda_\varepsilon}_{\tau_k}}, \theta \right) \right| d\zeta \\
        &\qquad
        \leq \left\{ \begin{aligned}
            & C \left( 1 + \left| 2 c_2 V_{N^{\lambda_\varepsilon}_{\tau_k}} \right|^p \right)
            & &\text{under the condition \ref{item:asmp:ContrastJumpPreConv1(i)}}, \\
            & \left| \frac{\partial g_k}{\partial y}
              \left( \frac{c_1}{2} V_{N^{\lambda_\varepsilon}_{\tau_k}}, \theta \right) \right|
              + \left| \frac{\partial g_k}{\partial y} \left( 2c_2 V_{N^{\lambda_\varepsilon}_{\tau_k}}, \theta \right) \right|
            & &\text{under the condition \ref{item:asmp:ContrastJumpPreConv1(ii)}},
          \end{aligned} \right.
      \end{align*}
      so that, \eqref{eq:ContrastJumpPreConv1(4)} holds.
    \end{proof}

    \begin{lemma}\label{lem:4.13}
      Let $\rho\in(0,1/2)$.
      Under Assumptions \ref{asmp:A2} to \ref{asmp:A6},
      suppose that for $\theta\in\Theta$
      \begin{equation}\label{eq:ContrastJumpPreConv2(1)}
        x \mapsto \int g(x,c(x,\alpha_0)z,\theta) f_{\alpha_0}(z) \, dz, \quad
        x \mapsto \int |g(x,c(x,\alpha_0)z,\theta)|^2 f_{\alpha_0}(z) \, dz
      \end{equation}
      are continuous at every points in $I_{x_0}$, and that there exist $\delta>0$, $C>0$ and $q\geq0$ such that
      \begin{equation}\label{eq:ContrastJumpPreConv2(2)}
          \int \left\{
            \sup_{(x,\theta)\in I_{x_0}^\delta \times \Theta}
            \left| g (x,c(x,\alpha_0)z,\theta) \right|
            + \sum_{j=1}^d \sup_{(x,\theta)\in I_{x_0}^\delta \times \Theta}
            \left| \frac{\partial g}{\partial \theta_j} (x,c(x,\alpha_0)z,\theta) \right|
            \right\} f_{\alpha_0} (z) \, dz
            < \infty.
      \end{equation}
      Then,
      \begin{equation*}
        \frac{1}{\lambda_\varepsilon} \sum_{k=1}^n
          g \left( X^\varepsilon_{t_{k-1}}, c(X^\varepsilon_{t_{k-1}},\alpha_0) V_{N^{\lambda_\varepsilon}_{\tau_k}}, \theta \right)
          1_{D^{n,\varepsilon,\rho}_{k,1}}
        \overset{p}\longrightarrow
        \int_0^1 \int g(x_t,c(x_t,\alpha_0)z,\theta) f_{\alpha_0}(z) \, dz \, dt
      \end{equation*}
      as $n\to\infty$, $\varepsilon\to0$, $\lambda_\varepsilon\to\infty$ and $\varepsilon \lambda_\varepsilon\to0$, uniformly in $\theta\in\Theta$.
    \end{lemma}

    \begin{proof}
      It follows from Lemma \ref{lem:4.4} and the continuity of \eqref{eq:ContrastJumpPreConv2(1)} that for each $\theta\in\Theta$
      \begin{align*}
        \sum_{k=1}^n E \left[
          \frac{1}{\lambda_\varepsilon}
          g \left( X^\varepsilon_{t_{k-1}}, c(X^\varepsilon_{t_{k-1}},\alpha_0) V_{N^{\lambda_\varepsilon}_{\tau_k}}, \theta \right)
          1_{J^{n,\varepsilon}_{k,1}}
          \, \middle| \, \mathcal{F}_{t_{k-1}} \right]
        &= \frac{1}{n}
          \sum_{k=1}^n \int
          g \left( X^\varepsilon_{t_{k-1}}, c(X^\varepsilon_{t_{k-1}},\alpha_0) z, \theta \right) f_{\alpha_0}(z) \, dz \\
        &\overset{p}\longrightarrow
          \int_0^1 \int g(x_t,c(x_t,\alpha_0)z,\theta) f_{\alpha_0}(z) \, dz \, dt
      \end{align*}
      as $n\to\infty$, $\varepsilon\to0$, $\lambda_\varepsilon\to\infty$ and $\varepsilon\lambda_\varepsilon\to0$,
      and that
      \begin{equation*}
        \sum_{k=1}^n E \left[
            \frac{1}{\lambda_\varepsilon^2}
            \left| g \left( X^\varepsilon_{t_{k-1}}, c(X^\varepsilon_{t_{k-1}},\alpha_0) V_{N^{\lambda_\varepsilon}_{\tau_k}}, \theta \right) \right|^2
            1_{J^{n,\varepsilon}_{k,1}}
          \, \middle| \, \mathcal{F}_{t_{k-1}} \right]
        \overset{p}\longrightarrow 0
      \end{equation*}
      as $n\to\infty$, $\varepsilon\to0$, $\lambda_\varepsilon\to\infty$ and $\varepsilon\lambda_\varepsilon\to0$.
      Thus, Lemma 9 in Genon-Catalot and Jocod \cite{genon-catalot1993ontheest} shows us that for each $\theta\in\Theta$
      \begin{equation*}
        \frac{1}{\lambda_\varepsilon}
          \sum_{k=1}^n
              g \left( X^\varepsilon_{t_{k-1}}, c(X^\varepsilon_{t_{k-1}},\alpha_0) V_{N^{\lambda_\varepsilon}_{\tau_k}}, \theta \right)
              1_{J^{n,\varepsilon}_{k,1}}
        \overset{p}\longrightarrow
        \int_0^1 \int g(x_t,c(x_t,\alpha_0)z,\theta) f_{\alpha_0}(z) \, dz \, dt
      \end{equation*}
      as $n\to\infty$, $\varepsilon\to0$, $\lambda_\varepsilon\to\infty$ and $\varepsilon\lambda_\varepsilon\to0$.
      Put
      \begin{equation*}
        \tilde{J}^{n,\varepsilon}_{k,1}
        := J^{n,\varepsilon}_{k,1} \cap  \{ X^\varepsilon_t \in I_{x_0}^\delta \text{ for all } t\in[0,1] \}.
      \end{equation*}
      Then, by the same argument in the proof of Lemma \ref{lem:4.10}, it follows from Lemma \ref{lem:4.3} that
      \begin{equation*}
        \frac{1}{\lambda_\varepsilon}
          \sum_{k=1}^n
              g \left( X^\varepsilon_{t_{k-1}}, c(X^\varepsilon_{t_{k-1}},\alpha_0) V_{N^{\lambda_\varepsilon}_{\tau_k}}, \theta \right)
              \left\{ 1_{J^{n,\varepsilon}_{k,1}} - 1_{\tilde{J}^{n,\varepsilon}_{k,1}} \right\}
        \overset{p}\longrightarrow
        0
      \end{equation*}
      as $n\to\infty$, $\varepsilon\to0$, $\lambda_\varepsilon\to\infty$ and $\varepsilon \lambda_\varepsilon\to0$, uniformly in $\theta\in\Theta$.

      Now, we have for each $\theta\in\Theta$
      \begin{equation*}
        \frac{1}{\lambda_\varepsilon}
          \sum_{k=1}^n
              g \left( X^\varepsilon_{t_{k-1}}, c(X^\varepsilon_{t_{k-1}},\alpha_0) V_{N^{\lambda_\varepsilon}_{\tau_k}}, \theta \right)
              1_{\tilde{J}^{n,\varepsilon}_{k,1}}
        \overset{p}\longrightarrow
        \int_0^1 \int g(x_t,c(x_t,\alpha_0)z,\theta) f_{\alpha_0}(z) \, dz \, dt
      \end{equation*}
      as $n\to\infty$, $\varepsilon\to0$, $\lambda_\varepsilon\to\infty$ and $\varepsilon\lambda_\varepsilon\to0$.
      To say the uniformity of this convergence in $\theta\in\Theta$,
      put
      \begin{equation*}
        \chi^{n,\varepsilon}(\theta)
        :=
        \frac{1}{\lambda_\varepsilon}
          \sum_{k=1}^n
              g \left( X^\varepsilon_{t_{k-1}}, c(X^\varepsilon_{t_{k-1}},\alpha_0) V_{N^{\lambda_\varepsilon}_{\tau_k}}, \theta \right)
              1_{\tilde{J}^{n,\varepsilon}_{k,1}}
        - \int_0^1 \int g(x_t,c(x_t,\alpha_0)z,\theta) f_{\alpha_0}(z) \, dz \, dt
      \end{equation*}
      and we shall use Theorem 5.1 in Billingsley \cite{billingsley1999convergence} with the state space $C(\Theta)$, same as in the proofs of Propositions 3.3 and 3.6 in Shimizu and Yoshida \cite{shimizu2006estimation}.\footnote{We cannot use Theorem 20 in Ibragimov and Has'minskii \cite[Appendix I]{Ibragimov1981statistical} (or Lemma 3.1 in Yoshida \cite{yoshida1990asymptotic}),
      as in the proof of Lemma 2 in S{\o}rensen and Uchida \cite{sorensen2003small}.
      In fact, we fail to say that $\{\chi^{n,\varepsilon}\}$ satisfies (1) and (2) in Lemma 3.1 in Yoshida \cite{yoshida1990asymptotic}.}
      From \eqref{eq:ContrastJumpPreConv2(2)}, we obtain
      \begin{align*}
        E \left[ \sup_{\theta\in\Theta} \left| \frac{1}{\lambda_\varepsilon} \sum_{k=1}^n
          g\left( X^\varepsilon_{t_{k-1}}, c(X^\varepsilon_{t_{k-1}},\alpha_0) V_{N^{\lambda_\varepsilon}_{\tau_k}}, \theta \right)
          1_{\tilde{J}^{n,\varepsilon}_{k,1}}
          \right| \right]
        &\leq \frac{1}{\lambda_\varepsilon} \sum_{k=1}^n E \left[
          \sup_{(x,\theta)\in I_{x_0}^\delta \times \Theta}
          \left| g(x,c(x,\alpha_0) V_{N^{\lambda_\varepsilon}_{\tau_k}},\theta) \right|
          1_{J^{n,\varepsilon}_{k,1}}
          \right] \\
        &= \int
          \sup_{(x,\theta)\in I_{x_0}^\delta \times \Theta}
          \left| g(x,c(x,\alpha_0)z,\theta) \right|
          f_{\alpha_0} (z) \, dz \, (< \infty)
      \end{align*}
      and
      \begin{align*}
        &E \left[ \sup_{\theta\in\Theta} \left| \frac{1}{\lambda_\varepsilon} \sum_{k=1}^n
          \frac{\partial g}{\partial \theta_j} \left( X^\varepsilon_{t_{k-1}}, c(X^\varepsilon_{t_{k-1}},\alpha_0) V_{N^{\lambda_\varepsilon}_{\tau_k}}, \theta \right)
          1_{\tilde{J}^{n,\varepsilon}_{k,1}}
          \right| \right] \\
        & \qquad
        \begin{aligned}
          &\leq \frac{1}{\lambda_\varepsilon} \sum_{k=1}^n E \left[
            \sup_{(x,\theta)\in I_{x_0}^\delta \times \Theta}
            \left| \frac{\partial g}{\partial \theta_j} (x,c(x,\alpha_0) V_{N^{\lambda_\varepsilon}_{\tau_k}},\theta) \right|
            1_{J^{n,\varepsilon}_{k,1}}
            \right] \\
          &= \int
            \sup_{(x,\theta)\in I_{x_0}^\delta \times \Theta}
            \left| \frac{\partial g}{\partial \theta_j} (x,c(x,\alpha_0)z,\theta) \right|
            f_{\alpha_0} (z) \, dz \, (< \infty)
        \end{aligned}
      \end{align*}
      for $j=1,\dots,d$.
      The above equalities hold from the fact that $V_{N^{\lambda_\varepsilon}_{\tau_k}}$ and $1_{J^{n,\varepsilon}_{k,1}}$ are independent.
      Hence, for any closed ball $B_r$ of radius $r>0$ centered at zero in $W^{1,\infty}(\Theta)$,
      we obtain from Markov's inequality that
      \begin{equation*}
        \sup_{n,\varepsilon} P \left( \chi^{n,\varepsilon}
          \not\in B_r \right)
        = P \left( \| \chi^{n,\varepsilon} \|_{W^{1,\infty}(\Theta)} \geq r \right)
        \leq \frac{2C}{r},
      \end{equation*}
      where $C$ is defined as \eqref{eq:ContrastJumpPreConv2(2)}.
      From Rellich-Kondrachov's theorem (see, e.g., Theorem 9.16 in Brezis \cite{brezis2010functional}),
      it follows that the balls $B_r$, $r>0$ are compact in $C(\Theta)$, and so from Theorem 5.1 in Billingsley \cite{billingsley1999convergence} that
      $\{\chi^{n,\varepsilon}\}$ is relatively compact in distribution sense as in the Billingsley's book.
      Since for each $\theta\in\Theta$ $\{\chi^{n,\varepsilon}(\theta)\}$ converges to zero in probability, all convergent subsequences of $\{\chi^{n,\varepsilon}\}$ converges to zero in probability.
      Analogously, all subnet of $\{\chi^{n,\varepsilon}\}$ has a subsequence convergent in probability to zero,
      and so $\{\chi^{n,\varepsilon}\}$ converges to zero in probability as $n\to\infty$, $\varepsilon\to0$, $\lambda_\varepsilon\to\infty$ and $\varepsilon\lambda_\varepsilon\to0$.
    \end{proof}

    \begin{lemma}\label{lem:4.14}
      Under Assumptions \ref{asmp:A2} to \ref{asmp:A6},
      let $\rho\in(0,1/2)$,
      and suppose that a function $g:\mathbb{R}\times\Theta\to\mathbb{R}$
      satisfies
      that $\{\frac{\partial g}{\partial \theta_j}(\cdot,\theta)\}_{\theta\in\Theta}$, $j=1,\dots,d$ are equi-Lipschitz continuous on $I_{x_0}^\delta$ for some small $\delta>0$.
      Then,
      \begin{equation*}
        \frac{1}{\varepsilon} \sum_{k=1}^n
          g ( X^\varepsilon_{t_{k-1}}, \theta )
          \left\{ \Delta^n_k X^{\varepsilon} - \frac{1}{n} a(X^\varepsilon_{t_{k-1}}, \mu_0) \right\}
          1_{C^{n,\varepsilon,\rho}_k}
        \overset{p}\longrightarrow
        \int_0^1 g(x_t, \theta) b(x_t,\theta) \, dW_t
      \end{equation*}
      as $n\to\infty$, $\varepsilon\to0$, $\lambda_\varepsilon\to\infty$, $\lambda_\varepsilon^2/n \to 0$, $\varepsilon\lambda_\varepsilon\to0$ and $\lambda_\varepsilon \int_{|z|\leq\kappa/n^\rho} f_{\alpha_0}(z) \, dz\to0$, uniformly in $\theta\in\Theta$.
    \end{lemma}

    \begin{proof}
      At first, we can easily check that
      \begin{align}\label{eq:E[g(Xk)bdW1Ck](1)}
        \frac{1}{\varepsilon} \sum_{k=1}^n
          g ( X^\varepsilon_{t_{k-1}}, \theta )
          \left\{ \int_{t_{k-1}}^{t_k} a(X^\varepsilon_t,\mu_0) \, dt  - \frac{1}{n} a(X^\varepsilon_{t_{k-1}}, \mu_0) \right\}
          1_{C^{n,\varepsilon,\rho}_k}
          \overset{p}\longrightarrow 0
      \end{align}
      as $n\to\infty$, $\varepsilon\to0$, $\lambda_\varepsilon\to\infty$, $\varepsilon n\to\infty$, and $\varepsilon\lambda_\varepsilon\to0$, uniformly in $\theta\in\Theta$.
      Indeed, this follows from Lemmas \ref{lem:4.3}, \ref{lem:A.2} and \ref{lem:A.3} with the equicontinuity of $g$ on $I_{x_0}$ and the following estimate:
      \begin{align*}
        &\frac{1}{\varepsilon} \sum_{k=1}^n
          E \left[ \sup_{\theta\in\Theta} \left| g ( X^\varepsilon_{t_{k-1}}, \theta )
          \left\{ \int_{t_{k-1}}^{t_k} a(X^\varepsilon_t,\mu_0) \, dt  - \frac{1}{n} a(X^\varepsilon_{t_{k-1}}, \mu_0) \right\}
          1_{C^{n,\varepsilon,\rho}_{k}} \right|
          \, \middle| \, \mathcal{F}_{t_{k-1}} \right] \\
        & \quad
          \leq C
            \left( \frac{1}{n} \sum_{k=1}^n
            E \left[ \sup_{\theta\in\Theta} \left| g ( X^\varepsilon_{t_{k-1}}, \theta )
            \sup_{t\in[t_{k-1},t_k]} \frac{|X^\varepsilon_t -X^\varepsilon_{t_{k-1}}|}{\varepsilon} \right|^{2}
            \, \middle| \, \mathcal{F}_{t_{k-1}} \right] \right)^{1/2} \\
        & \quad
          = O_p \left( \frac{1}{\varepsilon n} + \frac{1}{\sqrt{n}} + \frac{\lambda_\varepsilon}{n} \right)
      \end{align*}
      as $n\to\infty$, $\varepsilon\to0$, $\lambda_\varepsilon\to\infty$ and $\varepsilon\lambda_\varepsilon\to0$.
      Here, the first inequality holds from Schwartz's inequality and \ref{asmp:A2}, and the last equality holds from Lemmas \ref{lem:4.1} and \ref{lem:4.4}.

      At second, we show that
      \begin{equation}\label{eq:E[g(Xk)bdW1Ck](2)}
        \begin{aligned}
          &\sum_{k=1}^n
            g ( X^\varepsilon_{t_{k-1}}, \theta )
            \int_{t_{k-1}}^{t_k} b(X^\varepsilon_t,\sigma_0) \, dW_t
            1_{C^{n,\varepsilon,\rho}_k} - \int_0^1 g(X^\varepsilon_t,\theta) b(X^\varepsilon_t,\sigma_0) \, dW_t \\
          &\quad
            = \sum_{k=1}^n \int_{t_{k-1}}^{t_k}
              \left\{ g ( X^\varepsilon_{t_{k-1}}, \theta ) - g ( X^\varepsilon_{t}, \theta ) \right\}
              b(X^\varepsilon_t,\sigma_0) \, dW_t
              1_{C^{n,\varepsilon,\rho}_k} \\
          &\quad \quad
            - \sum_{k=1}^n \int_{t_{k-1}}^{t_k}
            g ( X^\varepsilon_{t}, \theta )
            b(X^\varepsilon_t,\sigma_0) \, dW_t
            1_{D^{n,\varepsilon,\rho}_k}
            \overset{p}\longrightarrow
            0
        \end{aligned}
      \end{equation}
      as $n\to\infty$, $\varepsilon\to0$, $\lambda_\varepsilon\to\infty$, $\lambda_\varepsilon^2/n \to 0$, $\varepsilon\lambda_\varepsilon\to0$ and $\lambda_\varepsilon \int_{|z|\leq\kappa/n^\rho} f_{\alpha_0}(z) \, dz\to0$, uniformly in $\theta\in\Theta$.
      When we put
      \begin{equation*}
        \tilde{C}^{n,\varepsilon,\rho}_{k}
        := C^{n,\varepsilon,\rho}_{k}
        \cap \left\{ \sup_{t\in[0,1]} |X^\varepsilon_{t} - x_{t}| < \delta \right\},
      \end{equation*}
      it holds from Morrey's inequality (see, e.g., Theorem 5 in Evans \cite[Section 5.6]{evans2010partial}) that for $q\in(d,\infty)$
      \begin{align*}
        &\sum_{k=1}^n E \left[ \sup_{\theta\in\Theta} \left|
          \int_{t_{k-1}}^{t_k}
          \left\{ g ( X^\varepsilon_{t_{k-1}}, \theta ) - g ( X^\varepsilon_{t}, \theta ) \right\}
          b(X^\varepsilon_t,\sigma_0) \, dW_t
          1_{\tilde{C}^{n,\varepsilon,\rho}_{k}} \right|
          \, \middle| \, \mathcal{F}_{t_{k-1}} \right] \\
        &\quad
          \leq C_1 \sum_{k=1}^n E \left[ \left\|
            \int_{t_{k-1}}^{t_k}
            \left\{ g ( X^\varepsilon_{t_{k-1}}, \theta ) - g ( X^\varepsilon_{t}, \theta ) \right\}
            b(X^\varepsilon_t,\sigma_0) \, dW_t
            1_{\tilde{C}^{n,\varepsilon,\rho}_{k}} \right\|_{W^{1,q}(\Theta)}
            \, \middle| \, \mathcal{F}_{t_{k-1}} \right],
      \end{align*}
      where the constant $C_1$ depends only on $d,q$ and $\Theta$.
      Then, it follows from H\"older's and Burkholder's inequalities, the equi-Lipschitz continuity of $g$ and $b$, and Lemmas \ref{lem:4.1} and \ref{lem:4.4} that\footnote{More precisely, the first inequality holds from H\"older's and Burkholder's inequalities, the second inequality holds from H\"older's inequality and the equi-Lipschitz continuity of $g$ and $b$, and the last equality holds from Lemmas \ref{lem:4.1} and \ref{lem:4.4}.}
      \begin{align*}
        &\sum_{k=1}^n E \left[ \left\|
          \int_{t_{k-1}}^{t_k}
          \left\{ g ( X^\varepsilon_{t_{k-1}}, \theta ) - g ( X^\varepsilon_{t}, \theta ) \right\}
          b(X^\varepsilon_t,\sigma_0) \, dW_t
          1_{\tilde{C}^{n,\varepsilon,\rho}_{k}}
          \right\|_{L^{q}(\Theta)}
          \, \middle| \, \mathcal{F}_{t_{k-1}} \right] \\
        &\quad
        \begin{aligned}
          &\leq C_2 \sum_{k=1}^n \left( \int_\Theta E \left[ \int_{t_{k-1}}^{t_k} \left|
            \left\{ g ( X^\varepsilon_{t_{k-1}}, \theta ) - g ( X^\varepsilon_{t}, \theta ) \right\}
            1_{\tilde{C}^{n,\varepsilon,\rho}_{k}}
            b(X^\varepsilon_t,\sigma_0) \right|^{2} \, dt
            \, \middle | \, \mathcal{F}_{t_{k-1}} \right]^{q/2} d\theta \right)^{1/q}
        \end{aligned} \\
        &\quad
        \begin{aligned}
          &\leq \frac{C_3}{\sqrt{n}} \sum_{k=1}^n \left( E \left[
            \sup_{t\in[t_{k-1},t_k]} | X^\varepsilon_{t_{k-1}} - X^\varepsilon_{t} |^2
            \times \left( 1 + \sup_{t\in[t_{k-1},t_k]} | X^\varepsilon_{t_{k-1}} - X^\varepsilon_{t} |^2 + |X^\varepsilon_{t_{k-1}}|^2 \right)
            \, \middle | \, \mathcal{F}_{t_{k-1}} \right]^{q/2} \right)^{1/q}
        \end{aligned} \\
        &\quad
          =\ O_p \left( \frac{1}{\sqrt{n}} + \varepsilon + \varepsilon \sqrt{\lambda_\varepsilon} \right)
      \end{align*}
      as $n\to\infty$, $\varepsilon\to0$, $\lambda_\varepsilon\to\infty$ and $\varepsilon\lambda_\varepsilon\to0$,
      where $C_2$ depends only on $q$, and $C_3$ depends only on $q,b,g$ and $\Theta$.
      By the same argument with Theorem B.4 in Prakasa Rao \cite{prakasarao1999semimartingale}, it follows that
      \begin{align*}
        &\sum_{k=1}^n E \left[ \left\|
          \int_{t_{k-1}}^{t_k}
          \frac{\partial}{\partial \theta_j} \left\{ g ( X^\varepsilon_{t_{k-1}}, \theta ) - g ( X^\varepsilon_{t}, \theta ) \right\}
          b(X^\varepsilon_t,\sigma_0) \, dW_t
          1_{\tilde{C}^{n,\varepsilon,\rho}_{k}}
          \right\|_{L^{q}(\Theta)}
          \, \middle| \, \mathcal{F}_{t_{k-1}} \right] \\
        &\quad= O_p \left( \frac{1}{\sqrt{n}} + \varepsilon + \varepsilon \sqrt{\lambda_\varepsilon} \right)
      \end{align*}
      as $n\to\infty$, $\varepsilon\to0$, $\lambda_\varepsilon\to\infty$ and $\varepsilon\lambda_\varepsilon\to0$.
      Thus, it follows from Lemma \ref{lem:A.3} that
      \begin{align*}
        \sum_{k=1}^n
          \int_{t_{k-1}}^{t_k}
          \left\{ g ( X^\varepsilon_{t_{k-1}}, \theta ) - g ( X^\varepsilon_{t}, \theta ) \right\}
          b(X^\varepsilon_t,\sigma_0) \, dW_t
          1_{\tilde{C}^{n,\varepsilon,\rho}_{k}}
          = O_p \left( \frac{1}{\sqrt{n}} + \varepsilon + \varepsilon \sqrt{\lambda_\varepsilon} \right)
      \end{align*}
      as $n\to\infty$, $\varepsilon\to0$, $\lambda_\varepsilon\to\infty$ and $\varepsilon\lambda_\varepsilon\to0$,
      uniformly in $\theta\in\Theta$, and therefore, from Lemma \ref{lem:4.3} we obtain the convergence of the first term in the left-hand side of \eqref{eq:E[g(Xk)bdW1Ck](2)}.
      To obtain \eqref{eq:E[g(Xk)bdW1Ck](2)}, we remain to prove
      \begin{equation*}
        \sum_{k=1}^n \int_{t_{k-1}}^{t_k}
          g ( X^\varepsilon_{t}, \theta )
          b(X^\varepsilon_t,\sigma_0) \, dW_t
          1_{D^{n,\varepsilon,\rho}_k}
        \overset{p}\longrightarrow 0
        \label{eq:E[g(Xk)bdW1Ck](3)}
      \end{equation*}
      as $\lambda_\varepsilon \to \infty$, $\varepsilon \lambda_\varepsilon \to 0$, $\frac{\lambda_\varepsilon^2}{n} \to 0$, $\lambda_\varepsilon \int_{|z|\leq\kappa/n^\rho} f_{\alpha_0}(z) \, dz \to 0$,
      uniformly in $\theta\in\Theta$.
      Put
      $\tilde{D}^{n,\varepsilon,\rho}_{k,1}
      := D^{n,\varepsilon,\rho}_{k,1}
        \cap \{ X^\varepsilon_t \in I_{x_0}^\delta \text{ for all } t\in[0,1] \}$.
      We begin with showing that for any $p\in(2,\infty)$ and $q'\in(1,d/(d-1))$
      \begin{align}\label{eq:E[g(Xk)bdW1Ck](4)}
        &\sum_{k=1}^n \int_{t_{k-1}}^{t_k}
          g ( X^\varepsilon_{t}, \theta )
          b(X^\varepsilon_t,\sigma_0) \, dW_t
          1_{\tilde{D}^{n,\varepsilon,\rho}_{k,1}}
        = O_p \left( \frac{1}{\sqrt{n}}
          \left( \frac{\varepsilon^p\lambda_\varepsilon^2}{n} + \lambda_\varepsilon \right)^{1/2+1/q'} \right)
      \end{align}
      as $n\to\infty$, $\varepsilon\to0$, $\lambda_\varepsilon\to\infty$ and $\varepsilon\lambda_\varepsilon\to0$, uniformly in $\theta\in\Theta$.
      It follows from Morrey's inequality (see, e.g., Theorem 5 in Evans \cite[Section 5.6]{evans2010partial}) that for $q\in(d,\infty)$
      \begin{align*}
        &\sum_{k=1}^n E \left[ \sup_{\theta\in\Theta} \left| \int_{t_{k-1}}^{t_k}
          g ( X^\varepsilon_{t}, \theta )
          b(X^\varepsilon_t,\sigma_0) \, dW_t
          1_{\tilde{D}^{n,\varepsilon,\rho}_{k,1}}
          \right| \, \middle | \, \mathcal{F}_{t_{k-1}} \right] \\
        &\qquad
          \leq C_1 \sum_{k=1}^n E \left[ \left\| \int_{t_{k-1}}^{t_k}
          g ( X^\varepsilon_{t}, \theta )
          b(X^\varepsilon_t,\sigma_0) \, dW_t
          \right\|_{W^{1,q}(\Theta)} 1_{\tilde{D}^{n,\varepsilon,\rho}_{k,1}}  \, \middle | \, \mathcal{F}_{t_{k-1}} \right],
      \end{align*}
      where the constant $C_1$ depends only on $d,q$ and $\Theta$.
      If we put $q'=q/(q-1)$, then it follows from H\"older's inequality, Burkholder's inequality (see, e.g., Theorem 4.4.21 in Applebaum \cite{applebaum2009levy}), the equicontinuity of $g$ and Assumption \ref{asmp:A2} that
      \begin{align*}
        &\sum_{k=1}^n E \left[ \left\| \int_{t_{k-1}}^{t_k}
          g ( X^\varepsilon_{t}, \theta )
          b(X^\varepsilon_t,\sigma_0) \, dW_t
          \right\|_{L^{q}(\Theta)} 1_{\tilde{D}^{n,\varepsilon,\rho}_{k,1}}  \, \middle | \, \mathcal{F}_{t_{k-1}} \right] \\
        &\quad
        \begin{aligned}
          &\leq \sum_{k=1}^n \left\{ \left( \int_\Theta E \left[ \left| \int_{t_{k-1}}^{t_k}
            g ( X^\varepsilon_{t}, \theta )
            b(X^\varepsilon_t,\sigma_0) 1_{\tilde{D}^{n,\varepsilon,\rho}_{k,1}} \, dW_t
            \right|^{q} \, \middle | \, \mathcal{F}_{t_{k-1}} \right] d\theta \right)^{1/q}
            P \left( \tilde{D}^{n,\varepsilon,\rho}_{k,1} \, \middle | \, \mathcal{F}_{t_{k-1}} \right)^{1/q'} \right\} \\
          &\leq C_2 \sum_{k=1}^n \left\{ \left( \int_\Theta E \left[ \int_{t_{k-1}}^{t_k} \left|
            g ( X^\varepsilon_{t}, \theta )
            b(X^\varepsilon_t,\sigma_0) 1_{\tilde{D}^{n,\varepsilon,\rho}_{k,1}}
            \right|^{2} dt \, \middle | \, \mathcal{F}_{t_{k-1}} \right]^{q/2} \, d\theta \right)^{1/q}
            P \left( \tilde{D}^{n,\varepsilon,\rho}_{k,1} \, \middle | \, \mathcal{F}_{t_{k-1}} \right)^{1/q'} \right\} \\
          &\leq C_2
            \sup_{(x,\theta)\in I_{x_0}^\delta\times\Theta} |g(x,\theta)b(x,\sigma_0)|
            \frac{|\Theta|^{1/q}}{n^{1/2}} \sum_{k=1}^n
            P \left( \tilde{D}^{n,\varepsilon,\rho}_{k,1} \, \middle | \, \mathcal{F}_{t_{k-1}} \right)^{1/2+1/q'},
        \end{aligned}
      \end{align*}
      where $C_2$ depends only on $q$.
      By using Lemmas \ref{lem:4.4} and \ref{lem:4.7}, for any $p>2$ we obtain
      \begin{align*}
        &\sum_{k=1}^n E \left[ \left\| \int_{t_{k-1}}^{t_k}
          g ( X^\varepsilon_{t}, \theta )
          b(X^\varepsilon_t,\sigma_0) \, dW_t
          \right\|_{L^{q}(\Theta)} 1_{\tilde{D}^{n,\varepsilon,\rho}_{k,1}}  \, \middle | \, \mathcal{F}_{t_{k-1}} \right] \\
        &\qquad
          = O_p \left( \sqrt{n}
          \left\{ \frac{\lambda_\varepsilon}{n} \left( \frac{1}{n^{p(1-\rho)}} + \frac{\varepsilon^p}{n^{p(1/2-\rho)}}
          + \frac{\varepsilon^p\lambda_\varepsilon}{n} \right)
          + \frac{\lambda_\varepsilon}{n} \right\}^{1/2+1/q'} \right)
      \end{align*}
      as $n\to\infty$, $\varepsilon\to0$, $\lambda_\varepsilon\to\infty$ and $\varepsilon\lambda_\varepsilon\to0$.
      Similarly, by using Theorem B.4 in Prakasa Rao \cite{prakasarao1999semimartingale}, we obtain for $j=1,\dots,d$
      \begin{align*}
        &\sum_{k=1}^n E \left[ \left\| \int_{t_{k-1}}^{t_k}
          \frac{\partial g}{\partial \theta_j} ( X^\varepsilon_{t}, \theta )
          b(X^\varepsilon_t,\sigma_0) \, dW_t
          \right\|_{L^{q}(\Theta)} 1_{\tilde{D}^{n,\varepsilon,\rho}_{k,1}}  \, \middle | \, \mathcal{F}_{t_{k-1}} \right] \\
        &\qquad
        = O_p \left( \sqrt{n}
          \left\{ \frac{\lambda_\varepsilon}{n} \left( \frac{1}{n^{p(1-\rho)}} + \frac{\varepsilon^p}{n^{p(1/2-\rho)}}
          + \frac{\varepsilon^p\lambda_\varepsilon}{n} \right)
          + \frac{\lambda_\varepsilon}{n} \right\}^{1/2+1/q'} \right)
      \end{align*}
      as $n\to\infty$, $\varepsilon\to0$, $\lambda_\varepsilon\to\infty$ and $\varepsilon\lambda_\varepsilon\to0$.
      Since we can take $q'<2$ small enough, we obtain \eqref{eq:E[g(Xk)bdW1Ck](4)} from Remark \ref{rmk:ConseqFromLenglart2}.
      Hence, \eqref{eq:E[g(Xk)bdW1Ck](3)} holds from \eqref{eq:E[g(Xk)bdW1Ck](4)} and Lemma \ref{lem:4.10}.

      At last, it is an immediate consequence from Lemma \ref{lem:4.9} that
      \begin{equation*}
        \sum_{k=1}^n
          g ( X^\varepsilon_{t_{k-1}}, \theta )
          \int_{t_{k-1}}^{t_k} c(X^\varepsilon_{t-},\alpha_0) \, dZ^{\lambda_\varepsilon}_t
          1_{C^{n,\varepsilon,\rho}_k}
          \overset{p}\longrightarrow
          0
      \end{equation*}
      as $n\to\infty$, $\varepsilon\to0$, $\lambda_\varepsilon\to\infty$, $\lambda_\varepsilon^2/n \to 0$, $\varepsilon\lambda_\varepsilon\to0$ and $\lambda_\varepsilon \int_{|z|\leq\kappa/n^\rho} f_{\alpha_0}(z) \, dz\to0$, uniformly in $\theta\in\Theta$.
    \end{proof}

    \begin{lemma}\label{lem:4.15}
      Under Assumptions \ref{asmp:A2} to \ref{asmp:A6},
      let $\rho\in(0,1/2)$,
      and suppose that a function $g:\mathbb{R}\times\Theta\to\mathbb{R}$
      satisfies
      that $\{\frac{\partial g}{\partial \theta_j}(\cdot,\theta)\}_{\theta\in\Theta}$ ($j=1,\dots,d$) are equicontinuous on $I_{x_0}^\delta$ for some small $\delta>0$.
      Then,
      \begin{equation*}
        \frac{1}{\varepsilon^2} \sum_{k=1}^n
          g ( X^\varepsilon_{t_{k-1}}, \theta )
          \left| \Delta^n_k X^{\varepsilon} - \frac{1}{n} a(X^\varepsilon_{t_{k-1}}, \mu_0) \right|^2
          1_{C^{n,\varepsilon,\rho}_k}
        \overset{p}\longrightarrow
        \int_0^1 g(x_t, \theta) |b(x_t,\sigma_0)|^2 \, dt
      \end{equation*}
      as $n\to\infty$, $\varepsilon\to0$, $\lambda_\varepsilon\to\infty$,
      $\varepsilon\lambda_\varepsilon\to0$, $\lambda_\varepsilon^2/n\to0$ and $\lambda_\varepsilon \int_{|z|\leq\kappa/n^\rho} f_{\alpha_0}(z) \, dz\to0$, uniformly in $\theta\in\Theta$.
    \end{lemma}

    \begin{proof}
      From Lemma \ref{lem:4.9}, it is sufficient to show that
      \begin{equation*}
        \frac{1}{\varepsilon^2} \sum_{k=1}^n
          g ( X^\varepsilon_{t_{k-1}}, \theta )
          \left| \Delta^n_k X^{\varepsilon} - \frac{1}{n} a(X^\varepsilon_{t_{k-1}}, \mu_0) \right|^2
          1_{C^{n,\varepsilon,\rho}_{k,0}}
        \overset{p}\longrightarrow
        \int_0^1 g(x_t, \theta) |b(x_t,\sigma_0)|^2 \, dt
      \end{equation*}
      as $n\to\infty$, $\varepsilon\to0$, $\lambda_\varepsilon\to\infty$,
      $\varepsilon\lambda_\varepsilon\to0$, $\lambda_\varepsilon^2/n\to0$ and $\lambda_\varepsilon \int_{|z|\leq\kappa/n^\rho} f_{\alpha_0}(z) \, dz\to0$, uniformly in $\theta\in\Theta$,
      and we note that
      \begin{align*}
        \left| \Delta^n_k X^{\varepsilon} - \frac{1}{n} a(X^\varepsilon_{t_{k-1}}, \mu_0) \right|^2
        1_{J^{n,\varepsilon}_{k,0}}
        &= \left\{ \left| \int_{t_{k-1}}^{t_k} \left\{ a(X^\varepsilon_t,\mu_0) - a(X^\varepsilon_{t_{k-1}}, \mu_0) \right\} dt \right|^2
        + \left| \int_{t_{k-1}}^{t_k} b(X^\varepsilon_t,\sigma_0) \, dW_t \right|^2 \right. \\
        &\quad \left. + 2 \int_{t_{k-1}}^{t_k} \left\{ a(X^\varepsilon_t,\mu_0) - a(X^\varepsilon_{t_{k-1}}, \mu_0) \right\} dt
        \int_{t_{k-1}}^{t_k} b(X^\varepsilon_t,\sigma_0) \, dW_t \right\} 1_{J^{n,\varepsilon}_{k,0}}.
      \end{align*}
      Similarly to the proof of \eqref{eq:E[g(Xk)bdW1Ck](1)}, it follows that
      \begin{align*}
        \sup_{\theta\in\Theta} \left| \frac{1}{\varepsilon^2} \sum_{k=1}^n
          g ( X^\varepsilon_{t_{k-1}}, \theta )
          \left| \int_{t_{k-1}}^{t_k} \left\{ a(X^\varepsilon_t,\mu_0) - a(X^\varepsilon_{t_{k-1}}, \mu_0) \right\} dt \right|^2 1_{C^{n,\varepsilon,\rho}_{k,0}} \right|
        = O_p \left( \frac{1}{\varepsilon^2 n^3} + \frac{1}{n^2} \right)
      \end{align*}
      as $n\to\infty$, $\varepsilon\to0$, $\lambda_\varepsilon\to\infty$, $\varepsilon n\to\infty$, and $\varepsilon\lambda_\varepsilon\to0$.
      Also, it holds that
      \begin{align*}
        \sup_{\theta\in\Theta} \left| \frac{2}{\varepsilon} \sum_{k=1}^n
          g ( X^\varepsilon_{t_{k-1}}, \theta )
          \int_{t_{k-1}}^{t_k} \left\{ a(X^\varepsilon_t,\mu_0) - a(X^\varepsilon_{t_{k-1}}, \mu_0) \right\} dt
          \int_{t_{k-1}}^{t_k} b(X^\varepsilon_t,\sigma_0) \, dW_t
          \, 1_{C^{n,\varepsilon,\rho}_{k,0}} \right|
          = O_p \left( \frac{1}{\varepsilon n^{3/2}} + \frac{1}{n} \right)
      \end{align*}
      as $n\to\infty$, $\varepsilon\to0$, $\lambda_\varepsilon\to\infty$ and $\varepsilon\lambda_\varepsilon\to0$, uniformly in $\theta\in\Theta$.
      Indeed, by using \ref{asmp:A2}, H\"older's inequality and Burkholder's inequality, we obtain
      \begin{align*}
        &\frac{2}{\varepsilon} \sum_{k=1}^n
          E \left[ \sup_{\theta\in\Theta} \left| g ( X^\varepsilon_{t_{k-1}}, \theta )
          \int_{t_{k-1}}^{t_k} \left\{ a(X^\varepsilon_t,\mu_0) - a(X^\varepsilon_{t_{k-1}}, \mu_0) \right\} dt
          \int_{t_{k-1}}^{t_k} b(X^\varepsilon_t,\sigma_0) \, dW_t
          \, 1_{C^{n,\varepsilon,\rho}_{k,0}} \right|
          \, \middle | \, \mathcal{F}_{t_{k-1}} \right] \\
        &\quad
          \leq \frac{2C}{n} \sum_{k=1}^n \sup_{\theta\in\Theta} |g ( X^\varepsilon_{t_{k-1}}, \theta )|
          \left( E \left[
            \frac{1}{\varepsilon^2} \sup_{t\in[t_{k-1},\tau_k]} |X^\varepsilon_t - X^\varepsilon_{t_{k-1}}|^2
            % \, 1_{C^{n,\varepsilon,\rho}_{k,0}}
            \, \middle | \, \mathcal{F}_{t_{k-1}} \right] \right)^{1/2} \\
        &\quad\quad
          \times \left( \frac{1}{n} E \left[
            \sup_{t\in[t_{k-1},\tau_k]} |X^\varepsilon_t - X^\varepsilon_{t_{k-1}}|^2 + |b(X^\varepsilon_{t_{k-1}},\sigma_0)|^2
            \, \middle | \, \mathcal{F}_{t_{k-1}} \right]\right)^{1/2}
      \end{align*}
      where $C$ depends only on $a,b$.
      By applying Lemmas \ref{lem:4.3} to \ref{lem:4.5} and \ref{lem:A.3}
      and the boundedness of $g$ on $I_{x_0}^\delta\times\Theta$ for some small $\delta>0$,
      we obtain the above convergence.

      From Lemma \ref{lem:4.11},
      we remain to prove that
      \begin{equation*}
        \sum_{k=1}^n
          g ( X^\varepsilon_{t_{k-1}}, \theta )
          \left| \int_{t_{k-1}}^{t_k} b(X^\varepsilon_t,\sigma_0) \, dW_t \right|^2
          1_{C^{n,\varepsilon,\rho}_k}
        \overset{p}\longrightarrow
        \int_0^1 g(x_t, \theta) |b(x_t,\sigma_0)|^2 \, dt
      \end{equation*}
      as $n\to\infty$, $\varepsilon\to0$, $\lambda_\varepsilon\to\infty$ and $\varepsilon\lambda_\varepsilon\to0$, uniformly in $\theta\in\Theta$.
      At first, by using Lemma \ref{lem:4.4}, we have
      \begin{align*}
        \sum_{k=1}^n
          E \left[ g ( X^\varepsilon_{t_{k-1}}, \theta )
          \left| \int_{t_{k-1}}^{t_k} b(X^\varepsilon_{t_{k-1}},\sigma_0) \, dW_t \right|^2
          \, \middle | \, \mathcal{F}_{t_{k-1}} \right]
        & = \frac{1}{n} \sum_{k=1}^n
          g ( X^\varepsilon_{t_{k-1}}, \theta ) | b(X^\varepsilon_{t_{k-1}},\sigma_0) |^2 \\
        & \overset{p}\longrightarrow
          \int_0^1 g(x_t, \theta) |b(x_t,\sigma_0)|^2 \, dt
      \end{align*}
      as $n\to\infty$, $\varepsilon\to0$, $\lambda_\varepsilon\to\infty$ and $\varepsilon\lambda_\varepsilon\to0$,
      and
      \begin{equation*}
        \sum_{k=1}^n
          E \left[ \left| g ( X^\varepsilon_{t_{k-1}}, \theta )
          \left| \int_{t_{k-1}}^{t_k} b(X^\varepsilon_{t_{k-1}},\sigma_0) \, dW_t \right|^2 \right|^2
          \, \middle | \, \mathcal{F}_{t_{k-1}} \right]
        \overset{p}\longrightarrow 0
      \end{equation*}
      as $n\to\infty$, $\varepsilon\to0$, $\lambda_\varepsilon\to\infty$ and $\varepsilon\lambda_\varepsilon\to0$.
      Thus, by Lemma 9 in Genon-Catalot and Jocod \cite{genon-catalot1993ontheest}, we obtain
      \begin{equation*}
        \sum_{k=1}^n
          g ( X^\varepsilon_{t_{k-1}}, \theta )
          \left| \int_{t_{k-1}}^{t_k} b(X^\varepsilon_{t_{k-1}},\sigma_0) \, dW_t \right|^2
        \overset{p}\longrightarrow
          \int_0^1 g(x_t, \theta) |b(x_t,\theta)|^2 \, dt
      \end{equation*}
      as $n\to\infty$, $\varepsilon\to0$, $\lambda_\varepsilon\to\infty$ and $\varepsilon\lambda_\varepsilon\to0$.
      From the equidifferentiablities of $g$ on $I_{x_0}^\delta$ for some $\delta>0$,
      the uniform tightness is shown by the same argument in the proof of Lemma \ref{lem:4.13}.
      At second, we shall see
      \begin{equation*}
        \sum_{k=1}^n
          g ( X^\varepsilon_{t_{k-1}}, \theta )
          \left\{ \left| \int_{t_{k-1}}^{t_k} b(X^\varepsilon_t,\sigma_0) \, dW_t \right|^2
          % - \frac{1}{n} |b(X^\varepsilon_{t_{k-1}},\sigma_0)|^2
          - \left| \int_{t_{k-1}}^{t_k} b(X^\varepsilon_{t_{k-1}},\sigma_0) \, dW_t \right|^2 \right\}
          1_{C^{n,\varepsilon,\rho}_{k,0}}
        \overset{p}\longrightarrow 0
      \end{equation*}
      as $n\to\infty$, $\varepsilon\to0$, $\lambda_\varepsilon\to\infty$, $\varepsilon\lambda_\varepsilon\to0$ and $\lambda_\varepsilon/n\to0$, uniformly in $\theta\in\Theta$.
      This convergence is obtained from Lemma \ref{lem:A.3} and the following estimate:\footnote{The first inequality holds from H\"older's inequality, the second inequality holds from Burkholder's inequality, the third inequality holds from Lipschitz continuity of $b$, and the last equality holds from Lemmas \ref{lem:4.4} and \ref{lem:4.5}}
      \begin{align*}
        &\sum_{k=1}^n
          E \left[ \sup_{\theta\in\Theta} \left| g ( X^\varepsilon_{t_{k-1}}, \theta )
          \left\{ \left| \int_{t_{k-1}}^{t_k} b(X^\varepsilon_t,\sigma_0) \, dW_t \right|^2
          - \left| \int_{t_{k-1}}^{t_k} b(X^\varepsilon_{t_{k-1}},\sigma_0) \, dW_t \right|^2 \right\}
          1_{J^{n,\varepsilon}_{k,0}} \right|
          \, \middle | \, \mathcal{F}_{t_{k-1}} \right] \\
        & \quad
          \leq \sum_{k=1}^n \left\{
            \sup_{\theta\in\Theta} | g ( X^\varepsilon_{t_{k-1}}, \theta )|
          \left( E \left[ \left|
            \int_{t_{k-1}}^{t_k} \{ b(X^\varepsilon_t,\sigma_0) + b(X^\varepsilon_{t_{k-1}},\sigma_0) \} \, dW_t
            \right|^2 1_{J^{n,\varepsilon}_{k,0}}
            \, \middle | \, \mathcal{F}_{t_{k-1}} \right] \right)^{1/2} \right.\\
        & \quad \quad \left.
          \times \left( E \left[ \left|
          \int_{t_{k-1}}^{t_k} \{ b(X^\varepsilon_t,\sigma_0) - b(X^\varepsilon_{t_{k-1}},\sigma_0) \} \, dW_t
          \right|^2 1_{J^{n,\varepsilon}_{k,0}}
          \, \middle | \, \mathcal{F}_{t_{k-1}} \right] \right)^{1/2} \right\} \\
        & \quad
          \leq \sum_{k=1}^n \left\{
            \sup_{\theta\in\Theta} | g ( X^\varepsilon_{t_{k-1}}, \theta )|
            \left( E \left[
            \int_{t_{k-1}}^{t_k} | b(X^\varepsilon_t,\sigma_0) + b(X^\varepsilon_{t_{k-1}},\sigma_0) |^2
            1_{J^{n,\varepsilon}_{k,0}} \, dt
            \, \middle | \, \mathcal{F}_{t_{k-1}} \right] \right)^{1/2} \right.\\
        & \quad \quad \left.
          \times \left( E \left[
          \int_{t_{k-1}}^{t_k} | b(X^\varepsilon_t,\sigma_0) - b(X^\varepsilon_{t_{k-1}},\sigma_0) |^2 1_{J^{n,\varepsilon}_{k,0}} \, dt
          \, \middle | \, \mathcal{F}_{t_{k-1}} \right] \right)^{1/2} \right\} \\
        & \quad \leq \frac{C}{n} \sum_{k=1}^n \left\{
          \sup_{\theta\in\Theta} | g ( X^\varepsilon_{t_{k-1}}, \theta )|
          \left( E \left[
          \sup_{t\in[t_{k-1},\tau_k]} ( 1 + |X^\varepsilon_{t} - X^\varepsilon_{t_{k-1}}|^2 + |X^\varepsilon_{t_{k-1}}|^2)
          \, \middle | \, \mathcal{F}_{t_{k-1}} \right] \right)^{1/2} \right. \\
        & \quad \quad \left.
          \times \left( E \left[
          \sup_{t\in[t_{k-1},\tau_k]} |X^\varepsilon_{t} - X^\varepsilon_{t_{k-1}}|^2
          \, \middle | \, \mathcal{F}_{t_{k-1}} \right] \right)^{1/2} \right\} \\
        & \quad
          = O_p \left( \frac{1}{n} + \frac{\varepsilon}{\sqrt{n}} \right)
      \end{align*}
      as $n\to\infty$, $\varepsilon\to0$, $\lambda_\varepsilon\to\infty$ and $\varepsilon\lambda_\varepsilon\to0$.

      At last, since
      \begin{equation*}
        \sup_{k=1,\dots,n} n \left| \int_{t_{k-1}}^{t_k} \, dW_t \right|^2
      \end{equation*}
      is bounded in probability,
      it follows from Lemmas \ref{lem:4.1}, \ref{lem:4.8} and \ref{lem:4.9} and the linarity of $b$ that
      \begin{equation*}
        \sum_{k=1}^n
          g ( X^\varepsilon_{t_{k-1}}, \theta )
          \left| \int_{t_{k-1}}^{t_k} b(X^\varepsilon_{t_{k-1}},\sigma_0) \, dW_t \right|^2
          1_{D^{n,\varepsilon,\rho}_k\cup C^{n,\varepsilon,\rho}_{k,1} \cup C^{n,\varepsilon,\rho}_{k,2}}
        \overset{p}\longrightarrow 0
      \end{equation*}
      as $n\to\infty$, $\varepsilon\to0$, $\lambda_\varepsilon\to\infty$, $\varepsilon\lambda_\varepsilon\to0$ and $\lambda_\varepsilon/n\to0$, uniformly in $\theta\in\Theta$.
    \end{proof}

  \subsection{Proof of main results}

    \subsubsection{Proof of Theorem \ref{thm:3.1}}

      \begin{proof}[Proof of Theorem \ref{thm:3.1}]
        It follows from  that
        \begin{align*}
          \Phi_{n,\varepsilon}^{(1)}(\mu,\sigma)
          &= \sum_{k=1}^{n} \frac{ \left ( \Delta^n_k X^{\varepsilon} - a(X^\varepsilon_{t_{k-1}}, \mu_0) / n \right ) \left ( a(X^\varepsilon_{t_{k-1}}, \mu) - a(X^\varepsilon_{t_{k-1}}, \mu_0) \right )}{\left | b(X^\varepsilon_{t_{k-1}},\sigma) \right |^2}
          1_{C^{n,\varepsilon,\rho}_k}\\
          &\quad - \frac{1}{2n} \sum_{k=1}^{n} \frac{ \left | a(X^\varepsilon_{t_{k-1}}, \mu) - a(X^\varepsilon_{t_{k-1}}, \mu_0) \right |^2}{\left | b(X^\varepsilon_{t_{k-1}},\sigma) \right |^2}
          1_{C^{n,\varepsilon,\rho}_k}
          \overset{p}\longrightarrow
          - \frac{1}{2} \int_0^1 \frac{ \left | a(x_t, \mu) - a(x_t, \mu_0) \right |^2}{\left | b(x_{t},\sigma) \right |^2} \, dt
        \end{align*}
        as $n\to\infty$, $\varepsilon\to0$, $\lambda_\varepsilon\to\infty$, $\lambda_\varepsilon^2/n \to 0$, $\varepsilon\lambda_\varepsilon\to0$ and $\lambda_\varepsilon \int_{|z|\leq\kappa/n^\rho} f_{\alpha_0}(z) \, dz\to0$, uniformly in $(\mu,\sigma)\in\bar{\Theta}_1\times\bar{\Theta}_2$,
        and from Lemmas \ref{lem:4.11}, \ref{lem:4.14} and \ref{lem:4.15} that
        \begin{align*}
          \Psi_{n,\varepsilon}^{(1)}(\mu,\sigma)
            &= \frac{1}{\varepsilon^2 n} \Phi_{n,\varepsilon}^{(1)}(\mu,\sigma) + \Psi_{n,\varepsilon}^{(1)}(\mu_0,\sigma) \\
            &= \frac{1}{\varepsilon^2 n} \Phi_{n,\varepsilon}^{(1)}(\mu,\sigma)
             - \frac{1}{n} \sum_{k=1}^{n} \left \{ \frac{ \left | \Delta^n_k X^{\varepsilon} - a(X^\varepsilon_{t_{k-1}}, \mu_0) / n \right |^2 }{2\left | \varepsilon b(X^\varepsilon_{t_{k-1}},\sigma) \right |^2 / n}
            + \frac{1}{2} \log |b(X^\varepsilon_{t_{k-1}},\sigma)|^2 \right \}
            1_{C^{n,\varepsilon,\rho}_k} \\
            & \overset{p}\longrightarrow
              - \left ( \lim_{\substack{n\to\infty \\ \varepsilon\to0}} \frac{1}{\varepsilon^2 n} \right ) \int_0^1 \frac{ \left | a(x_t, \mu) - a(x_t, \mu_0) \right |^2}{2\left | b(x_{t},\sigma) \right |^2} \, dt
              - \frac{1}{2} \int_0^1 \left | \frac{b(x_t,\sigma_0)}{b(x_t,\sigma)} \right |^2 \, dt
              - \frac{1}{2} \int_0^1 \log |b(x_t,\sigma)|^2 \, dt
        \end{align*}
        as $n\to\infty$, $\varepsilon\to0$, $\lambda_\varepsilon\to\infty$, $\lambda_\varepsilon^2/n \to 0$, $\varepsilon\lambda_\varepsilon\to0$ and $\lambda_\varepsilon \int_{|z|\leq\kappa/n^\rho} f_{\alpha_0}(z) \, dz\to0$, uniformly in $(\mu,\sigma)\in\bar{\Theta}_1\times\bar{\Theta}_2$,
        Also, it follows from Lemmas \ref{lem:4.12} and \ref{lem:4.13} that
        \begin{align*}
          \Psi_{n,\varepsilon}^{(2)}(\alpha)
          &= \frac{1}{\lambda_\varepsilon} \sum_{k=1}^{n}
            \psi \left( X^\varepsilon_{t_{k-1}}, \frac{\Delta^n_k X^{\varepsilon}}\varepsilon, \alpha \right)
            1_{D^{n,\varepsilon,\rho}_k} \\
          &\overset{p}\longrightarrow
          \int_0^1 \int_{-\infty}^\infty \frac{1}{c(x_t,\alpha_0)} f_{\alpha_0}\left ( \frac{y}{c(x_t,\alpha_0)} \right ) \log \left \{ \frac{1}{c(x_t,\alpha)} f_{\alpha}\left ( \frac{y}{c(x_t,\alpha)} \right ) \right \} dy \, dt
        \end{align*}
        as $n\to\infty$, $\varepsilon\to0$, $\lambda_\varepsilon\to\infty$, $\lambda_\varepsilon^2/n \to 0$, $\varepsilon\lambda_\varepsilon\to0$ and $\lambda_\varepsilon \int_{|z|\leq\kappa/n^\rho} f_{\alpha_0}(z) \, dz\to0$, uniformly in $\alpha\in\bar{\Theta}_3$.
        Thus, by using usual argument (see, e.g., the proof of Theorem 1 in S{\o}rensen and Uchida \cite{sorensen2003small}),
        the consistency of $\hat{\theta}_{n,\varepsilon}$ holds under Assumption \ref{asmp:A1}.
      \end{proof}

    \subsubsection{Proof of Theorem \ref{thm:3.2}}

    To establish the proof of Theorem \ref{thm:3.2},
    % Same as in the proof of Lemma 6 in S{\o}rensen and Uchida \cite{sorensen2003small},
    we set up random variables $\xi^i_{\ell k}$, $\tilde\xi^i_{\ell k}$ ($\ell=1,\dots,3$, $i=1,\dots,d_\ell$, $k=1,\dots,n$) as the followings:
    \begin{align*}
      \left. \sqrt{\varepsilon^{-2}} \frac{\partial}{\partial \mu_i}
      \Phi_{n,\varepsilon}^{(1)}(\mu,\sigma) \right|_{\theta=\theta_0}
      &= - \left. \frac1\varepsilon \sum_{k=1}^{n} \frac{ \left ( \Delta^n_k X^{\varepsilon} - a(X^\varepsilon_{t_{k-1}}, \mu) / n \right ) \frac{\partial a}{\partial \mu_i} (X^\varepsilon_{t_{k-1}}, \mu) }{\left | b(X^\varepsilon_{t_{k-1}},\sigma) \right |^2}
      1_{C^{n,\varepsilon,\rho}_k} \right|_{\theta=\theta_0}
      =: \sum_{k=1}^n \xi^i_{1,k} \\
      \sqrt{n} \left. \frac{\partial}{\partial \sigma_i}
      \Psi_{n,\varepsilon}^{(1)}(\mu,\sigma) \right|_{\theta=\theta_0}
      &= - \frac{1}{\sqrt{n}} \sum_{k=1}^{n}
        \left\{ \left ( - \frac{ \left | \Delta^n_k X^{\varepsilon} - a(X^\varepsilon_{t_{k-1}}, \mu) / n \right |^2 }{\left | \varepsilon b(X^\varepsilon_{t_{k-1}},\sigma) \right |^2 / n}
        + 1 \right )
        \left.
        \frac{\frac{\partial b}{\partial \sigma_i}(X^\varepsilon_{t_{k-1}},\sigma)}{b(X^\varepsilon_{t_{k-1}},\sigma)}
        1_{C^{n,\varepsilon,\rho}_k} \right|_{\theta=\theta_0} \right\} \\
      &=: \sum_{k=1}^n \xi^i_{2,k}, \\
      \left. \sqrt{\lambda_\varepsilon} \frac{\partial}{\partial \alpha_i}
      \Psi_{n,\varepsilon}^{(2)}(\alpha) \right|_{\alpha=\alpha_0}
      &= \frac{1}{\sqrt{\lambda_\varepsilon}}
        \sum_{k=1}^{n} \frac{\partial\psi}{\partial\alpha_i} \left( X^\varepsilon_{t_{k-1}}, \frac{\Delta^n_k X^{\varepsilon}}\varepsilon, \alpha_0 \right)
        1_{D^{n,\varepsilon,\rho}_k}
      =: \sum_{k=1}^n \xi^i_{3,k},
    \end{align*}
    and
    \begin{align*}
      \sum_{k=1}^n \tilde\xi^i_{1,k}
      &:= \sum_{k=1}^{n} \frac{ \frac{\partial a}{\partial \mu_i} (X^\varepsilon_{t_{k-1}}, \mu_0) }{b(X^\varepsilon_{t_{k-1}},\sigma_0)} \int_{t_{k-1}}^{t_k} dW_t
      1_{C^{n,\varepsilon,\rho}_{k,0}}, \\
      \sum_{k=1}^n \tilde\xi^i_{2,k}
      &:= - \sqrt{n} \sum_{k=1}^{n}
        \left \{ - \left | \int_{t_{k-1}}^{t_k} dW_t \right |^2 + \frac{1}{n}
        \right \} \frac{\frac{\partial b}{\partial \sigma_i}(X^\varepsilon_{t_{k-1}},\sigma_0)}{b(X^\varepsilon_{t_{k-1}},\sigma_0)}
        1_{C^{n,\varepsilon,\rho}_{k,0}}, \\
      \sum_{k=1}^n \tilde{\xi}^i_{3,k}
      &:= \sum_{k=1}^{n} \frac{1}{\sqrt{\lambda_\varepsilon}}
        \frac{\partial\psi}{\partial\alpha_i} \left( X^\varepsilon_{t_{k-1}}, c(X^\varepsilon_{t_{k-1}},\alpha_0) V_{N^{\lambda_\varepsilon}_{\tau_k}}, \alpha_0 \right)
        1_{J^{n,\varepsilon}_{k,1}}.
    \end{align*}

    \begin{remark}
      Under Assumptions \ref{asmp:A2} to \ref{asmp:A8},
      let $\rho\in(0,1/2)$,
      it follows from Lemma \ref{lem:4.14} that
      \begin{equation*}
        \sum_{k=1}^n \xi^i_{1,k}
        \overset{p}\longrightarrow \int_0^1 \frac{ \frac{\partial a}{\partial \mu_i} (x_{t},\mu_0)}{b(x_{t},\sigma_0)} \, dW_t
      \end{equation*}
      as $n\to\infty$, $\varepsilon\to0$, $\lambda_\varepsilon\to\infty$, $\lambda_\varepsilon^2/n \to 0$, $\varepsilon\lambda_\varepsilon\to0$ and $\lambda_\varepsilon \int_{|z|\leq\kappa/n^\rho} f_{\alpha_0}(z) \, dz\to0$, while we will never use this fact in this paper.
    \end{remark}

    \begin{lemma}\label{lem:4.16}
      Under Assumptions \ref{asmp:A2} to \ref{asmp:A9},
      for $\ell=1,2$,
      \begin{equation*}
        \sum_{k=1}^n \xi^i_{\ell k}
        - \sum_{k=1}^{n} \tilde\xi^i_{\ell k}
        \overset{p}\longrightarrow 0
        \quad (i=1,\dots,d_\ell)
      \end{equation*}
      as $n\to\infty$, $\varepsilon\to0$, $\lambda_\varepsilon\to\infty$, $\varepsilon n\to\infty$, $\varepsilon\lambda_\varepsilon\to0$, $\lambda_\varepsilon^2/n\to0$ and
      $\lambda_\varepsilon \int_{|z|\leq\kappa/n^\rho} f_{\alpha_0}(z) \, dz\to0$.

      For $\ell=3$, take $\rho$ as either of the following:
      \begin{enumerate}
        \item[(i)]  Under Assumption \ref{asmp:A4} \ref{asmp:A4(i)}, take $\rho\in(0,1/2)$.
        \item[(ii)] Under Assumption \ref{asmp:A4} \ref{asmp:A4(ii)},
                    take $\rho\in(0,\min\{1/2,1/4q\})$, where $q$ is the constant in Assumption \ref{asmp:A9} \ref{asmp:A9(iib)}.
      \end{enumerate}
      Then,
      \begin{equation*}
        \sum_{k=1}^n \xi^i_{\ell k}
        - \sum_{k=1}^{n} \tilde\xi^i_{\ell k}
        \overset{p}\longrightarrow 0
        \quad (i=1,\dots,d_\ell)
      \end{equation*}
      as $n\to\infty$, $\varepsilon\to0$, $\lambda_\varepsilon\to\infty$, $\varepsilon\lambda_\varepsilon\to0$, $\lambda_\varepsilon^2/n\to0$ and
      $\lambda_\varepsilon \int_{|z|\leq\kappa/n^\rho} f_{\alpha_0}(z) \, dz\to0$ with $\lim (\varepsilon^2 n)^{-1}<\infty$.
    \end{lemma}

    \begin{proof}
      For $\ell=1,2$, from Lemmas \ref{lem:4.9} and \ref{lem:A.3}, it is suffcient to show that for $\rho\in(0,1/2)$
      \begin{align*}
        &\sum_{k=1}^n E \left[ \left| \xi^i_{\ell k} 1_{J^{n,\varepsilon}_{k,0}} - \tilde\xi^i_{\ell k} \right|
        \, \middle| \, \mathcal{F}_{t_{k-1}} \right]
        \overset{p}\longrightarrow 0
      \end{align*}
      as $n\to\infty$, $\varepsilon\to0$, $\lambda_\varepsilon\to\infty$, $\varepsilon n\to\infty$, $\varepsilon\lambda_\varepsilon\to0$, $\lambda_\varepsilon^2/n\to0$ and
      $\lambda_\varepsilon \int_{|z|\leq\kappa/n^\rho} f_{\alpha_0}(z) \, dz\to0$.

      For $\ell=1$, let $i\in\{1,\dots,d_1\}$, and put $g(x)=\frac{\partial a}{\partial \mu_i}(x,\mu_0)/|b(x,\sigma_0)|^2$.
      Then,
      \begin{align*}
        &\xi^i_{1, k} 1_{J^{n,\varepsilon}_{k,0}} - \tilde\xi^i_{1, k} \\
        &= g(X^\varepsilon_{t_{k-1}}) \left\{ \int_{t_{k-1}}^{t_k} a(X^\varepsilon_t,\mu_0) \, dt  - \frac{1}{n} a(X^\varepsilon_{t_{k-1}}, \mu_0)
        + \, \varepsilon \int_{t_{k-1}}^{t_k} \{ b(X^\varepsilon_{t},\sigma_0) - b(X^\varepsilon_{t_{k-1}},\sigma_0) \} \, dW_t \right\}
        1_{C^{n,\varepsilon,\rho}_{k,0}}.
      \end{align*}
      As in the same argument in Lemma \ref{lem:4.14},
      it holds from Assumption \ref{asmp:A2}, \ref{asmp:A3} and \ref{asmp:A8} and Lemmas \ref{lem:4.4} and \ref{lem:4.5} that
      \begin{align*}
        &\frac{1}{\varepsilon} \sum_{k=1}^n
          E \left[ |g ( X^\varepsilon_{t_{k-1}} )|
          \left| \int_{t_{k-1}}^{t_k} a(X^\varepsilon_t,\mu_0) \, dt  - \frac{1}{n} a(X^\varepsilon_{t_{k-1}}, \mu_0) \right|
          1_{C^{n,\varepsilon,\rho}_{k,0}}
          \, \middle| \, \mathcal{F}_{t_{k-1}} \right]
          = O_p \left( \frac{1}{\varepsilon n} + \frac{1}{\sqrt{n}} \right)
      \end{align*}
      as $n\to\infty$, $\varepsilon\to0$, $\lambda_\varepsilon\to\infty$ and $\varepsilon\lambda_\varepsilon\to0$,
      and from \ref{asmp:A2}, Burkholder's inequality, Lemmas \ref{lem:4.4} and \ref{lem:4.5} that
      \begin{align*}
        &\sum_{k=1}^n
          E \left[ |g ( X^\varepsilon_{t_{k-1}} )| \left|
          \int_{t_{k-1}}^{t_k} \{ b(X^\varepsilon_{t},\sigma_0) - b(X^\varepsilon_{t_{k-1}},\sigma_0) \} \, dW_t \right|
          1_{C^{n,\varepsilon,\rho}_{k,0}}
          \, \middle| \, \mathcal{F}_{t_{k-1}} \right] \\
        & \quad
          \leq \frac{C}{\sqrt{n}} \sum_{k=1}^n
          |g ( X^\varepsilon_{t_{k-1}} )|
          \left( E \left[ \sup_{t\in[t_{k-1},\tau_k]}
          |X^\varepsilon_{t} - X^\varepsilon_{t_{k-1}} |^2
          \, \middle| \, \mathcal{F}_{t_{k-1}} \right] \right)^{1/2}
        = O_p \left( \frac{1}{\sqrt{n}} + \varepsilon \right)
      \end{align*}
      as $n\to\infty$, $\varepsilon\to0$, $\lambda_\varepsilon\to\infty$ and $\varepsilon\lambda_\varepsilon\to0$.

      For $\ell=2$, let $i\in\{1,\dots,d_2\}$, and put $g(x)=-\frac{\partial b}{\partial \sigma_i}(x,\sigma_0)/|b(x,\sigma_0)|^3$.
      Then, we have
      \begin{align*}
        \xi^i_{2, k} 1_{J^{n,\varepsilon}_{k,0}} - \tilde\xi^i_{2, k}
        &= g(X^\varepsilon_{t_{k-1}})
          \left\{ \left| \int_{t_{k-1}}^{t_k} \left\{ a(X^\varepsilon_t,\mu_0) - a(X^\varepsilon_{t_{k-1}}, \mu_0) \right\} dt \right|^2 \right. \\
        &\quad+ 2 \varepsilon \int_{t_{k-1}}^{t_k} \left\{ a(X^\varepsilon_t,\mu_0) - a(X^\varepsilon_{t_{k-1}}, \mu_0) \right\} dt
        \int_{t_{k-1}}^{t_k} b(X^\varepsilon_t,\sigma_0) \, dW_t \\
        &\quad\left.+ \left| \varepsilon \int_{t_{k-1}}^{t_k} b(X^\varepsilon_t,\sigma_0) \, dW_t \right|^2
          - \left| \varepsilon \int_{t_{k-1}}^{t_k} b(X^\varepsilon_{t_{k-1}},\sigma_0) \, dW_t \right|^2 \right\}
          1_{C^{n,\varepsilon,\rho}_{k,0}},
      \end{align*}
      and by the same argument as in the proof of Lemma \ref{lem:4.15}, we obtain
      \begin{align*}
        &\frac{\sqrt{n}}{\varepsilon^2} \sum_{k=1}^n E \left[
          |g ( X^\varepsilon_{t_{k-1}})|
          \left| \int_{t_{k-1}}^{t_k} \left\{ a(X^\varepsilon_t,\mu_0) - a(X^\varepsilon_{t_{k-1}}, \mu_0) \right\} dt \right|^2 1_{C^{n,\varepsilon,\rho}_{k,0}}
          \, \middle| \, \mathcal{F}_{t_{k-1}} \right]
          = O_p \left( \frac{1}{\varepsilon^2 n^{5/2}} + \frac{1}{n^{3/2}} \right), \\
        &\frac{\sqrt{n}}{\varepsilon} \sum_{k=1}^{n} E \left[
          \left|g(X^\varepsilon_{t_{k-1}})
          \int_{t_{k-1}}^{t_k} \left\{ a(X^\varepsilon_t,\mu_0) - a(X^\varepsilon_{t_{k-1}}, \mu_0) \right\} dt
          \int_{t_{k-1}}^{t_k} b(X^\varepsilon_t,\sigma_0) \, dW_t \right|
          1_{C^{n,\varepsilon,\rho}_{k,0}}
          \, \middle| \, \mathcal{F}_{t_{k-1}} \right] \\
        &\quad
          = O_p \left( \frac{1}{\varepsilon n} + \frac{1}{\sqrt{n}} \right), \\
        &\sqrt{n} \sum_{k=1}^{n} E \left[
          \left|g(X^\varepsilon_{t_{k-1}})
          \left| \int_{t_{k-1}}^{t_k} b(X^\varepsilon_t,\sigma_0) \, dW_t \right|^2
          - \left| \int_{t_{k-1}}^{t_k} b(X^\varepsilon_{t_{k-1}},\sigma_0) \, dW_t \right|^2 \right|
          1_{C^{n,\varepsilon,\rho}_{k,0}}
          \, \middle| \, \mathcal{F}_{t_{k-1}} \right]
          = O_p \left( \frac{1}{\sqrt{n}} + \varepsilon \right)
      \end{align*}
      as $n\to\infty$, $\varepsilon\to0$, $\lambda_\varepsilon\to\infty$, $\varepsilon n\to\infty$ and $\varepsilon\lambda_\varepsilon\to0$.

      For $\ell=3$, let $r_{n,\varepsilon}$ be defined as either of the following:
      \begin{enumerate}
        \item[(i)]  Under Assumption \ref{asmp:A4} \ref{asmp:A4(i)},
                    $r_{n,\varepsilon}:=\frac{1}{\varepsilon n^{1-1/p}} + \frac{1}{n^{1/2-1/p}}$ with sufficiently large $p>1$.
        \item[(ii)] Under Assumption \ref{asmp:A4} \ref{asmp:A4(ii)},
                    $r_{n,\varepsilon}:=\frac{1}{\varepsilon n^{1-1/p-q\rho}} + \frac{1}{n^{1/2-1/p-q\rho}}$ with sufficiently large $p>1$.
      \end{enumerate}
      Then, it follows from Lemmas \ref{lem:4.10}, \ref{lem:4.12} and \ref{lem:A.3} that
      \begin{align*}
        &\sum_{k=1}^n E \left[ \left| \xi^i_{3,k} 1_{J^{n,\varepsilon}_{k,1}} - \tilde\xi^i_{3,k} \right|
        \, \middle| \, \mathcal{F}_{t_{k-1}} \right]
        = O_p \left( \sqrt{\lambda_\varepsilon} r_{n,\varepsilon} \right)
      \end{align*}
      as $n\to\infty$, $\varepsilon\to0$, $\lambda_\varepsilon\to\infty$, $\varepsilon n\to\infty$, $\varepsilon\lambda_\varepsilon\to0$, $\lambda_\varepsilon^2/n\to0$ and
      $\lambda_\varepsilon \int_{|z|\leq\kappa/n^\rho} f_{\alpha_0}(z) \, dz\to0$.
      Here,
      \begin{equation*}
        \sqrt{\lambda_\varepsilon} r_{n,\varepsilon} \to 0
      \end{equation*}
      as $n\to\infty$, $\varepsilon\to0$, $\lambda_\varepsilon\to\infty$, $\varepsilon\lambda_\varepsilon\to0$, $\lambda_\varepsilon^2/n\to0$ and
      $\lambda_\varepsilon \int_{|z|\leq\kappa/n^\rho} f_{\alpha_0}(z) \, dz\to0$ with $\lim (\varepsilon^2 n)^{-1}<\infty$, when $p>4$ under Assumption \ref{asmp:A4} \ref{asmp:A4(i)},
      and when $p>4/(1-4q\rho)$ under Assumption \ref{asmp:A4} \ref{asmp:A4(ii)}.
    \end{proof}

    \begin{lemma}\label{lem:4.17}
      Under Assumptions \ref{asmp:A2} to \ref{asmp:A8} and \ref{asmp:A10},
      \begin{equation*}
        \sum_{k=1}^n E \left[ \tilde\xi^i_{\ell k} \, \middle| \, \mathcal{F}_{t_{k-1}} \right]
        \overset{p}\longrightarrow 0
      \end{equation*}
      as $n\to\infty$, $\varepsilon\to0$, $\lambda_\varepsilon\to\infty$, $\varepsilon\lambda_\varepsilon\to0$ and $\lambda_\varepsilon/n\to0$.
    \end{lemma}

    \begin{proof}
      For $\ell=1$, let $i\in\{1,\dots,d_1\}$, and put $g(x)=\frac{\partial a}{\partial \mu_i}(x,\mu_0)/b(x,\sigma_0)$.
      Since
      \begin{equation*}
        E \left[
        \int_{t_{k-1}}^{t_k} dW_t
        \, \middle| \, \mathcal{F}_{t_{k-1}} \right] = 0,
        \quad
        \text{and } \int_{t_{k-1}}^{t_k} dW_t \text{ and } 1_{J^{n,\varepsilon}_{k,i}}~(i=1,2)
        \text{ are independent},
      \end{equation*}
      it holds from Lemmas \ref{lem:4.4} and \ref{lem:4.7} that for any $p\geq1$
      \begin{align*}
        \left| \sum_{k=1}^n
          E \left[ \tilde{\xi}^i_{1,k}
          \, \middle| \, \mathcal{F}_{t_{k-1}} \right] \right|
        &= \left| \sum_{k=1}^n g ( X^\varepsilon_{t_{k-1}} )
          E \left[
          \int_{t_{k-1}}^{t_k} \, dW_t
          1_{C^{n,\varepsilon,\rho}_{k,0}}
          \, \middle| \, \mathcal{F}_{t_{k-1}} \right] \right| \\
        &= \left| \sum_{k=1}^n
            g ( X^\varepsilon_{t_{k-1}} )
            E \left[
            \int_{t_{k-1}}^{t_k} \, dW_t
            1_{D^{n,\varepsilon,\rho}_{k,0}\cup J^{n,\varepsilon}_{k,1}\cup J^{n,\varepsilon}_{k,2}}
            \, \middle| \, \mathcal{F}_{t_{k-1}} \right] \right| \\
        &= O_p \left(\frac{1}{n^{p(1-\rho)-1/2}} + \frac{\varepsilon^p}{n^{p(1/2-\rho)-1/2}} \right) + O_p \left( \frac{\lambda_\varepsilon}{n} \right) + O_p \left( \frac{\lambda_\varepsilon^2}{n^2} \right)
      \end{align*}
      as $n\to\infty$, $\varepsilon\to0$, $\lambda_\varepsilon\to\infty$ and $\varepsilon\lambda_\varepsilon\to0$.

      For $\ell=2$,
      let $i\in\{1,\dots,d_2\}$, and put $g(x)=-\frac{\partial b}{\partial \sigma_i}(x,\sigma_0)/b(x,\sigma_0)$.
      Since
      \begin{equation*}
        E \left[
        \left| \int_{t_{k-1}}^{t_k} \, dW_t \right|^2
        \, \middle| \, \mathcal{F}_{t_{k-1}} \right]
        = \frac{1}{n},
        \quad \text{and} ~\left| \int_{t_{k-1}}^{t_k} \, dW_t \right|^2
        \text{ and } 1_{J^{n,\varepsilon}_{k,i}}~(i=1,2) \text{ are independent},
      \end{equation*}
      it follows from Lemmas \ref{lem:4.4} and \ref{lem:4.7} that for any $p\geq1$
      \begin{align*}
        \left| \sum_{k=1}^n
          E \left[ \tilde{\xi}^i_{2,k}
          \, \middle| \, \mathcal{F}_{t_{k-1}} \right] \right|
        &=\left| \sqrt{n} \sum_{k=1}^{n} g(X^\varepsilon_{t_{k-1}}) E \left[
          \left\{ \left| \int_{t_{k-1}}^{t_k} dW_t \right|^2
          - \frac{1}{n} \right\}
          1_{C^{n,\varepsilon,\rho}_{k,0}}
          \, \middle| \, \mathcal{F}_{t_{k-1}} \right] \right| \\
        &= \left| \sqrt{n} \sum_{k=1}^{n}
          g(X^\varepsilon_{t_{k-1}})
          E \left[
          \left\{ \left| \int_{t_{k-1}}^{t_k}  \, dW_t \right|^2
          - \frac{1}{n} \right\}
          1_{D^{n,\varepsilon,\rho}_{k,0}}
          \, \middle| \, \mathcal{F}_{t_{k-1}} \right] \right| \\
        &= O_p \left( \frac{1}{n^{p(1-\rho)-1/2}} + \frac{\varepsilon^p}{n^{p(1/2-\rho)-1/2}} \right)
      \end{align*}
      as $n\to\infty$, $\varepsilon\to0$, $\lambda_\varepsilon\to\infty$ and $\varepsilon\lambda_\varepsilon\to0$.

      For $\ell=3$, we may assume $\sup_t |X^\varepsilon_{t} - x_t |<\delta$ for some enough small $\delta>0$.
      From Assumption \ref{asmp:A10}, we obtain
      \begin{align*}
        \sum_{k=1}^n E \left[ \tilde\xi^i_{3, k} \, \middle| \, \mathcal{F}_{t_{k-1}} \right]
        &= \frac{\sqrt{\lambda_\varepsilon}}{n} \sum_{k=1}^n
          \int \frac{\partial\psi}{\partial\alpha_i} \left( X^\varepsilon_{t_{k-1}}, c(X^\varepsilon_{t_{k-1}},\alpha_0) z, \alpha_0 \right) f_{\alpha_0}(z) \, dz \\
        &= \frac{\sqrt{\lambda_\varepsilon}}{n} \sum_{k=1}^n
          \frac{\partial}{\partial\alpha_i} \left( \int \psi \left( X^\varepsilon_{t_{k-1}}, c(X^\varepsilon_{t_{k-1}},\alpha_0) z, \alpha \right) f_{\alpha_0}(z) \, dz \right)_{\alpha=\alpha_0}
        = 0.
      \end{align*}
      The last equality holds from the fact that
      \begin{equation*}
        \alpha \mapsto \int \psi \left( X^\varepsilon_{t_{k-1}}, c(X^\varepsilon_{t_{k-1}},\alpha_0) z, \alpha \right) f_{\alpha_0}(z) \, dz
      \end{equation*}
      behaves like the Kullbac--Leibler divergence from $p_{\alpha,x}$ to $p_{\alpha_0,x}$ at $x=X^\varepsilon_{t_{k-1}}$, where
      \begin{equation*}
        p_{\alpha,x}(y):=\frac{1}{c(x,\alpha)} f_{\alpha} \left( \frac{y}{c(x,\alpha)} \right) \quad (y\in\mathbb{R})
      \end{equation*}
      for $(x,\alpha)\in I_{x_0}^\delta\times\Theta_3$.
    \end{proof}

    \begin{lemma}\label{lem:4.18}
      Under Assumptions \ref{asmp:A2} to \ref{asmp:A8}, \ref{asmp:A10} and \ref{asmp:A11},
      \begin{align*}
        \sum_{k=1}^n E \left[ \tilde\xi^{i_1}_{\ell k} \tilde\xi^{i_2}_{\ell k} \, \middle| \, \mathcal{F}_{t_{k-1}} \right]
        &\overset{p}\longrightarrow I^{i_1i_2}_{\ell}
        & &(\ell=1,2,3,~i_1,i_2=1,\dots,d_\ell),\\
        \sum_{k=1}^n E \left[ \tilde\xi^{i_1}_{\ell_1k} \tilde\xi^{i_2}_{\ell_2k} \, \middle| \, \mathcal{F}_{t_{k-1}} \right]
        &\overset{p}\longrightarrow 0 & & (\ell_1,\ell_2=1,2,3,~\ell_1\neq\ell_2,~i_j=1,\dots,d_{\ell_j},~j=1,2)
      \end{align*}
      as $n\to\infty$, $\varepsilon\to0$, $\lambda_\varepsilon\to\infty$ and $\varepsilon\lambda_\varepsilon\to0$,
      where
      \begin{align*}
        I^{i_1i_2}_1 &:= \int_0^1 \frac{\frac{\partial a}{\partial \mu_{i_1}}(x_t,\mu_0)\frac{\partial a}{\partial \mu_{i_2}}(x_t,\mu_0)}{|b(x_t,\mu_0)|^2} dt, \quad
        I^{i_1i_2}_2 := 2 \int_0^1 \frac{\frac{\partial b}{\partial \sigma_{i_1}}(x_t,\sigma_0)\frac{\partial b}{\partial \sigma_{i_2}}(x_t,\sigma_0)}{|b(x_t,\sigma_0)|^2} dt, \\
        I^{i_1i_2}_3 &:= \int_0^1 \int \frac{\partial \psi}{\partial \alpha_{i_1}}(x_t,c(x_t,\alpha_0) z,\alpha_0)\frac{\partial \psi}{\partial \alpha_{i_2}}(x_t,c(x_t,\alpha_0) z,\alpha_0) f_{\alpha_0}(z) \, dz \, dt.
      \end{align*}
    \end{lemma}

    \begin{proof}
      For $\ell=1$, $i_1,i_2\in\{1,\dots,d_1\}$,
      put $g(x)=\frac{\partial a}{\partial \mu_{i_1}}\frac{\partial a}{\partial \mu_{i_2}}(x,\mu_0)/b(x,\sigma_0)^2$.
      Since from Lemmas \ref{lem:4.4} and \ref{lem:4.7} for any $p>1$ we have
      \begin{align*}
        &\sum_{k=1}^n g(X^\varepsilon_{t_{k-1}}) E \left[ \left | \int_{t_{k-1}}^{t_k} dW_t \right |^2
        1_{D^{n,\varepsilon,\rho}_{k,0}\cup J^{n,\varepsilon}_{k,1} \cup J^{n,\varepsilon}_{k,2}} \, \middle| \, \mathcal{F}_{t_{k-1}} \right]
        = O_p \left( \frac{1}{n^{p(1-\rho)}} + \frac{\varepsilon^p}{n^{p(1/2-\rho)}} + \frac{\lambda_\varepsilon}{n} \right)
      \end{align*}
      as $n\to\infty$, $\varepsilon\to0$, $\lambda_\varepsilon\to\infty$ and $\varepsilon\lambda_\varepsilon\to0$,
      we obtain
      \begin{align*}
        \sum_{k=1}^n E \left[ \tilde\xi^{i_1}_{1,k} \tilde\xi^{i_2}_{1,k} \, \middle| \, \mathcal{F}_{t_{k-1}} \right]
        &=\sum_{k=1}^n g(X^\varepsilon_{t_{k-1}}) E \left[ \left | \int_{t_{k-1}}^{t_k} dW_t \right |^2
        1_{C^{n,\varepsilon,\rho}_{k,0}} \, \middle| \, \mathcal{F}_{t_{k-1}} \right] \\
        &= \frac{1}{n} \sum_{k=1}^n g(X^\varepsilon_{t_{k-1}})
          + O_p \left( \frac{1}{n^{p(1-\rho)}} + \frac{\varepsilon^p}{n^{p(1/2-\rho)}} + \frac{\lambda_\varepsilon}{n} \right)
        \overset{p}\longrightarrow
          \int_0^1 g(x_t) dt
      \end{align*}
      as $n\to\infty$, $\varepsilon\to0$, $\lambda_\varepsilon\to\infty$, $\varepsilon\lambda_\varepsilon\to0$, and $\lambda_\varepsilon/n\to0$.

      For $\ell=2$, $i_1,i_2\in\{1,\dots,d_2\}$,
      put $g(x)=\frac{\partial b}{\partial \sigma_{i_1}}(x,\sigma_0)\frac{\partial b}{\partial \sigma_{i_2}}(x,\sigma_0)/b(x,\sigma_0)^2$.
      Since similarly to the proof of Lemma \ref{lem:4.17},
      it follows from Lemmas \ref{lem:4.4} and \ref{lem:4.7} that for any $p>1$
      \begin{align*}
        &n \sum_{k=1}^{n} g(X^\varepsilon_{t_{k-1}}) E \left[
          \left| \left| \int_{t_{k-1}}^{t_k} \, dW_t \right|^2
          - \frac{1}{n} \right|^2
          1_{D^{n,\varepsilon,\rho}_{k,0}\cup J^{n,\varepsilon}_{k,1} \cup J^{n,\varepsilon}_{k,2}}
          \, \middle| \, \mathcal{F}_{t_{k-1}} \right] \\
        &\quad = O_p \left( \frac{1}{n^{p(1-\rho)}} + \frac{\varepsilon^p}{n^{p(1/2-\rho)}} \right)
          + O_p \left( \frac{\lambda_\varepsilon}{n} \right) + O_p \left( \frac{\lambda_\varepsilon^2}{n^2} \right)
      \end{align*}
      as $n\to\infty$, $\varepsilon\to0$, $\lambda_\varepsilon\to\infty$ and $\varepsilon\lambda_\varepsilon\to0$,
      we obtain from Lemma \ref{lem:4.4} that
      \begin{align*}
        &\sum_{k=1}^n E \left[ \tilde\xi^{i_1}_{2,k} \tilde\xi^{i_2}_{2,k} \, \middle| \, \mathcal{F}_{t_{k-1}} \right]
        = n \sum_{k=1}^{n} g(X^\varepsilon_{t_{k-1}}) E \left[
          \left| \left| \int_{t_{k-1}}^{t_k} \, dW_t \right|^2
          - \frac{1}{n} \right|^2
          1_{C^{n,\varepsilon,\rho}_{k,0}}
          \, \middle| \, \mathcal{F}_{t_{k-1}} \right] \\
        &\quad= n \sum_{k=1}^{n} g(X^\varepsilon_{t_{k-1}}) E \left[
          \left| \left| \int_{t_{k-1}}^{t_k} \, dW_t \right|^2
          - \frac{1}{n} \right|^2
          \, \middle| \, \mathcal{F}_{t_{k-1}} \right]
        + O_p \left( \frac{1}{n^{p(1-\rho)}} + \frac{\varepsilon^p}{n^{p(1/2-\rho)}} + \frac{\lambda_\varepsilon}{n} \right)
        \overset{p}\longrightarrow 2\int_0^1 g(x_t) \, dt
      \end{align*}
      as $n\to\infty$, $\varepsilon\to0$, $\lambda_\varepsilon\to\infty$, $\varepsilon\lambda_\varepsilon\to0$ and $\lambda_\varepsilon/n\to0$.

      For $\ell=3$, $i_1,i_2\in\{1,\dots,d_3\}$, put $g(x,y)=\frac{\partial \psi}{\partial \alpha_{i_1}}(x,y,\alpha_0)\frac{\partial \psi}{\partial \alpha_{i_2}}(x,y,\alpha_0)$.
      Then, it follows from Lemma \ref{lem:4.4} and Assumption \ref{asmp:A11} that
      \begin{align*}
        \sum_{k=1}^n E \left[ \tilde\xi^{i_1}_{3,k} \tilde\xi^{i_2}_{3,k} \, \middle| \, \mathcal{F}_{t_{k-1}} \right]
        &= \frac{1}{\lambda_\varepsilon} \sum_{k=1}^n E \left[
          g \left( X^\varepsilon_{t_{k-1}}, c(X^\varepsilon_{t_{k-1}},\alpha_0) V_{N^{\lambda_\varepsilon}_{\tau_k}} \right)
          1_{J^{n,\varepsilon}_{k,1}}
          \, \middle| \, \mathcal{F}_{t_{k-1}} \right] \\
        &= \frac1n \sum_{k=1}^n \int g \left( X^\varepsilon_{t_{k-1}}, c(X^\varepsilon_{t_{k-1}},\alpha_0) z \right) f_{\alpha_0}(z) \, dz \\
        &\overset{p}\longrightarrow
        \int_0^1 \int g \left( x_t, c(x_t,\alpha_0) z \right) f_{\alpha_0}(z) \, dz \, dt
      \end{align*}
      as $n\to\infty$, $\varepsilon\to0$, $\lambda_\varepsilon\to\infty$ and $\varepsilon\lambda_\varepsilon\to0$.
      The second equality holds from the fact that $V_{N^{\lambda_\varepsilon}_{\tau_k}}$ and $1_{J^{n,\varepsilon}_{k,1}}$ are independent.

      For $\ell_j=j$, $i_j=1,\dots,d_j$ ($j=1,2$),
      put $g(x)=-\frac{\partial a}{\partial \mu_{i_1}}(x,\mu_0)\frac{\partial b}{\partial \sigma_{i_2}}(x,\sigma_0)/|b(x,\sigma_0)|^2$.
      Since for $i=1,3$ and $j=1,2$
      \begin{equation*}
        E \left[
        \left( \int_{t_{k-1}}^{t_k} dW_t \right)^i
        \, \middle| \, \mathcal{F}_{t_{k-1}} \right] = 0,
        \text{ and } \left( \int_{t_{k-1}}^{t_k} dW_t \right)^i \text{ and } 1_{J^{n,\varepsilon}_{k,j}}
        \text{ are independent},
      \end{equation*}
      it follows from Lemmas \ref{lem:4.4} and \ref{lem:4.7} that for any $p\geq1$
      \begin{align*}
        \sum_{k=1}^n E \left[ \tilde\xi^{i_1}_{1,k} \tilde\xi^{i_2}_{2,k} \, \middle| \, \mathcal{F}_{t_{k-1}} \right]
        &
          = \sqrt{n} \sum_{k=1}^n g( X^\varepsilon_{t_{k-1}}) E \left[
          \left \{ - \left | \int_{t_{k-1}}^{t_k} dW_t \right |^2 + \frac{1}{n} \right \}
          \int_{t_{k-1}}^{t_k} dW_t
          1_{C^{n,\varepsilon,\rho}_{k,0}}
          \, \middle| \, \mathcal{F}_{t_{k-1}} \right] \\
        &
          = \sqrt{n} \sum_{k=1}^n g( X^\varepsilon_{t_{k-1}}) E \left[
          \left \{ - \left | \int_{t_{k-1}}^{t_k} dW_t \right |^2 + \frac{1}{n} \right \}
          \int_{t_{k-1}}^{t_k} dW_t
          1_{D^{n,\varepsilon,\rho}_{k,0}}
          \, \middle| \, \mathcal{F}_{t_{k-1}} \right] \\
        &
          = O_p \left( \frac{1}{n^{p(1-\rho)}} + \frac{\varepsilon^p}{n^{p(1/2-\rho)}} \right)
      \end{align*}
      as $n\to\infty$, $\varepsilon\to0$, $\lambda_\varepsilon\to\infty$ and $\varepsilon\lambda_\varepsilon\to0$.
    \end{proof}

    \begin{lemma}\label{lem:4.19}
      Under Assumptions \ref{asmp:A2} to \ref{asmp:A8} and \ref{asmp:A10},
      \begin{equation*}
        \sum_{k=1}^n \left|E \left[ \tilde\xi^i_{\ell k} \, \middle| \, \mathcal{F}_{t_{k-1}} \right]\right|^2
        \overset{p}\longrightarrow 0
      \end{equation*}
      as $n\to\infty$, $\varepsilon\to0$, $\lambda_\varepsilon\to\infty$, $\varepsilon\lambda_\varepsilon\to0$ and $\lambda_\varepsilon/n\to0$.
    \end{lemma}

    \begin{proof}
      This follows from the same argument as in the proof of Lemma \ref{lem:4.17}.
    \end{proof}

    \begin{lemma}\label{lem:4.20}
      Under Assumptions \ref{asmp:A2} to \ref{asmp:A8} and \ref{asmp:A11},
      \begin{equation*}
        \sum_{k=1}^n E \left[ |\tilde\xi^i_{\ell k}|^4 \, \middle| \, \mathcal{F}_{t_{k-1}} \right]
        \overset{p}\longrightarrow 0
      \end{equation*}
      as $n\to\infty$, $\varepsilon\to0$, $\lambda_\varepsilon\to\infty$ and $\varepsilon\lambda_\varepsilon\to0$.
    \end{lemma}

    \begin{proof}
      For $\ell=1$, let $i\in\{1,\dots,d_1\}$, and put $g(x)=|\frac{\partial a}{\partial \mu_i}(x,\mu_0)/b(x,\sigma_0)|^4$.
      Then, it holds from Lemma \ref{lem:4.4} that
      \begin{equation*}
        \sum_{k=1}^n
          E \left[ |\tilde{\xi}^i_{1,k}|^4
          \, \middle| \, \mathcal{F}_{t_{k-1}} \right]
        \leq \sum_{k=1}^n g ( X^\varepsilon_{t_{k-1}} )
          E \left[
          \left| \int_{t_{k-1}}^{t_k} \, dW_t \right|^4
          \, \middle| \, \mathcal{F}_{t_{k-1}} \right]
          \overset{p}\longrightarrow 0
      \end{equation*}
      as $n\to\infty$, $\varepsilon\to0$, $\lambda_\varepsilon\to\infty$ and $\varepsilon\lambda_\varepsilon\to0$.

      For $\ell=2$,
      let $i\in\{1,\dots,d_2\}$, and put $g(x)=|\frac{\partial b}{\partial \sigma_i}(x,\sigma_0)/b(x,\sigma_0)|^4$.
      Then, it follows from Lemma \ref{lem:4.4} that
      \begin{equation*}
        \sum_{k=1}^n
          E \left[ | \tilde{\xi}^i_{2,k}|^4
          \, \middle| \, \mathcal{F}_{t_{k-1}} \right]
        \leq n^2 \sum_{k=1}^{n} g(X^\varepsilon_{t_{k-1}}) E \left[
          \left| \left| \int_{t_{k-1}}^{t_k} dW_t \right|^2
          - \frac{1}{n}\right|^4
          \, \middle| \, \mathcal{F}_{t_{k-1}} \right]
          \overset{p}\longrightarrow 0
      \end{equation*}
      as $n\to\infty$, $\varepsilon\to0$, $\lambda_\varepsilon\to\infty$ and $\varepsilon\lambda_\varepsilon\to0$.

      For $\ell=3$, $i\in\{1,\dots,d_3\}$, put $g(x,y)=|\frac{\partial \psi}{\partial \alpha_i}(x,y,\alpha_0)|^4$.
      Then, similarly to the proof of Lemma \ref{lem:4.18}, it follows from Lemma \ref{lem:4.4} and Assumption \ref{asmp:A11} that
      \begin{align*}
        \sum_{k=1}^n E \left[ |\tilde\xi^{i}_{3,k}|^4 \, \middle| \, \mathcal{F}_{t_{k-1}} \right]
        = \frac{1}{\lambda_\varepsilon^2} \sum_{k=1}^n E \left[
          g \left( X^\varepsilon_{t_{k-1}}, c(X^\varepsilon_{t_{k-1}},\alpha_0) V_{N^{\lambda_\varepsilon}_{\tau_k}} \right)
          1_{J^{n,\varepsilon}_{k,1}}
          \, \middle| \, \mathcal{F}_{t_{k-1}} \right]
        \overset{p}\longrightarrow 0
      \end{align*}
      as $n\to\infty$, $\varepsilon\to0$, $\lambda_\varepsilon\to\infty$ and $\varepsilon\lambda_\varepsilon\to0$.
    \end{proof}

    \begin{proof}[Proof of Theorem \ref{thm:3.2}]
      From Theorem A.3 in Shimizu \cite{shimizu2007asymptotic} and Lemmas \ref{lem:4.16} to \ref{lem:4.20},
      \begin{equation*}
        \Lambda_{n,\varepsilon} :=
        \sum_{k=1}^n \left( \xi^1_{1,k}, \dots, \xi^{d_1}_{1,k},
        \xi^1_{2,k}, \dots, \xi^{d_2}_{2,k},
        \xi^1_{3,k}, \dots, \xi^{d_3}_{3,k} \right)^T
        \overset{d}\longrightarrow
        \mathcal{N} \left( 0, I_{\theta_0} \right)
      \end{equation*}
      as $n\to\infty$, $\varepsilon\to0$, $\lambda_\varepsilon\to\infty$, $\varepsilon\lambda_\varepsilon\to0$, $\lambda_\varepsilon^2/n\to0$ and
      $\lambda_\varepsilon \int_{|z|\leq\kappa/n^\rho} f_{\alpha_0}(z) \, dz\to0$ with $\lim (\varepsilon^2 n)^{-1}<\infty$.
      Also, it follows from Lemmas \ref{lem:4.11} to \ref{lem:4.15}
      under Assumption \ref{asmp:A12}
      that
      \begin{equation}\label{eq:matrixCtoI0inp}
        \begin{aligned}
          C_{\varepsilon,n} (\theta)
          &:= \left( \begin{array}{ccc}
            \varepsilon^2n \left( \frac{\partial^2}{\partial\mu_{i}\partial\mu_{j}} \Psi_{n,\varepsilon}(\theta) \right)_{ij}&
            \varepsilon^2n \left( \frac{\partial^2}{\partial\mu_{i}\partial\sigma_{j}} \Psi_{n,\varepsilon}(\theta) \right)_{ij} &
            0 \\
            \left( \frac{\partial^2}{\partial\sigma_{i}\partial\mu_{j}} \Psi_{n,\varepsilon}(\theta) \right)_{ij}&
            \left( \frac{\partial^2}{\partial\sigma_{i}\partial\sigma_{j}} \Psi_{n,\varepsilon}(\theta) \right)_{ij} &
            0 \\
            0 &
            0 &
            \left( \frac{\partial^2}{\partial\alpha_{i}\partial\alpha_{j}} \Psi_{n,\varepsilon}(\theta) \right)_{ij} \\
          \end{array} \right)
          \overset{p}\longrightarrow - I_{\theta_0}
        \end{aligned}
      \end{equation}
      as $n\to\infty$, $\varepsilon\to0$, $\lambda_\varepsilon\to\infty$, $\varepsilon\lambda_\varepsilon\to0$, $\lambda_\varepsilon^2/n\to0$ and
      $\lambda_\varepsilon \int_{|z|\leq\kappa/n^\rho} f_{\alpha_0}(z) \, dz\to0$ with $\lim (\varepsilon^2 n)^{-1}<\infty$, uniformly in $\theta\in\Theta$.
      Indeed,
      \begin{align*}
        \varepsilon^2n \frac{\partial^2}{\partial\mu_{i}\partial\mu_{j}} \Psi_{n,\varepsilon}(\theta)
        &= \sum_{k=1}^{n}
          \left \{ \Delta^n_k X^{\varepsilon} - \frac{1}{n}a(X^\varepsilon_{t_{k-1}}, \mu) \right \} \frac{ \frac{\partial^2 a}{\partial \mu_i \partial \mu_j} (X^\varepsilon_{t_{k-1}}, \mu) }{\left | b(X^\varepsilon_{t_{k-1}},\sigma) \right |^2}
          1_{C^{n,\varepsilon,\rho}_k} \\
        &\quad - \frac{1}{n} \sum_{k=1}^{n}
          \frac{ \frac{\partial a}{\partial\mu_i}\frac{\partial a}{\partial\mu_j} (X^\varepsilon_{t_{k-1}}, \mu) }{\left | b(X^\varepsilon_{t_{k-1}},\sigma) \right |^2}
          1_{C^{n,\varepsilon,\rho}_k}
          \overset{p}\longrightarrow - I^{i_1i_2}_1, \\
        \varepsilon^2n\frac{\partial^2}{\partial\mu_{i}\partial\sigma_{j}} \Psi_{n,\varepsilon}(\theta)
        &= - 2 \sum_{k=1}^{n}
          \left \{ \Delta^n_k X^{\varepsilon} - \frac{1}{n}a(X^\varepsilon_{t_{k-1}}, \mu) \right \}
          \frac{ \frac{\partial a}{\partial \mu_i} (X^\varepsilon_{t_{k-1}}, \mu) \frac{\partial b}{\partial\sigma_j}(X^\varepsilon_{t_{k-1}},\sigma) }{\left | b(X^\varepsilon_{t_{k-1}},\sigma) \right |^3}
          1_{C^{n,\varepsilon,\rho}_k}
          \overset{p}\longrightarrow 0, \\
        \frac{\partial^2}{\partial\sigma_{i}\partial\sigma_{j}} \Psi_{n,\varepsilon}(\theta)
        &= - \frac{1}{n} \sum_{k=1}^{n}
          \left \{ - \frac{ \left | \Delta^n_k X^{\varepsilon} - a(X^\varepsilon_{t_{k-1}}, \mu) / n \right |^2 }{\left | \varepsilon b(X^\varepsilon_{t_{k-1}},\sigma) \right |^2 / n}
          + 1 \right \} \frac{\partial (\frac{\partial b}{\partial\sigma_i}/b)}{\partial\sigma_j}(X^\varepsilon_{t_{k-1}},\sigma) 1_{C^{n,\varepsilon,\rho}_k} \\
        &\quad - \frac{2}{\varepsilon^2} \sum_{k=1}^{n}
          \left | \Delta^n_k X^{\varepsilon} - a(X^\varepsilon_{t_{k-1}}, \mu) / n \right |^2
          \frac{\frac{\partial b}{\partial \sigma_i}\frac{\partial b}{\partial \sigma_j}(X^\varepsilon_{t_{k-1}},\sigma)}{|b(X^\varepsilon_{t_{k-1}},\sigma)|^4} 1_{C^{n,\varepsilon,\rho}_k}
          \overset{p}\longrightarrow - I^{i_1i_2}_2,\\
        \frac{\partial^2}{\partial\alpha_{i}\partial\alpha_{j}} \Psi_{n,\varepsilon}(\theta)
        &= \frac{1}{\lambda_\varepsilon}
          \sum_{k=1}^{n} \frac{1}{\varphi} \frac{\partial^2 \varphi}{\partial\alpha_i\partial\alpha_j}
          \left( X^\varepsilon_{t_{k-1}}, \frac{\Delta^n_k X^{\varepsilon}}\varepsilon, \alpha \right)
          1_{D^{n,\varepsilon,\rho}_k} \\
        &\quad - \frac{1}{\lambda_\varepsilon}
          \sum_{k=1}^{n}
          \frac{1}{|\varphi|^2} \frac{\partial \varphi}{\partial\alpha_i}\frac{\partial \varphi}{\partial\alpha_j}
          \left( X^\varepsilon_{t_{k-1}}, \frac{\Delta^n_k X^{\varepsilon}}\varepsilon, \alpha \right)
          1_{D^{n,\varepsilon,\rho}_k}
          \overset{p}\longrightarrow - I^{i_1i_2}_3,
      \end{align*}
      where $\varphi(x,y,\alpha) := \exp \psi(x,y,\alpha)$.
      Since
      \begin{equation*}
        D_{n,\varepsilon}
        \left( \begin{aligned}
          \varepsilon^{-1} (\hat{\mu}_{n,\varepsilon}-\mu_0) \\
          \sqrt{n} (\hat{\sigma}_{n,\varepsilon}-\sigma_0) \\
          \sqrt{\lambda_\varepsilon} (\hat{\alpha}_{n,\varepsilon}-\alpha_0)
        \end{aligned} \right)
        = \Lambda_{n,\varepsilon},
      \end{equation*}
      where
      \begin{equation*}
        D_{n,\varepsilon} := \int_0^1 C_{n,\varepsilon}(\theta_0+u(\hat{\theta}_{n,\varepsilon}-\theta_0)),
      \end{equation*}
      the conclusion follows by the same argument in the proof of Theorem 1 in S{\o}rensen and Uchida \cite{sorensen2003small}.
    \end{proof}

\section{Examples}\label{sec:examples}
  This section is devoted to give some examples of densities which satisfies Assumptions \ref{asmp:A9} to \ref{asmp:A12}.
  For simplicity, suppose that $c(x,\alpha)$ is an enough smooth postive function on $I_{x_0}^\delta\times\Theta_3$, and derivatives of $c$ are uniformly continuous.
  Let $D_+$ is the interior of the common support of $\{f_\alpha\}_{\alpha\in\Theta_3}$, i.e.,
  \begin{equation*}
    f_\alpha(z) \left\{ \begin{aligned}
      & > 0     & & \text{for } z \in D_+, \\
      & = 0     & & \text{otherwise}.
    \end{aligned} \right.
  \end{equation*}
  Note that $y\in D_+(=\mathbb{R}\text{ or }\mathbb{R}_+)$ if and only if $y/c(x,\alpha)\in D_+$ for $(x,\alpha)\in I_{x_0}^\delta\times\Theta_3$ owing to Assumption \ref{asmp:A4}.
  If $(x,y,\alpha)\in I_{x_0}^\delta\times D_+\times\Theta_3$,
  \begin{align*}
    \frac{\partial \psi}{\partial y}(x, y, \alpha)
    & = \frac{1}{c(x,\alpha)}
          \frac{f_\alpha'\left( \frac{y}{c(x,\alpha)} \right)}{f_\alpha\left( \frac{y}{c(x,\alpha)} \right)}, \\
    \frac{\partial \psi}{\partial \alpha_j} (x, y, \alpha)
    & = - \frac{\partial(\log c)}{\partial \alpha_j}(x,\alpha)
          \left( 1 + y \, \frac{\partial \psi}{\partial y}(x,y,\alpha) \right)
          + \frac{\frac{\partial f_\alpha}{\partial \alpha_j}\left( \frac{y}{c(x,\alpha)} \right)}{c(x,\alpha) f_\alpha\left( \frac{y}{c(x,\alpha)} \right)}
  \end{align*}
  for $(x,\alpha)\in I_{x_0}^\delta\times\Theta_3$.
  The values of these functions may be undefined if $(x,y,\alpha)\in I_{x_0}^\delta\times \partial D_+ \times\Theta_3$. Otherwise their values are equal to zero.

  \subsection{Examples under Assumption \ref{asmp:A4} \ref{asmp:A4(i)}}

    \begin{example}[Normal distribution]\label{eg:normaldistr}
      Let $\Theta_3$ be a smooth open convex set which is compactly contained in $\mathbb{R}\times\mathbb{R}_+\times\mathbb{R}^{d_3-2}$,
      and let $f_\alpha$ be of the form
      \begin{equation*}
        f_\alpha(z) = \frac{1}{\sqrt{2\pi \alpha_2^2}} \exp \left( - \frac{|z-\alpha_1|^2}{2\alpha_2^2} \right)
        \quad \text{for } \alpha=(\alpha_1,\alpha_2)\in\Theta_3.
      \end{equation*}
      Then,
      \begin{equation*}
        \psi(x,y,\alpha)
        = - \log c(x,\alpha) - \frac{1}{2} \log (2\pi\alpha_2^2) - \frac{|\frac{y}{c(x,\alpha)}-\alpha_1|^2}{2\alpha_2^2}
        \quad \text{on } I_{x_0}^\delta\times\mathbb{R}\times\Theta_3.
      \end{equation*}
      Since
      \begin{equation*}
        f_\alpha'(z) = - \frac{z-\alpha_1}{\alpha_2^2} f_\alpha(z), \quad
        \frac{\partial f_{\alpha}}{\partial \alpha_1} (z)
        = \frac{z - \alpha_1}{\alpha_2} f_{\alpha}(z)
        \quad\text{and}\quad
        \frac{\partial f_{\alpha}}{\partial \alpha_2} (z)
        = \left\{ - \frac{1}{\alpha_2} + \frac{(z-\alpha_1)^2}{\alpha_2^3} \right\} f_\alpha(z),
      \end{equation*}
      we have
      \begin{align*}
        \frac{\partial \psi}{\partial y}(x,y,\alpha)
        &= - \frac{1}{c(x,\alpha)} \frac{1}{\alpha_2^2} \left( \frac{y}{c(x,\alpha)} - \alpha_1 \right), \\
        \frac{\partial \psi}{\partial \alpha_1}(x,y,\alpha)
        &= - \frac{\partial(\log c)}{\partial \alpha_1}(x,\alpha)
              \left( 1 + y \, \frac{\partial \psi}{\partial y}(x,y,\alpha) \right)
              - \frac{\frac{y}{c(x,\alpha)} - \alpha_1}{\alpha_2 c(x,\alpha)}, \\
        \frac{\partial \psi}{\partial \alpha_2}(x,y,\alpha)
        &= - \frac{\partial(\log c)}{\partial \alpha_2}(x,\alpha)
              \left( 1 + y \, \frac{\partial \psi}{\partial y}(x,y,\alpha) \right)
              + \frac{1}{c(x,\alpha)} \left\{ - \frac{1}{\alpha_2} + \frac{|\frac{y}{c(x,\alpha)}-\alpha_1|^2}{\alpha_2^3} \right\}, \\
        \frac{\partial \psi}{\partial \alpha_2}(x,y,\alpha)
        &= - \frac{\partial(\log c)}{\partial \alpha_j}(x,\alpha)
              \left( 1 + y \, \frac{\partial \psi}{\partial y}(x,y,\alpha) \right)
      \end{align*}
      for $(x,y,\alpha)\in I_{x_0}^\delta\times\mathbb{R}_+\times\Theta_3$ and $j=3,\dots,d_3$.
      Furthermore, the derivatives of $c$ and $\log c$ with respect to $\alpha$ are bounded on $I_{x_0}^\delta\times\Theta_3$,
      and so for $(x,y,\alpha)\in I_{x_0}^\delta\times \mathbb{R}\times\Theta_3$
      \begin{equation*}
        \left| \frac{\partial^2 \psi}{\partial \alpha_j\partial y}(x,y,\alpha) \right|
        \leq C (1 + |y|), ~
        \left| \frac{\partial^2 \psi}{\partial \alpha_i\partial \alpha_j}(x,y,\alpha) \right|
        \leq C (1 + |y|^2),
      \end{equation*}
      where $C$ is a constant not depending on $(x,y,\alpha)$.
      Thus, Assumptions \ref{asmp:A9} to \ref{asmp:A12} are satisfied.
    \end{example}

\subsection{Examples under Assumption \ref{asmp:A4} \ref{asmp:A4(ii)}}

  \begin{example}[Gamma distribution]\label{eg:gammadistr}
    Let $\Theta_3$ be an open interval compactly contained in $\mathbb{R}_+\times(1,\infty)\times\mathbb{R}^{d_3-2}$,
    and let $f_\alpha$ be of the form
    \begin{equation*}
      f_\alpha(z)
        = \left\{ \begin{aligned}
            & \frac{1}{\Gamma(\alpha_2)\alpha_1^{\alpha_2}} z^{\alpha_2-1} e^{-z/{\alpha_1}}
                & & (z>0), \\
            & 0 & & (z\leq0)
          \end{aligned} \right.
    \end{equation*}
    for $\alpha\in\Theta_3$.
    Then,
    \begin{equation*}
      \psi(x,y,\alpha) = - \log c(x,\alpha) - \log \Gamma(\alpha_2) - \alpha_2 \log \alpha_1
        + (\alpha_2-1) \log z - \frac{z}{\alpha_1}
      \quad \text{on } I_{x_0}^\delta\times\mathbb{R}_+\times\Theta_3.
    \end{equation*}
    Since
    \begin{equation*}
      f_\alpha' (z)
        = \left( \frac{\alpha_2-1}{z} - \frac{1}{\alpha_1} \right) f_\alpha(z), \quad
      \frac{\partial f_{\alpha}}{\partial \alpha_1} (z)
        = \left( -\frac{\alpha_2}{\alpha_1} + \frac{z}{\alpha_1^2} \right) f_\alpha(z)
    \end{equation*}
    and
    \begin{equation*}
      \frac{\partial f_{\alpha}}{\partial \alpha_2} (z)
        = \left\{ - \frac{\Gamma'(\alpha_2)}{\Gamma(\alpha_2)} - \log \alpha_1 + \log z \right\} f_\alpha(z)
    \end{equation*}
    for $z>0$ and $\alpha\in\Theta_3$,
    we have
    \begin{align*}
      \frac{\partial \psi}{\partial y}(x,y,\alpha)
        &= \frac{\alpha_2-1}{y} - \frac{1}{\alpha_1 c(x,\alpha)}, \\
      \frac{\partial \psi}{\partial \alpha_1}(x,y,\alpha)
        &=  - \frac{\partial(\log c)}{\partial \alpha_1}(x,\alpha)
              \left( 1 + y \, \frac{\partial \psi}{\partial y}(x,y,\alpha) \right)
              + \frac{1}{c(x,\alpha)} \left\{ -\frac{\alpha_2}{\alpha_1} + \frac{y}{\alpha_1^2 c(x,\alpha)} \right\}, \\
      \frac{\partial \psi}{\partial \alpha_2}(x,y,\alpha)
        &= - \frac{\partial(\log c)}{\partial \alpha_2}(x,\alpha)
              \left( 1 + y \, \frac{\partial \psi}{\partial y}(x,y,\alpha) \right)
              + \frac{1}{c(x,\alpha)} \left\{ - \frac{\Gamma'(\alpha_2)}{\Gamma(\alpha_2)} - \log \alpha_1 + \log \frac{y}{c(x,\alpha)} \right\}, \\
      \frac{\partial \psi}{\partial \alpha_j}(x,y,\alpha)
        &=  - \frac{\partial(\log c)}{\partial \alpha_j}(x,\alpha)
              \left( 1 + y \, \frac{\partial \psi}{\partial y}(x,y,\alpha) \right)
    \end{align*}
    for $(x,y,\alpha)\in I_{x_0}^\delta\times\mathbb{R}_+\times\Theta_3$ and $j=3,\dots,d_3$.
    Furthermore, the derivatives of $c$ and $\log c$ with respect to $\alpha$ are bounded on $I_{x_0}^\delta\times\Theta_3$,
    and so for $(x,y,\alpha)\in I_{x_0}^\delta\times \mathbb{R}\times\Theta_3$
    \begin{equation*}
      \left| \frac{\partial^2 \psi}{\partial \alpha_j\partial y}(x,y,\alpha) \right|
      \leq C, \quad
      \left| \frac{\partial^2 \psi}{\partial \alpha_i\partial \alpha_j}(x,y,\alpha) \right|
      \leq C (1 + |y|)
    \end{equation*}
    where $C$ is a constant not depending on $(x,y,\alpha)$.
    Thus, Assumptions \ref{asmp:A9} to \ref{asmp:A12} are satisfied,
    and $\rho$ in Theorem \ref{thm:3.1} and \ref{thm:3.2} can be taken as $\rho\in(0,1/4)$.
    Here, we remark that if $\alpha_2>1$, then
    \begin{equation*}
      \int \frac{1}{z} f_{\alpha}(z) \, dz < \infty \quad \text{if and only if} \quad \alpha_2>1.
    \end{equation*}
  \end{example}

  \begin{example}[Inverse Gaussian distribution]
    Let $\Theta_3$ be smooth, open, convex and compactly contained in $\mathbb{R}_+^2\times\mathbb{R}^{d_3-2}$,
    and let $f_\alpha$ be of the form
    \begin{equation*}
      f_\alpha(z)
        = \left\{ \begin{aligned}
            & \sqrt\frac{\alpha_2}{2\pi z^3} e^{-\alpha_2(z - \alpha_1)^2/2\alpha_1^2z}
                & & (z>0), \\
            & 0 & & (z\leq0)
          \end{aligned} \right.
    \end{equation*}
    for $\alpha\in\Theta_3$.
    Then,
    \begin{equation*}
      \psi(x,y,\alpha)
      = \frac{1}{2(x,\alpha)} \left\{ \log \frac{\alpha_2}{2\pi} - 3 \log \frac{y}{c(x,\alpha)} \right\}
      - \frac{\alpha_2 \left| \frac{y}{c(x,\alpha)} - \alpha_1 \right|^2}{2 \alpha_1^2 y}
      \quad \text{on } I_{x_0}^\delta\times\mathbb{R}_+\times\Theta_3.
    \end{equation*}
    Since
    \begin{gather*}
      f_\alpha' (z)
        = \left\{ - \frac{3}{2z} - \frac{\alpha_2(z-\alpha_1)}{\alpha_1^2z} - \frac{\alpha_2(z-\alpha_1)^2}{2\alpha_1z^2} \right\} f_\alpha(z), \\
      \frac{\partial f_{\alpha}}{\partial \alpha_1} (z)
        = \frac{\alpha_2(z-\alpha_1)}{\alpha_1^2} f_\alpha(z) \quad\text{and}\quad
      \frac{\partial f_{\alpha}}{\partial \alpha_2} (z)
        = \left\{ \frac{1}{2\alpha_2} - \frac{|z-\alpha_1|^2}{2\alpha_1^2z} \right\} f_\alpha(z)
    \end{gather*}
    for $z>0$ and $\alpha\in\Theta_3$,
    we have
    \begin{align*}
      \frac{\partial \psi}{\partial y} (x,y,\alpha)
        &= - \frac{3}{2y} - \frac{\alpha_2(\frac{y}{c(x,\alpha)}-\alpha_1)}{\alpha_1^2y} - \frac{\alpha_2|\frac{y}{c(x,\alpha)}-\alpha_1|^2}{2\alpha_1\frac{y^2}{c(x,\alpha)}}, \\
      \frac{\partial \psi}{\partial \alpha_1} (x, y, \alpha)
        & = - \frac{\partial(\log c)}{\partial \alpha_1}(x,\alpha)
            \left( 1 + y \, \frac{\partial \psi}{\partial y}(x,y,\alpha) \right)
            + \frac{\alpha_2(\frac{y}{c(x,\alpha)}-\alpha_1)}{\alpha_1^2 c(x,\alpha)}, \\
      \frac{\partial \psi}{\partial \alpha_2}(x,y,\alpha)
        &= - \frac{\partial(\log c)}{\partial \alpha_2}(x,\alpha)
              \left( 1 + y \, \frac{\partial \psi}{\partial y}(x,y,\alpha) \right)
              + \frac{1}{2\alpha_2 c(x,\alpha)} - \frac{|\frac{y}{c(x,\alpha)}-\alpha_1|^2}{2\alpha_1^2y}, \\
      \frac{\partial \psi}{\partial \alpha_j}(x,y,\alpha)
        &=  - \frac{\partial(\log c)}{\partial \alpha_j}(x,\alpha)
            \left( 1 + y \, \frac{\partial \psi}{\partial y}(x,y,\alpha) \right)
    \end{align*}
    for $(x,y,\alpha)\in I_{x_0}^\delta\times\mathbb{R}_+\times\Theta_3$ and $j=3,\dots,d_3$.
    Furthermore, the derivatives of $c$ and $\log c$ with respect to $\alpha$ are bounded on $I_{x_0}^\delta\times\Theta_3$,
    and so
    \begin{equation*}
      \sup_{(x,\alpha) \in I_{x_0}^\delta\times\Theta_3} \left| \frac{\partial^2 \psi}{\partial \alpha_j\partial y}(x,y,\alpha) \right|
      \leq O \left( \frac{1}{|y|^2} \right) \quad \text{as} ~ y\to0,
    \end{equation*}
    for $(x,y,\alpha)\in I_{x_0}^\delta\times \mathbb{R}_+\times\Theta_3$ with $y/c(x,\alpha)\neq\alpha_1$.
    Thus, Assumptions \ref{asmp:A9} to \ref{asmp:A12} are satisfied,
    and $\rho$ in Theorems \ref{thm:3.1} and \ref{thm:3.2} can be taken as $\rho\in(0,1/8)$.
  \end{example}

  \begin{example}[Weibull distribution]
    Let $\Theta_3$ be smooth, open, convex and compactly contained in $\mathbb{R}_+\times(1,\infty)\times\mathbb{R}^{d_3-2}$,
    and let $f_\alpha$ be of the form
    \begin{equation*}
      f_\alpha(z)
        = \left\{ \begin{aligned}
            &\frac{\alpha_2}{\alpha_1} \left( \frac{z}{\alpha_1} \right)^{\alpha_2-1} e^{-(z/\alpha_1)^{\alpha_2}}
                & & (z>0), \\
            & 0 & & (z\leq0)
          \end{aligned} \right.
    \end{equation*}
    for $\alpha\in\Theta_3$.
    Then,
    \begin{equation*}
      \psi(x,y,\alpha)
      = \frac{1}{c(x,\alpha)} \left\{ \log \alpha_2 - \alpha_2 \log \alpha_1 - (\alpha_2-1) \log \frac{y}{c(x,\alpha)} \right\}
      \quad \text{on } I_{x_0}^\delta\times\mathbb{R}_+\times\Theta_3.
    \end{equation*}
    Since
    \begin{equation*}
      f_\alpha' (z)
        = \left( \frac{\alpha_2-1}{z} - \alpha_2 \left( \frac{z}{\alpha_1} \right)^{\alpha_2-1} \right) f_\alpha(z), \quad
      \frac{\partial f_\alpha}{\partial \alpha_1}(z)
        = - \frac{\alpha_2}{\alpha_1} \left\{ 1 + \left( \frac{z}{\alpha_1} \right)^{\alpha_2} \right\} f_\alpha(z)
    \end{equation*}
    and
    \begin{equation*}
      \frac{\partial f_\alpha}{\partial \alpha_2}(z)
        = \left\{ \frac{1}{\alpha_2} + \log \frac{z}{\alpha_1} - \left( \frac{z}{\alpha_1} \right)^{\alpha_2} \log \frac{z}{\alpha_1} \right\} f_\alpha(z)
    \end{equation*}
    for $z>0$ and $\alpha\in\Theta_3$,
    we have
    \begin{align*}
      \frac{\partial \psi}{\partial y}(x,y,\alpha)
        &= \frac{(\alpha_2-1)}{y} - \frac{\alpha_2}{c(x,\alpha)} \left( \frac{y}{\alpha_1 c(x,\alpha)} \right)^{\alpha_2-1} \\
      \frac{\partial \psi}{\partial \alpha_1} (x, y, \alpha)
        & = - \frac{\partial(\log c)}{\partial \alpha_1}(x,\alpha)
            \left( 1 + y \, \frac{\partial \psi}{\partial y}(x,y,\alpha) \right)
            - \frac{\alpha_2}{\alpha_1 c(x,\alpha)} \left\{ 1 + \left( \frac{y}{\alpha_1 c(x,\alpha)} \right)^{\alpha_2} \right\}, \\
      \frac{\partial \psi}{\partial \alpha_2}(x,y,\alpha)
        &= - \frac{\partial(\log c)}{\partial \alpha_2}(x,\alpha)
              \left( 1 + y \, \frac{\partial \psi}{\partial y}(x,y,\alpha) \right) \\
              &\quad + \frac{1}{c(x,\alpha)} \left\{ \frac{1}{\alpha_2} + \log \frac{y}{\alpha_1c(x,\alpha)} - \left( \frac{y}{\alpha_1 c(x,\alpha)} \right)^{\alpha_2} \log \frac{y}{\alpha_1 c(x,\alpha)} \right\}, \\
      \frac{\partial \psi}{\partial \alpha_j}(x,y,\alpha)
        &=  - \frac{\partial(\log c)}{\partial \alpha_j}(x,\alpha)
            \left( 1 + y \, \frac{\partial \psi}{\partial y}(x,y,\alpha) \right)
    \end{align*}
    for $(x,y,\alpha)\in I_{x_0}^\delta\times\mathbb{R}_+\times\Theta_3$ and $j=3,\dots,d_3$.
    Furthermore, the derivatives of $c$ and $\log c$ with respect to $\alpha$ are bounded on $I_{x_0}^\delta\times\Theta_3$,
    and so
    \begin{equation*}
      \sup_{(x,\alpha) \in I_{x_0}^\delta\times\Theta_3} \left| \frac{\partial^2 \psi}{\partial \alpha_j\partial y}(x,y,\alpha) \right|
      \leq O \left( \frac{1}{y} \right) \quad \text{as} ~ y\to0,
    \end{equation*}
    for $(x,y,\alpha)\in I_{x_0}^\delta\times \mathbb{R}\times\Theta_3$ with $y/c(x,\alpha)\neq\alpha_1$,
    where $C$ is a constant not depending on $(x,y,\alpha)$.
    Here, we remark that
    \begin{equation*}
      \int \frac{1}{y} f_\alpha(y) \, dy < \infty \quad \text{if and only if} \quad \alpha_2 > 1
    \end{equation*}
    and that there exists a constant $C>0$ such that
    \begin{equation*}
      | y^{\alpha_2-1} \log y | \leq |y_1^{\alpha_2-1} \log y_1| + | y_2^{\alpha_2-1} \log y_2| + C \quad
      \text{for } y_1 \leq y \leq y_2.
    \end{equation*}
    Thus, Assumptions \ref{asmp:A9} to \ref{asmp:A12} are satisfied,
    and $\rho$ in Theorems \ref{thm:3.1} and \ref{thm:3.2} can be taken as $\rho\in(0,1/4)$.
  \end{example}

  \begin{example}[Log-normal distribution]
    Let $\Theta_3$ be smooth, open, convex and compactly contained in $\mathbb{R}\times[0,\infty)$,
    and let $f_\alpha$ be of the form
    \begin{equation*}
      f_\alpha(z)
        = \left\{ \begin{aligned}
            &\frac{1}{\sqrt{2\pi}\alpha_2 z} e^{-(\log z - \alpha_1)^2/2\alpha_2^2}
                & & (z>0), \\
            & 0 & & (z\leq0)
          \end{aligned} \right.
    \end{equation*}
    for $\alpha\in\Theta_3$.
    Then,
    \begin{equation*}
      \psi(x,y,\alpha) = \frac{1}{c(x,\alpha)} \left\{ - \log \frac{\sqrt{2\pi}\alpha_2 y}{c(x,\alpha)}- \frac{1}{2\alpha_2} \left| \log \frac{y}{c(x,\alpha)} - \alpha_1 \right|^2 \right\}
      \quad \text{on } I_{x_0}^\delta\times\mathbb{R}_+\times\Theta_3.
    \end{equation*}
    Since
    \begin{gather*}
      f_\alpha' (z)
        = \left\{ - \frac{1}{z} - \frac{\log z - \alpha_1}{\alpha_2^2 z} \right\} f_\alpha(z), \\
      \frac{\partial f_\alpha}{\partial \alpha_1}(z)
        = \frac{\log z - \alpha_1}{\alpha_2^2} f_\alpha(z), \quad
      \frac{\partial f_\alpha}{\partial \alpha_2}(z)
        = \left\{ - \frac{1}{\alpha_2} + \frac{|\log z - \alpha_1|^2}{\alpha_2^3} \right\} f_\alpha(z)
    \end{gather*}
    for $z>0$ and $\alpha\in\Theta_3$,
    we have
    \begin{align*}
      \frac{\partial \psi}{\partial y}(x,y,\alpha)
        &= - \frac{1}{\alpha_2^2 y} \left( \alpha_1 + \alpha_2^2 + \log \frac{y}{c(x,\alpha)} \right) \\
      \frac{\partial \psi}{\partial \alpha_1} (x, y, \alpha)
        & = - \frac{\partial(\log c)}{\partial \alpha_1}(x,\alpha)
            \left( 1 + y \, \frac{\partial \psi}{\partial y}(x,y,\alpha) \right)
            + \frac{\log \frac{y}{c(x,\alpha)} - \alpha_1}{\alpha_2^2 c(x,\alpha)}, \\
      \frac{\partial \psi}{\partial \alpha_2}(x,y,\alpha)
        &= - \frac{\partial(\log c)}{\partial \alpha_2}(x,\alpha)
              \left( 1 + y \, \frac{\partial \psi}{\partial y}(x,y,\alpha) \right)
              + \frac{1}{c(x,\alpha)} \left\{ - \frac{1}{\alpha_2} + \frac{|\log \frac{y}{c(x,\alpha)} - \alpha_1|^2}{\alpha_2^3} \right\}, \\
      \frac{\partial \psi}{\partial \alpha_j}(x,y,\alpha)
        &=  - \frac{\partial(\log c)}{\partial \alpha_j}(x,\alpha)
            \left( 1 + y \, \frac{\partial \psi}{\partial y}(x,y,\alpha) \right)
    \end{align*}
    for $(x,y,\alpha)\in I_{x_0}^\delta\times\mathbb{R}_+\times\Theta_3$ and $j=3,\dots,d_3$.
    Furthermore, the derivatives of $c$ and $\log c$ with respect to $\alpha$ are bounded on $I_{x_0}^\delta\times\Theta_3$,
    and so
    \begin{equation*}
      \sup_{(x,\alpha) \in I_{x_0}^\delta\times\Theta_3} \left| \frac{\partial^2 \psi}{\partial \alpha_j\partial y}(x,y,\alpha) \right|
      \leq O \left( \frac{1}{y} + \frac{1}{y} \log y \right) \quad \text{as} ~ y\to0,
    \end{equation*}
    for $(x,y,\alpha)\in I_{x_0}^\delta\times \mathbb{R}\times\Theta_3$ with $y/c(x,\alpha)\neq\alpha_1$,
    where $C$ is a constant not depending on $(x,y,\alpha)$.
    Here, we remark that
    \begin{equation*}
      \int \left( \frac{1}{y} + \frac{\log y}{y} \right) f_\alpha(y) \, dy < \infty
    \end{equation*}
    and that there exists a constant $C>0$ such that
    \begin{equation*}
      | \frac{1}{y} \log y | \leq | \frac{1}{y_1} \log y_1| + | \frac{1}{y_2} \log y_2| + C \quad
      \text{for } y_1 \leq y \leq y_2.
    \end{equation*}
    Thus, Assumptions \ref{asmp:A9} to \ref{asmp:A12} are satisfied,
    and $\rho$ in Theorems \ref{thm:3.1} and \ref{thm:3.2} can be taken as $\rho\in(0,1/4)$.
  \end{example}

\begin{appendix}
\section{Appendix}\label{appn}
    In this section, we state and prove some slightly different versions of well-known results.
    More precisely, we prepare Lemma \ref{lem:A.2} as localization of the continuous mapping theorem.
    Lemma \ref{lem:A.3} is a slightly different version of Lemma 9 in Genon-Catalot and Jacod \cite{genon-catalot1993ontheest}.

    \begin{lemma}\label{lem:A.1}
      Let $\mathcal{X}$ be a Banach space,
      and let $\{g_\theta\}_{\theta\in\Theta}$ be a family of functions from $\mathcal{X}$ to $\mathbb{R}$,
      and let $T_{g_\theta}$ be the composition operator on $L^\infty([0,1];\mathcal{X})$  generated by $g_\theta$, i.e.,
      \begin{equation*}
        T_{g_\theta}(\tilde{y}_\cdot):=g_\theta(\tilde{y}_\cdot) \quad \text{for}~\tilde{y}_\cdot\in L^\infty([0,1];\mathcal{X}).
      \end{equation*}
      Suppose that $y_\cdot$ is a version of a function of $C([0,1];\mathcal{X})$ in $L^\infty([0,1];\mathcal{X})$, and that $\{g_\theta\}_{\theta\in\Theta}$ is equicontinuous at every points in $\Image(y_\cdot):=\{y_t\,|\,t\in[0,1]\}$.
      Then, there is a neighborhood $\mathcal{N}_{y_\cdot}$ of $y_\cdot$ in $L^\infty([0,1];\mathcal{X})$ such that $\{T_{g_\theta}\}_{\theta\in\Theta}$ is a family of operators from $\mathcal{N}_{y_\cdot}$ to $L^\infty([0,1])$,
      and is equicontinuous at $y_\cdot$.
    \end{lemma}

    \begin{proof}
      Fix an arbirary $\eta>0$.
      For each $x\in\Image(y_\cdot)$, there exists $\delta_x>0$ such that if $\left\| x-x' \right\|_{\mathcal{X}}<\delta_x$, $x,x'\in\mathcal{X}$, then
      \begin{equation*}
        \sup_{\theta\in\Theta} \left| g_\theta(x) - g_\theta(x') \right| < \frac{\eta}{2}.
      \end{equation*}
      Since $\Image(y_\cdot)$ is compact in $\mathcal{X}$,
      there are finite points $x_1,\dots,x_k\in\Image(y_\cdot)$ such that
      \begin{equation*}
        \Image(y_\cdot) \subset \bigcup_{i=1}^k B(x_i,\delta_{x_i}/2),
      \end{equation*}
      where $B(x_i,\delta_{x_i}/2)$ is the ball in $\mathcal{X}$ centered at $x_i$ with radius $\delta_{x_i}/2$.
      If $\left\| \tilde{y}_\cdot-y_\cdot \right\|_{L^\infty([0,1];\mathcal{X})}<\delta$
      with $\delta:=\min\{\delta_{x_1}/2,\dots,\delta_{x_k}/2\}$,
      then for a.e. $t\in[0,1]$ there is $i_t\in\{1,\dots,k\}$
      such that $y_t,\tilde{y}_t\in B(x_{i_t},\delta_{x_{i_t}})$.
      Thus, we obtain
      \begin{equation*}
        \sup_{\theta\in\Theta} \left| g_\theta(\tilde{y}_t) - g_\theta(y_t) \right|
        \leq \sup_{\theta\in\Theta} \left| g_\theta(\tilde{y}_t) - g_\theta(x_{i_t}) \right|
        +  \sup_{\theta\in\Theta} \left| g_\theta(x_{i_t}) - g_\theta(y_t) \right| < \eta,
      \end{equation*}
      that is,
      \begin{equation*}
        \sup_{\theta\in\Theta} \left\| g_\theta(\tilde{y}_\cdot) - g_\theta(y_\cdot) \right\|_{L^\infty([0,1])} < \eta.
      \end{equation*}
      This implies the conclusion.
    \end{proof}

    We prepare the following lemma as localization of the continuous mapping theorem.

    \begin{lemma}\label{lem:A.2}
      Under the same assumptions as in Lemma \ref{lem:A.1},
      suppose that $\{g(\cdot,\theta)\}_{\theta\in\Theta}$ is equicontinuous at every points in $\Image(y_\cdot):=\{y_t\,|\,t\in[0,1]\}$,
      and that $(Y^\iota_\cdot)_{\iota\in I}$ is a net of $\mathcal{X}$-valued bounded random processes on $[0,1]$ with a directed set $I$.
      If the net $(Y^\iota_\cdot)_{\iota\in I}$ converges in probability to $y_\cdot$ in $L^\infty([0,1;\mathcal{X}])$, \textit{i.e.},
      \begin{equation*}
        \left\| Y^\iota_\cdot - y_\cdot \right\|_{L^\infty([0,1];\mathcal{X})}
        \overset{p}\longrightarrow 0,
      \end{equation*}
      then
      \begin{equation*}
        \sup_{\theta\in\Theta} \left\| g(Y^\iota_\cdot,\theta) - g(y_\cdot,\theta) \right\|_{L^\infty([0,1])}
        \overset{p}\longrightarrow 0.
      \end{equation*}
    \end{lemma}

    \begin{proof}
      Take an arbitrary $\eta>0$.
      It follows from Lemma \ref{lem:A.1} that there exists a sufficiently small $\delta>0$ such that if $\left\| \tilde{y}_\cdot-y_\cdot \right\|_{L^\infty([0,1];\mathcal{X})}<\delta$,
      then $\{g(\tilde{y}_\cdot,\theta)\}_{\theta\in\Theta}\subset L^\infty([0,1])$ and
      \begin{equation*}
        \sup_{\theta\in\Theta} \left\| g(\tilde{y}_\cdot,\theta) - g(y_\cdot,\theta) \right\|_{L^\infty([0,1])} < \eta,
      \end{equation*}
      and therefore,
      \begin{equation*}
        P \left ( \sup_{\theta\in\Theta} \left \| g(Y^\iota_\cdot,\theta)
                                               - g(y_\cdot,\theta) \right \|_{L^\infty([0,1])} > \eta \right )
        \leq P \left ( \left \| Y^\iota_\cdot - y_\cdot \right \|_{L^\infty([0,1];\mathcal{X})} > \delta \right ).
      \end{equation*}
      This implies the conclusion.
    \end{proof}

    \begin{remark}\label{rmk: continuous mapping}
      By the proof of Lemma \ref{lem:A.2}, it also follows that for any $C_1>0$,
      \begin{equation*}
        P \left ( \sup_{\theta\in\Theta} \left \| g(Y^\iota_\cdot,\theta)
                                               - g(y_\cdot,\theta) \right \|_{L^\infty([0,1])} > C_2 \right )
        \leq P \left ( \left \| Y^\iota_\cdot - y_\cdot \right \|_{L^\infty([0,1];\mathcal{X})} > C_1 \right ),
      \end{equation*}
      where $C_2$ depends only on $C_1$, $g$ and $\Image(y_\cdot)$.
    \end{remark}

    \begin{lemma}\label{lem:A.3}
      Suppose that $(\mathcal{X},\|\cdot\|)$ is a Banach space, $\{(n,\varepsilon)\}$ is a directed set and $\{\mathcal{G}^{n,\varepsilon}_i\}_{i}$ is a filtration for each $n,\varepsilon$. Let $\chi^{n,\varepsilon}_i$, $U$ be $\mathcal{X}$-valued $\mathcal{G}^{n,\varepsilon}_i$-measurable random variables.
      \begin{enumerate}[label=(\roman*)]
        \item \label{item:ConseqFromLenglart(i)}
              If for any $\eta>0$
              \begin{equation*}
                \lim_{n,\varepsilon} P \left( \sum_{i=1}^n E \left[ \| \chi^{n,\varepsilon}_i \| \, \middle| \, \mathcal{G}^{n,\varepsilon}_{i-1} \right] > \eta \right)
                =0,
              \end{equation*}
              then for any $\eta>0$
              \begin{equation*}
                \lim_{n,\varepsilon} P \left( \left\| \sum_{i=1}^n \chi^{n,\varepsilon}_i \right\| > \eta \right) =  0.
              \end{equation*}
        \item \label{item:ConseqFromLenglart(ii)}
              If
              \begin{equation*}
                \lim_{M\to\infty} \sup_{n,\varepsilon} P \left( \sum_{i=1}^n E \left[ \| \chi^{n,\varepsilon}_i \| \, \middle| \, \mathcal{G}^{n,\varepsilon}_{i-1} \right] > M \right)
                = 0,
              \end{equation*}
              then
              \begin{equation*}
                \lim_{M\to\infty} \sup_{n,\varepsilon} P \left( \left\| \sum_{i=1}^n \chi^{n,\varepsilon}_i \right\| > M \right) = 0.
              \end{equation*}
      \end{enumerate}
    \end{lemma}

    \begin{proof}
      Since for any $\eta,\eta'>0$
      \begin{align*}
        &P \left( \left\| \sum_{i=1}^n \chi^{n,\varepsilon}_i \right\| > \eta, ~ \sum_{i=1}^n E \left[ \| \chi^{n,\varepsilon}_i \| \, \middle| \, \mathcal{G}^{n,\varepsilon}_{i-1} \right] \leq \eta' \right) \\
        &\quad
          \leq \frac{1}{\eta} E \left[ \sum_{i=1}^n \|\chi^{n,\varepsilon}_i\| ~ 1_{\left\{ \sum_{i=1}^n E \left[ \| \chi^{n,\varepsilon}_i \| \, \middle| \, \mathcal{G}^{n,\varepsilon}_{i-1} \right] \leq \eta' \right\}} \right] \\
        &\quad
          \leq \frac{1}{\eta} E \left[ \left( E \left[ \|\chi^{n,\varepsilon}_n\| \, \middle| \, \mathcal{G}^{n,\varepsilon}_{n-1} \right] + \sum_{i=1}^{n-1} \| \chi^{n,\varepsilon}_i \| \right) ~ 1_{\left\{ \sum_{i=1}^n E \left[ \| \chi^{n,\varepsilon}_i \| \, \middle| \, \mathcal{G}^{n,\varepsilon}_{i-1} \right] \leq \eta' \right\}} \right] \\
        &\quad
          \leq \frac{1}{\eta} E \left[ \eta' + \sum_{i=1}^{n-1} \left( \| \chi^{n,\varepsilon}_i \| - E \left[ \| \chi^{n,\varepsilon}_i \| \, \middle| \, \mathcal{G}^{n,\varepsilon}_{i-1} \right] \right) 1_{\left\{ E \left[ \| \chi^{n,\varepsilon}_i \| \, \middle| \, \mathcal{G}^{n,\varepsilon}_{i-1} \right] \leq \eta' \right\}} \right]
        < \frac{\eta'}{\eta},
      \end{align*}
      we obtain
      \begin{equation*}
        P \left( \left\| \sum_{i=1}^n \chi^{n,\varepsilon}_i \right\| > \eta \right)
        \leq \frac{\eta'}{\eta}
        + P \left( \sum_{i=1}^n E \left[ \| \chi^{n,\varepsilon}_i \| \, \middle| \, \mathcal{G}^{n,\varepsilon}_{i-1} \right] > \eta' \right).
      \end{equation*}
      Thus, the assertions \ref{item:ConseqFromLenglart(i)} and \ref{item:ConseqFromLenglart(ii)} follows.
    \end{proof}

    \begin{remark}
      When $\mathcal{X}=\mathbb{R}$,
      this lemma can be shown by the same argument in the proof of Lemma 9 in Genon-Catalot and Jacod \cite{genon-catalot1993ontheest}.
      However, the argument does not work in general,
      since we may not have Lenglart's inequality
      (e.g., Lemma 3.30 in Jacod and Shiryaev \cite{jacod2003limit})
      when $\mathcal{X}$ is a Banach space.
    \end{remark}

    \begin{remark}\label{rmk:ConseqFromLenglart2}
      We have an immediate consequence from this lemma that
      \begin{align*}
        \sum_{i=1}^n E \left[ \| \chi^{n,\varepsilon}_i \| \, \middle| \, \mathcal{G}^{n,\varepsilon}_{i-1} \right]
        =o_p(r_{n,\varepsilon})
        \quad &\Longrightarrow \quad
        \sum_{i=1}^n \chi^{n,\varepsilon}_i = o_p(r_{n,\varepsilon}), \\
        \sum_{i=1}^n E \left[ \| \chi^{n,\varepsilon}_i \| \, \middle| \, \mathcal{G}^{n,\varepsilon}_{i-1} \right]
        =O_p(r_{n,\varepsilon})
        \quad &\Longrightarrow \quad
        \sum_{i=1}^n \chi^{n,\varepsilon}_i = O_p(r_{n,\varepsilon}),
      \end{align*}
      where $r_{n,\varepsilon}\in\mathbb{R}$.
    \end{remark}

\end{appendix}

\noindent
\textbf{Acknowledgments.}
% The authors would like to thank the anonymous referees, an Associate
% Editor and the Editor for their constructive comments that improved the
% quality of this paper.
%
% \noindent
% \textbf{Funding.}
The second author was supported by JSPS KAKENHI Grant-in-Aid for Scientific Research (C) \#21K03358 and JST CREST \#PMJCR14D7, Japan.

% \pdfbookmark[1]{References}{bib}
\bibliographystyle{amsplain}
  % \bibliography{C:/texlive/texmf-local/bibtex/bib/local/bibtex.bib}
  \bibliography{KS-MLE_arxiv.bbl}

\providecommand{\bysame}{\leavevmode\hbox to3em{\hrulefill}\thinspace}
\providecommand{\MR}{\relax\ifhmode\unskip\space\fi MR }
% \MRhref is called by the amsart/book/proc definition of \MR.
\providecommand{\MRhref}[2]{%
  \href{http://www.ams.org/mathscinet-getitem?mr=#1}{#2}
}
\providecommand{\href}[2]{#2}
\begin{thebibliography}{10}

\bibitem{amorino2021joint}
Chiara Amorino and Arnaud Gloter, \emph{Joint estimation for volatility and
  drift parameters of ergodic jump diffusion processes via contrast function.
  ({E}nglish summary)}, Statistical Inference for Stochastic Processes
  \textbf{24} (2021), no.~1, 61–148.

\bibitem{applebaum2009levy}
David Applebaum, \emph{L\'evy processes and stochastic calculus}, second ed.,
  Cambridge Studies in Advanced Mathematics, vol. 116, Cambridge University
  Press, Cambridge, 2009.

\bibitem{billingsley1999convergence}
Patrick Billingsley, \emph{Convergence of probability measures}, second ed.,
  Wiley Series in Probability and Statistics: Probability and Statistics, A
  Wiley-Interscience Publication. John Wiley \& Sons, Inc., New York, 1999.

\bibitem{brezis2010functional}
Haim Brezis, \emph{Functional analysis, sobolev spaces and partial differential
  equations}, Universitext, Springer, New York, 2011.

\bibitem{evans2010partial}
Lawrence~C. Evans, \emph{Partial differential equations}, second ed., Graduate
  Studies in Mathematics, vol.~19, American Mathematical Society, Providence,
  RI, 2010.

\bibitem{genon-catalot1993ontheest}
Valentine Genon-Catalot and Jean Jacod, \emph{On the estimation of the
  diffusion coefficient for multi-dimensional diffusion processes}, Annales de
  l'Institut Henri Poincar\'e Probabilit\'es et Statistiques \textbf{29}
  (1993), no.~1, 119–151.

\bibitem{gloter2018jump}
Arnaud Gloter, Dasha Loukianova, and Hilmar Mai, \emph{Jump filtering and
  efficient drift estimation for {L}\'evy-driven {SDE}s}, the Annals of
  Statistics \textbf{46} (2018), no.~4, 1445--1480.

\bibitem{gloter2009estimation}
Arnaud Gloter and Michael Sørensen, \emph{Estimation for stochastic
  differential equations with a small diffusion coefficient}, Stochastic
  Processes and their Applications \textbf{119} (2009), no.~3, 679–699.

\bibitem{Ibragimov1981statistical}
Il'dar~A. Ibragimov and Rafail~Z. Has’minskii, \emph{Statistical estimation.
  asymptotic theory. translated by {S}amuel {K}otz}, Applications of
  Mathematics, vol.~16, Springer-Verlag, New York-Berlin, 1981.

\bibitem{jacod2003limit}
Jean Jacod and Albert~N. Shiryaev, \emph{Limit theorems for stochastic
  processes}, second ed., vol. 288, Springer-Verlag, Berlin, 2003.

\bibitem{kobayashi2022least}
Mitsuki Kobayashi and Yasutaka Shimizu, \emph{Least-squares estimators based on
  the {A}dams method for stochastic differential equations with small {L}\'evy
  noise}, Japanese Journal of Statistics and Data Science \textbf{5} (2022),
  217--240.

\bibitem{long2013least}
H.~{L}ong, Y.~{S}himizu, and W.~{S}un, \emph{Least squares estimators for
  discretely observed stochastic processes driven by small {L}\'{e}vy noises},
  Journal of Multivariate Analysis \textbf{116} (2013), 422--439.

\bibitem{long2017least}
Hongwei Long, Chunhua Ma, and Yasutaka Shimizu, \emph{Least squares estimators
  for stochastic differential equations driven by small lévy noises},
  Stochastic Processes and their Applications \textbf{127} (2017), no.~5,
  1475–1495.

\bibitem{masuda2021estimating}
Hiroki Masuda and Yuma Uehara, \emph{Estimating diffusion with compound poisson
  jumps based on self-normalized residuals}, Journal of Statistical Planning
  and Inference \textbf{215} (2021), 158–183.

\bibitem{ogihara2011quasi}
Teppei Ogihara and Nakahiro Yoshida, \emph{Quasi-likelihood analysis for the
  stochastic differential equation with jumps}, Statistical Inference for
  Stochastic Processes \textbf{14} (2011), no.~3, 189--229.

\bibitem{prakasarao1999semimartingale}
Bhagavatula L.~S. Prakasa~Rao, \emph{Semimartingales and their statistical
  inference}, Monographs on Statistics and Applied Probability, vol.~83,
  Chapman \& Hall/CRC, Boca Raton, FL, 1999.

\bibitem{shimizu2007asymptotic}
Yasutaka Shimizu, \emph{Asymptotic inference for stochastic differential
  equations with jumps from discrete observations and some practical
  approaches}, Ph.D. thesis, University of Tokyo, 2007.

\bibitem{shimizu2008practical}
\bysame, \emph{A practical inference for discretely observed jump-diffusions
  from finite samples}, Journal of the Japan Statistical Society \textbf{38}
  (2008), no.~3, 391--413.

\bibitem{shimizu2010thredhold}
\bysame, \emph{Threshold selection in jump-discriminant filter for discretely
  observed jump processes}, Statistical Methods \& Applications \textbf{19}
  (2010), no.~3, 355--378.

\bibitem{shimizu2017threshold}
\bysame, \emph{Threshold estimation for stochastic processes with small noise},
  Scandinavian Journal of Statistics \textbf{44} (2017), no.~4, 951–988.

\bibitem{shimizu2006estimation}
Yasutaka Shimizu and Nakahiro Yoshida, \emph{Estimation of parameters for
  diffusion processes with jumps from discrete observations}, Statistical
  Inference for Stochastic Processes \textbf{9} (2006), no.~3, 227–277.

\bibitem{sorensen2003small}
Michael S{\o}rensen and Masayuki Uchida, \emph{Small-diffusion asymptotics for
  discretely sampled stochastic differential equations}, Bernoulli \textbf{9}
  (2003), no.~6, 1051–1069.

\bibitem{yoshida1990asymptotic}
Nakahiro Yoshida, \emph{Asymptotic behavior of m-estimator and related random
  field for diffusion process}, Annals of the Institute of Statistical
  Mathematics \textbf{42} (1990), 221–251.

\end{thebibliography}

\end{document}